\documentclass[a4paper,11pt]{article}
\usepackage{amsmath,amsthm}
\usepackage{authblk}
\usepackage[style=numeric-comp,maxbibnames=99,bibencoding=utf8,firstinits=true]{biblatex}
\usepackage{booktabs}
\usepackage[makeroom]{cancel}
\usepackage{enumitem}
\usepackage[hmargin={26mm,26mm},vmargin={30mm,35mm}]{geometry}
\usepackage{loglogslopetriangle}
\usepackage{multirow}
\usepackage{nicefrac}
\usepackage[colorlinks,allcolors={blue}]{hyperref}
\usepackage{pgfplots}
\usepackage[compact,tiny]{titlesec}
\usepackage{tikz-cd}
\usepackage{subcaption}

\usepackage{pgfplots,pgfplotstable}
\usepackage[font=small]{subcaption}

\usepackage{newtxtext}
\usepackage{newtxmath}
\usetikzlibrary{patterns}

\bibliography{serendipity}
\newcommand{\email}[1]{\href{mailto:#1}{#1}}

\numberwithin{equation}{section}

\newtheorem{theorem}{Theorem}
\newtheorem{proposition}[theorem]{Proposition}
\newtheorem{lemma}[theorem]{Lemma}

\theoremstyle{remark}
\newtheorem{remark}[theorem]{Remark}
\theoremstyle{definition}
\newtheorem{assumption}[theorem]{Assumption}

\newcommand{\st}{\,:\,}
\newcommand{\Real}{\mathbb{R}}

\newcommand\eqtext[1]{\mathrel{\stackrel{\mbox{\normalfont\scriptsize #1}}{=}}}

\DeclareRobustCommand{\bvec}[1]{\boldsymbol{#1}}
\newcommand{\s}[1]{\hat{#1}}
\newcommand{\uvec}[1]{\underline{\bvec{#1}}}
\newcommand{\cvec}[1]{\bvec{\mathcal{#1}}}
\newcommand{\suvec}[1]{\underline{\bvec{\s{#1}}}}
\newcommand{\sbvec}[1]{\s{\bvec{#1}}}
\newcommand{\su}[1]{\underline{\s{#1}}}

\DeclareMathOperator{\GRAD}{\bf grad}
\DeclareMathOperator{\CURL}{\bf curl}
\DeclareMathOperator{\DIV}{div}
\DeclareMathOperator{\ROT}{rot}
\DeclareMathOperator{\VROT}{\bf rot}

\newcommand{\compl}{{\rm c}}
 
\newcommand{\Hcurl}[1]{\bvec{H}(\CURL;#1)}

\newcommand{\Hdiv}[1]{\bvec{H}(\DIV;#1)}

\newcommand{\ser}{{\rm s}}
\newcommand{\Xgrad}[2][k]{\underline{X}_{\GRAD,#2}^{#1}}
\newcommand{\sXgrad}[2][k]{\su{X}_{\GRAD,#2}^{#1}}
\newcommand{\Xcurl}[2][k]{\uvec{X}_{\CURL,#2}^{#1}}
\newcommand{\sXcurl}[2][k]{\suvec{X}_{\CURL,#2}^{#1}}
\newcommand{\Xdiv}[2][k]{\uvec{X}_{\DIV,#2}^{#1}}
\newcommand{\sXdiv}[2][k]{\suvec{X}_{\DIV,#2}^{#1}}
\newcommand{\Xbullet}[2][k]{\underline{X}_{\bullet,#2}^{#1}}
\newcommand{\sXbullet}[2][k]{\s{X}_{\bullet,#2}^{#1}}

\newcommand{\Xg}[1]{X_{#1}}
\newcommand{\Hg}[1]{{\mathcal H}_{#1}}
\newcommand{\Lg}[1]{{\mathcal L}_{#1}}
\newcommand{\dom}{{\mathrm{dom}}}
\newcommand{\Ig}[1]{I_{#1}}
\newcommand{\dg}[1]{{\rm d}_{#1}}

\newcommand{\tXg}[1]{\s{X}_{#1}}
\newcommand{\tIg}[1]{\s{I}_{#1}}
\newcommand{\tdg}[1]{{\rm \s{d}}_{#1}}

\newcommand{\Eg}[1]{E_{#1}}
\newcommand{\Rg}[1]{\s{R}_{#1}}
\newcommand{\EPoly}[2][k-1]{E_{\Poly{},#2}^{#1}}
\newcommand{\RPoly}[1]{\s{R}^{\ell_{#1}}_{\Poly{},#1}}
\newcommand{\Egrad}[1]{\underline{E}_{\GRAD,#1}}
\newcommand{\Rgrad}[1]{\su{R}_{\GRAD,#1}}
\newcommand{\RRoly}[1]{\sbvec{R}^{k-1}_{\cvec{R},#1}}
\newcommand{\Ecurl}[1]{\uvec{E}_{\CURL,#1}}
\newcommand{\Rcurl}[1]{\suvec{R}_{\CURL,#1}}
\newcommand{\Ediv}[1]{\uvec{E}_{\DIV,#1}}
\newcommand{\Rdiv}[1]{\suvec{R}_{\DIV,#1}}
\newcommand{\Ebullet}[1]{\underline{E}_{\bullet,#1}}

\newcommand{\Id}{{\rm Id}}

\newcommand{\sfP}{{\mathsf{P}}}
\newcommand{\sfQ}{{\mathsf{Q}}}
\newcommand{\bdry}[1]{\mathcal{B}_{#1}}
\newcommand{\sbdry}[1]{\mathcal{B}_{\ser,#1}}
\newcommand{\sfb}{{\mathsf{b}}}
\newcommand{\dist}{\mathrm{dist}}

\newcommand{\Igrad}[1]{\underline{I}_{\GRAD,#1}^k}
\newcommand{\Icurl}[1]{\uvec{I}_{\CURL,#1}^k}
\newcommand{\Idiv}[1]{\uvec{I}_{\DIV,#1}^{k}}

\newcommand{\sIgrad}[1]{\su{I}_{\GRAD,#1}^k}
\newcommand{\sIcurl}[1]{\suvec{I}_{\CURL,#1}^k}

\newcommand{\lproj}[2]{\pi_{\Poly{},#2}^{#1}}

\newcommand{\Rproj}[2]{\bvec{\pi}_{\cvec{R},#2}^{#1}}
\newcommand{\Rcproj}[2]{\bvec{\pi}_{\cvec{R},#2}^{\compl,#1}}

\newcommand{\Gproj}[2]{\bvec{\pi}_{\cvec{G},#2}^{#1}}
\newcommand{\Gcproj}[2]{\bvec{\pi}_{\cvec{G},#2}^{\compl,#1}}

\newcommand{\Xcproj}[2]{\bvec{\pi}_{\cvec{X},#2}^{\compl,#1}}

\newcommand{\uGT}[1][k]{\uvec{G}_T^{#1}}
\newcommand{\uGF}[1][k]{\uvec{G}_F^{#1}}
\newcommand{\uGP}[1][k]{\uvec{G}_\sfP^{#1}}
\newcommand{\suGT}[1][k]{\suvec{G}_T^{#1}}
\newcommand{\suGF}[1][k]{\suvec{G}_F^{#1}}
\newcommand{\suGP}[1][k]{\suvec{G}_\sfP^{#1}}

\newcommand{\uCT}[1][k]{\uvec{C}_T^{#1}}
\newcommand{\suCT}[1][k]{\suvec{C}_T^{#1}}

\newcommand{\uGh}[1][k]{\uvec{G}_h^{#1}}
\newcommand{\suGh}[1][k]{\suvec{G}_h^{#1}}
\newcommand{\uCh}[1][k]{\uvec{C}_h^{#1}}
\newcommand{\suCh}[1][k]{\suvec{C}_h^{#1}}
\newcommand{\Dh}[1][k]{D_h^{#1}}

\newcommand{\GE}[1][k]{G_E^{#1}}
\newcommand{\cGF}[1][k]{\boldsymbol{\mathsf{G}}_F^{#1}}
\newcommand{\cGT}[1][k]{\boldsymbol{\mathsf{G}}_T^{#1}}
\newcommand{\cGP}[1][k]{\boldsymbol{\mathsf{G}}_\sfP^{#1}}

\newcommand{\CF}[1][k]{C_F^{#1}}
\newcommand{\cCT}[1][k]{\boldsymbol{\mathsf{C}}_T^{#1}}

\newcommand{\DT}[1][k]{D_T^{#1}}

\newcommand{\trE}{\gamma_E^{k+1}}
\newcommand{\trF}{\gamma_F^{k+1}}
\newcommand{\trFt}{\bvec{\gamma}_{{\rm t},F}^k}

\newcommand{\faces}[1]{\mathcal{F}_{#1}}
\newcommand{\edges}[1]{\mathcal{E}_{#1}}

\newcommand{\FT}{\faces{T}}
\newcommand{\ET}[1][T]{\edges{#1}}
\newcommand{\EF}{\edges{F}}

\newcommand{\normal}{\bvec{n}}
\newcommand{\tangent}{\bvec{t}}

\newcommand{\Poly}[2][]{\mathcal{P}_{#1}^{#2}}
\newcommand{\vPoly}[2][]{\cvec{P}_{#1}^{#2}}
\newcommand{\Roly}[1]{\cvec{R}^{#1}}
\newcommand{\Goly}[1]{\cvec{G}^{#1}}
\newcommand{\cRoly}[1]{\cvec{R}^{\compl,#1}}
\newcommand{\cGoly}[1]{\cvec{G}^{\compl,#1}}

\newcommand{\norm}[2][]{\|#2\|_{#1}}

\newcommand{\vvvert}{\vert\kern-0.25ex\vert\kern-0.25ex\vert}
\newcommand{\tnorm}[2][]{\vvvert #2\vvvert_{#1}}

\DeclareMathOperator{\Ker}{Ker}
\DeclareMathOperator{\Image}{Im}

\newcommand{\Mh}[1][h]{\mathcal{M}_{#1}}
\newcommand{\Th}[1][h]{\mathcal{T}_{#1}}
\newcommand{\Fh}[1][h]{\mathcal{F}_{#1}}
\newcommand{\Eh}[1][h]{\mathcal{E}_{#1}}
\newcommand{\Vh}{\mathcal{V}_h}

\newcommand{\Pgrad}[1][k+1]{P_{\GRAD,T}^{#1}}
\newcommand{\Pcurl}[1][T]{\bvec{P}_{\CURL,#1}^k}
\newcommand{\Pdiv}[1][T]{\bvec{P}_{\DIV,#1}^k}
\newcommand{\Pbullet}[1][l]{P_{\bullet,T}^{#1}}

\newcommand{\SG}[1]{\bvec{S}_{\GRAD,#1}^k}
\newcommand{\SC}[1]{\bvec{S}_{\CURL,#1}^k}

\newcommand{\sep}{$\diamond$}

\pgfplotsset{select coords between index/.style 2 args={
    x filter/.code={
        \ifnum\coordindex<#1\fi
        \ifnum\coordindex>#2\fi
    }
}}
\begin{document}

\title{Homological- and analytical-preserving serendipity framework for polytopal complexes, with application to the DDR method}

\author[1]{Daniele A. Di Pietro}
\author[2]{J\'er\^ome Droniou}
\affil[1]{IMAG, Univ Montpellier, CNRS, Montpellier, France, \email{daniele.di-pietro@umontpellier.fr}}
\affil[2]{School of Mathematics, Monash University, Melbourne, Australia, \email{jerome.droniou@monash.edu}}

\maketitle

\begin{abstract}
  In this work we investigate from a broad perspective the reduction of degrees of freedom through serendipity techniques for polytopal methods compatible with Hilbert complexes.
  We first establish an abstract framework that, given two complexes connected by graded maps, identifies a set of properties enabling the transfer of the homological and analytical properties from one complex to the other.
  This abstract framework is designed having in mind discrete complexes, with one of them being a reduced version of the other, such as occurring when applying serendipity techniques to numerical methods. We then use this framework as an overarching blueprint to design a serendipity DDR complex.
  Thanks to the combined use of higher-order reconstructions and serendipity, this complex compares favorably in terms of degrees of freedom (DOF) count to all the other polytopal methods previously introduced and also to finite elements on certain element geometries.
  The gain resulting from such a reduction in the number of DOFs is numerically evaluated on two model problems: a magnetostatic model, and the Stokes equations.
\end{abstract}

%% \tableofcontents

%------------------------------------------------------------------------------%

\section{Introduction}

\subsection{Hilbert complexes and their role in the stability of partial differential equations}

A \emph{Hilbert complex} is a sequence of Hilbert spaces $X_i$ connected by closed densely defined linear operators $\dg{i}:X_i\to X_{i+1}$ such that the range of $\dg{i}$ is contained in the kernel of $\dg{i+1}$ (\emph{complex property}).
The best-known example of Hilbert complex is the de Rham complex which, for a domain $\Omega$ of $\Real^3$, reads
\begin{equation*}%\label{eq:de-rham}
  \begin{tikzcd}
    \Real\arrow[r,hook] & H^1(\Omega)
    \arrow{r}{\GRAD} & \Hcurl{\Omega}
    \arrow{r}{\CURL} & \Hdiv{\Omega}
    \arrow{r}{\DIV} & L^2(\Omega)
    \arrow{r}{0} & 0,
  \end{tikzcd}
\end{equation*}
where $H^1(\Omega)$ is spanned by scalar-valued functions that are square-integrable over $\Omega$ along with their gradient, while
$\Hcurl{\Omega}$ and $\Hdiv{\Omega}$ are spanned by vector-valued functions that are square-integrable over $\Omega$ along with their curl and divergence, respectively.
The complex property corresponds, in this case,  to the classical relations $\GRAD C = \bvec{0}$ for all $C\in\Real$, $\CURL\GRAD = \bvec{0}$, and $\DIV\CURL = 0$.

Moreover, the divergence $\DIV:\Hdiv{\Omega}\to L^2(\Omega)$ is surjective and, depending on the topology of $\Omega$, the previous properties can become stronger.
Specifically, if $\Omega$ is not crossed by any tunnel (i.e., its first Betti number $b_1$ is zero), then $\Image\GRAD = \Ker\CURL$: for every function $\bvec{v}\in\Hcurl{\Omega}$ such that $\CURL\bvec{v} = \bvec{0}$, there exists a function $q\in H^1(\Omega)$ such that $\bvec{v} = \GRAD q$.
Similarly, if $\Omega$ does not enclose any void (i.e., its second Betti number $b_2$ is zero), then $\Image\CURL = \Ker\DIV$: for every function $\bvec{w}\in\Hdiv{\Omega}$ such that $\DIV\bvec{w} = 0$, there exists a function $\bvec{v}\in\Hcurl{\Omega}$ such that $\bvec{w} = \CURL\bvec{v}$.

The above \emph{exactness properties} are key to the well-posedness of certain classes of partial differential equations (PDEs).
An example is provided by the following formulation of magnetostatics \cite[Section 4.5.3]{Arnold:18}:
Given the magnetic permeability $\mu>0$ and the free current density $\bvec{\cal J}\in\CURL\Hcurl{\Omega}$, find the magnetic field $\bvec{\cal H}\in\Hcurl{\Omega}$ and the vector potential $\bvec{\cal A}\in\Hdiv{\Omega}$ such that
\begin{equation}\label{eq:magnetostatics}
  \begin{alignedat}{2}
    \int_\Omega\mu\bvec{\cal H}\cdot\bvec{v}
    - \int_\Omega\bvec{\cal A}\cdot\CURL\bvec{v}
    &= 0
    &\quad&\forall\bvec{v}\in\Hcurl{\Omega},
    \\
    \int_\Omega\CURL\bvec{\cal H}\cdot\bvec{w}
    + \int_\Omega\DIV\bvec{\cal A}\DIV\bvec{w}
    &= \int_\Omega\bvec{\cal J}\cdot\bvec{w}
    &\quad&\forall\bvec{w}\in\Hdiv{\Omega}.
  \end{alignedat}
\end{equation}
The first equation is the definition of the vector potential and incorporates the homogeneous natural boundary condition on the tangential trace of $\bvec{\cal A}$,
while the second expresses both Amp\`ere's law and the so-called Coulomb gauge.
The well-posedness of this problem hinges heavily on the exactness of the second portion of the de Rham complex. 
Specifically, to bound $\bvec{\cal A}$, we use $L^2$-orthogonal complements to write
\begin{equation}\label{eq:decompose.A}
\bvec{\cal A}=\bvec{\cal A}^*+\bvec{\cal A}^\bot\in \ker\DIV\oplus(\ker\DIV)^\bot=\Image\CURL \oplus(\ker\DIV)^\bot,
\end{equation}
where the last equality comes form the exactness property $\Image\CURL = \Ker\DIV$.
The component $\bvec{\cal A}^*\in\Image\CURL$ is then estimated taking $\bvec{v}$ such that $\CURL\bvec{v}=\bvec{\cal A}^*$ as test function in \eqref{eq:magnetostatics}, while a bound on $\bvec{\cal A}^\bot\in (\ker\DIV)^\bot$ is obtained using the fact that $\DIV:(\ker\DIV)^\bot\to L^2(\Omega)$ is an isomorphism, due to the exactness property $\Image \DIV=L^2(\Omega)$ (see \cite[Section 2]{Di-Pietro.Droniou.ea:20} for the detailed analysis).

When writing numerical schemes for problem \eqref{eq:magnetostatics}, reproducing exactness properties at the discrete level is crucial for stability (see, e.g., \cite[Theorem 5.4]{Arnold:18} or \cite[Theorem 10]{Di-Pietro.Droniou:21*1}), and essentially requires to derive discrete counterparts of Hilbert complexes, with continuous spaces replaced by finite-dimensional analogs.
Methods based on (exact) discrete Hilbert complexes are often referred to as \emph{compatible}.
A very general exposition of compatible finite element methods can be found in \cite{Arnold:18}.
Finite elements are, however, essentially confined to conforming meshes composed of tetrahedral or hexahedral elements.
Since the early 2010s, the development of methods supporting meshes made of much more general polytopal elements, as well as arbitrary order of accuracy, has been an extremely active field of research; a representative but by far non exhaustive list of references includes \cite{Di-Pietro.Ern:10,Di-Pietro.Ern:12,Antonietti.Cangiani.ea:16,Beirao-da-Veiga.Brezzi.ea:13,Beirao-da-Veiga.Brezzi.ea:14,Di-Pietro.Ern:15,Cockburn.Di-Pietro.ea:16,Di-Pietro.Droniou:20}.
Recent efforts have focused on the compatibility of polytopal methods; see, in particular, \cite{Beirao-da-Veiga.Brezzi.ea:16,Beirao-da-Veiga.Brezzi.ea:18*2} and \cite{Di-Pietro.Droniou.ea:20,Di-Pietro.Droniou:21*2} concerning, respectively, Virtual Element (VEM) and Discrete de Rham (DDR) methods, as well as \cite{Beirao-da-Veiga.Dassi.ea:22}, where bridges between these technologies have been built.

It is worth noticing that the notion of compatibility can be extended to the case when exactness does not hold.
Coming back to our example, when the Betti numbers $b_1$ or $b_2$ are nonzero, de Rham's cohomology characterises the non-trivial \emph{cohomology spaces} $\Ker\CURL/\Image\GRAD$ and $\Ker\DIV/\Image\CURL$.
In this case, $\ker\DIV=\Image\CURL\oplus {\mathfrak H}$ where $\mathfrak H\approx\Ker\DIV/\Image\CURL$ is the space of harmonic $2$-forms, and the decomposition \eqref{eq:decompose.A} becomes
\[
\bvec{\cal A}=\bvec{\cal A}^*+\bvec{\cal A}^\sharp+\bvec{\cal A}^\bot\in \Image\CURL \oplus\,\mathfrak H\oplus (\ker\DIV)^\bot.
\]
The weak formulation \eqref{eq:magnetostatics} must then be supplemented by conditions that enforce the orthogonality of the potential $\bvec{\cal A}$ to $\mathfrak H$ in order to control (in fact, eliminate) the component $\bvec{\cal A}^\sharp$; see \cite[Section 4.5.3]{Arnold:18}.
In these situations of non-exactness of the complexes, the notion of compatibility of the discretisation translates into the existence of isomorphisms between the cohomology spaces of the continuous and discrete complexes.

\subsection{Polytopal discrete Hilbert complexes}

Vanilla versions of polytopal methods often display more degrees of freedom (DOFs) than the corresponding finite element counterparts.
A remedy to this drawback has been proposed in \cite{Beirao-da-Veiga.Brezzi.ea:18}, where a version of the nodal ($H^1$-conforming) VEM space with fewer face DOFs has been introduced. 
The idea, inspired by serendipity finite elements, consists in selecting a subset of the DOFs sufficient to uniquely identify polynomials of the desired degree, and in using them to fix the values of the remaining DOFs.
A similar path had been previously followed in \cite{Di-Pietro:12} to reduce the number of element DOFs in the framework of discontinuous Galerkin methods.
In the context of low-order methods, this approach is also called barycentric elimination \cite{Droniou.Eymard.ea:18}, and has been used to eliminate cell unknowns in hybrid methods \cite{Eymard.Gallouet.ea:10}.
Notice that the word \emph{serendipity} historically originates from the fact that, ``by chance'', internal DOFs can be eliminated from low-order rectangular $Q^k$ finite elements without compromising accuracy; however, the above approaches, as well as the one developed in this paper, are systematic and apply to much more general geometries.

Serendipity techniques for DOFs reduction are appealing for several reasons:
they enable an interfacing of finite element and polytopal methods on standard geometries, as argued in \cite{Beirao-da-Veiga.Brezzi.ea:18};
they lead to smaller algebraic systems by reducing the number of face DOFs (which cannot be, in general, eliminated by static condensation);
in the context of nonlinear problems, they provide a means to eliminate (a portion of the) element DOFs that is potentially more efficient than static condensation, as the serendipity reduction operators need not be recomputed at each iteration of the nonlinear solver.

In the context of compatible polytopal methods, serendipity techniques have to be applied in such a way as to preserve the homological and analytical properties of the discrete complex, that is, its cohomology spaces, and its stability and approximation properties.
Concerning VEM methods, the issue of face DOFs reduction through compatible serendipity techniques has been considered in \cite{Beirao-da-Veiga.Brezzi.ea:17,Beirao-da-Veiga.Brezzi.ea:18*1}, where the authors build a virtual serendipity de Rham complex for which they provide a direct proof of local exactness properties; see, in particular, \cite[Propositions 3.2, 3.6, and 3.8]{Beirao-da-Veiga.Brezzi.ea:18*1}.
A variation of the VEM complex in the previous reference has been recently proposed in \cite{Beirao-da-Veiga.Dassi.ea:22}, where links with DDR have also been established.

Inspired by these recent results, we tackle in this paper the issue of DOFs reduction for compatible polytopal methods from a broader perspective.
  Specifically, we start by constructing an abstract framework which, given two complexes connected by \emph{graded maps} (that is, a family of maps between each space of the first complex and its counterpart in the second complex), identifies a set of requirements on these maps that enable the transfer of all the relevant homological and analytical results from one complex to the other. In particular, the properties we identify ensure that the original and serendipity complexes have isomorphic cohomologies, even in the case of non-exact complexes due to non-trivial topologies of the domain.
  Moreover, the identification of sufficient conditions for the preservation of analysis results such as Poincar\'e inequalities, primal and adjoint consistency, etc.~in a general nonconforming setting seems to be entirely new (in a conforming setting, the preservation of such properties is significantly easier).
  
This general framework is specifically adapted to the situation where one complex is a reduced version of the other, and the graded connecting maps are extensions/reductions which transfer properties from the full complex to the reduced one.
This framework then enables us to derive a serendipity version (hereafter referred to as SDDR) of the DDR complex of \cite{Di-Pietro.Droniou:21*2}.
The SDDR complex hinges on novel serendipity operators that rely on the same construction for both the discrete $H^1$ and $\bvec{H}(\CURL)$ spaces, and that enable the reduction of both faces and element DOFs.
The fulfillment of the requirements identified in the abstract framework ensures that the full panel of homological and analytical results derived in \cite{Di-Pietro.Droniou:21*2} transfers to the serendipity DDR complex.
The new SDDR complex is extremely efficient as it combines two DOFs-reduction strategies: the use of higher-order reconstructions, inherent to DDR and akin to the enhancement techniques in VEM, and serendipity on both faces and elements.
The DOFs count therefore compares favorably to existing polytopal de Rham complexes \cite{Beirao-da-Veiga.Brezzi.ea:16,Beirao-da-Veiga.Brezzi.ea:18*1,Beirao-da-Veiga.Dassi.ea:22,Di-Pietro.Droniou.ea:20,Di-Pietro.Droniou:21*1}, but also to finite elements on hexahedra (while being very close, although slightly higher, on tetrahedra).
  The achieved DOFs reduction for the discrete counterparts of the spaces $H^1(T)$ and $\Hcurl{T}$, when the element $T$ is a tetrahedron or a hexahedron and polynomial degrees $k$ between 1 and 3 are considered, is summarised in Table \ref{tab:summary.reduction} (see also Table \ref{tab:dof.count} for further details and a comparison with classical finite element spaces for $k=0,1,2$); serendipity procedures do not achieve any DOFs reduction for the lowest approximation order $k=0$, where internal DOFs are not present.
  We also represent in Figure \ref{fig:dofs} the degrees of freedom on a triangular and pentagonal face (the later being a typical face in a Voronoi tessellation, as the ones used in the tests of Section \ref{sec:tests}).
  \begin{table}
    \begin{center}
      \begin{tabular}{c|ccc|ccc}
        \toprule
        \multirow{2}{*}{Discrete space} & \multicolumn{3}{c|}{Tetrahedron} & \multicolumn{3}{c}{Hexahedron} \\
        & $k=1$ & $k=2$ & $k=3$ & $k=1$ & $k=2$ & $k=3$ \\
        \midrule
        $H^1(T)$  & -5 (-33\%)& -12 (-37.5\%) & -21 (-37.5\%)& -7 (-26\%) & -22 (-41\%) & -40 (-44\%)\\
        $\Hcurl{T}$ & -5 (-18\%) & -12 (-18\%) & -21 (-17.5\%)& -7 (-15\%) & -22 (-22\%) & -40 (-23\%)\\
        \bottomrule
      \end{tabular}
      \caption{Summary of the serendipity reduction of the number of DOFs for the DDR complex, for $T$ a tetrahedron or a hexahedron.\label{tab:summary.reduction}}
    \end{center}
  \end{table}
  \begin{figure}
    \begin{center}
      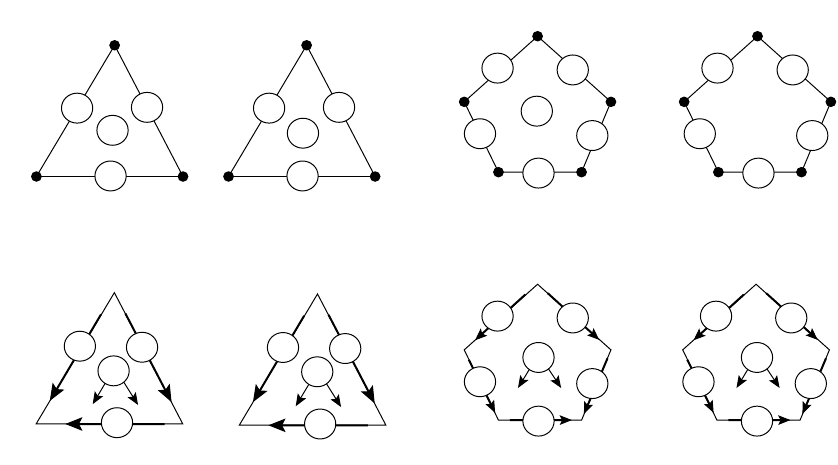
    \caption{Degrees of freedom on a triangular and pentagonal faces of the discrete $H^1$ and $\bvec{H}(\CURL)$ spaces for the DDR and SDDR methods, with polynomial degree $k=3$. A black circle represents a pointwise evaluation, a circle with number $N$ represents a scalar projection on a polynomial space of dimension $N$, and a circle with number $N$ and arrow(s) represents a tangential projection on a polynomial space of dimension $N$ (which is vector-valued, in case of two arrows).}
    \label{fig:dofs}
    \end{center}
  \end{figure}
It has to be noticed that serendipity versions of finite elements on cubical meshes have been recently explored with a DOF count comparable to the present construction (see, e.g., \cite{Gillette.Hu.ea:20} and references therein), but an analysis is presented and shape functions exhibited only for the lowest-order degree.
  The method discussed here is much more general as it applies to meshes of polyhedral elements, and is additionally designed so as to preserve the commutation properties between the interpolator and discrete differential operators discussed in \cite[Section 3.5]{Di-Pietro.Droniou:21*2}.
  
Having clarified the requirements for compatible serendipity techniques through an abstract framework also enables us to provide an answer to an open question left in \cite{Beirao-da-Veiga.Brezzi.ea:17}, namely whether it is possible to perform DOFs reduction on the discrete $\bvec{H}(\DIV)$ spaces in a compatible manner.
Specifically, we prove in this paper that such a reduction is not possible as it would violate crucial requirements of the discrete complex.

\subsection{Structure of the paper}

The rest of the paper is organised as follows.
In Section \ref{sec:blueprint} we develop an abstract framework for the preservation of the homological and analytical properties in reduced complexes.
Section \ref{sec:setting} contains a description of the setting (mesh, polynomial spaces), while the original DDR complex of \cite{Di-Pietro.Droniou:21*2} is briefly recalled in Section \ref{sec:ddr}.
The construction of the novel serendipity DDR complex is detailed in Section \ref{sec:serendipity.ddr} and its properties are proved in Section \ref{sec:prop.sddr}, which also contains numerical tests to assess the performance gain resulting from the reduction of DOFs.
Finally, Appendix \ref{sec:estimates.proofs} contains the proofs of estimates of polynomials from boundary values that are at the origin of the design of the serendipity operators.

%------------------------------------------------------------------------------%

\section{Blueprint for homological and analytical properties preservation in reduced complexes}\label{sec:blueprint}

Consider a sequence $(\Xg{i},\dg{i})_i$ of spaces and (discrete) operators defining a complex.
An important notion for discrete complexes is that of \emph{consistency}, which  determines a way to ``embed'' polynomials of certain degrees into the spaces of the complex through the use of \emph{interpolators}, and states that, when applied to these embedded polynomials, the discrete operators (e.g.~the discrete gradient) act as their continuous counterparts (e.g.~the usual gradient). To precisely state such consistency properties, we therefore select, for each space $\Xg{i}$, a (possibly vector-valued) polynomial space $\Poly{k_i}$, and an interpolator $\Ig{i}:\Poly{k_i}\to\Xg{i}$.

\begin{equation}\label{eq:ser.diagram}
  \begin{tikzcd}
    & 
    \arrow{d}{\Ig{i}}\Poly{k_i}  &[3em] \Poly{k_{i+1}}
    \arrow{d}{\Ig{i+1}} & \\[2em]
    \cdots\arrow{r} & \Xg{i}
    \arrow{r}{\dg{i}}\arrow[bend left, dashed]{d}{\Rg{i}} &[3em] \Xg{i+1}
    \arrow{r}{}\arrow[bend left, dashed]{d}{\Rg{i+1}} & \cdots\\[2em]
    \cdots\arrow{r} & \tXg{i}
    \arrow{r}{\tdg{i}}\arrow[bend left]{u}{\Eg{i}} &[3em] \tXg{i+1}
    \arrow{r}{}\arrow[bend left]{u}{\Eg{i+1}} & \cdots
  \end{tikzcd}
\end{equation}

Assume that we have a second sequence of spaces $(\tXg{i})_i$ and connecting linear graded maps $\Eg{}$ and $\Rg{}$ as in Diagram \eqref{eq:ser.diagram}, and set, for all index $i$,
\begin{equation}
  \label{eq:def.tdg}
  \tdg{i}\coloneq\Rg{i+1}\dg{i}\Eg{i}.
\end{equation}
Having this diagram in mind, we will refer to $(\Xg{i},\dg{i})_i$ as the \emph{top sequence} and to $(\tXg{i},\tdg{i})_i$ as the \emph{bottom sequence}.
Since we think of the spaces $\tXg{i}$ as ``reduced'' versions (with smaller dimension) of the spaces $\Xg{i}$, we hereafter call \emph{reductions} the maps $\Rg{}$ and \emph{extensions} the maps $\Eg{}$.
We now list properties on the connecting maps $\Rg{}$, $\Eg{}$ that ensure that the two sequences have the same cohomology.

\begin{assumption}[Homological properties]\label{ass:cohomology}~
  \begin{enumerate}[label=\textbf{(C{\arabic*})}]
  \item\label{P:extension} $\Rg{i}$ is a left-inverse of $\Eg{i}$ on $\ker\tdg{i}$, that is: $(\Rg{i}\Eg{i})_{|\ker\tdg{i}}=\Id_{\ker\tdg{i}}$.
  \item\label{P:complex.ER} $(\Eg{i+1}\Rg{i+1}-\Id_{\Xg{i+1}})(\ker\dg{i+1})\subset \Image\dg{i}$,
  \item\label{P:E.R.cochain} Both the reduction and extension are cochain maps, i.e.,
    $\Rg{i+1}\dg{i} = \tdg{i}\Rg{i}$ and
    $\Eg{i+1}\tdg{i} = \dg{i}\Eg{i}$.
    Equivalently, recalling the definition \eqref{eq:def.tdg} of $\tdg{i}$,
    $\Rg{i+1}\dg{i} = \Rg{i+1}\dg{i}\Eg{i}\Rg{i}$
    and $\Eg{i+1}\Rg{i+1}\dg{i}\Eg{i} = \dg{i}\Eg{i}$.
  \end{enumerate}
\end{assumption}

\begin{proposition}[Isomorphic cohomologies]\label{prop:cohomologies}
  Under Assumption \ref{ass:cohomology}, the sequence $(\tXg{i},\tdg{i})_i$ is a complex and the cohomologies of $(\Xg{i},\dg{i})_i$ and $(\tXg{i},\tdg{i})_i$ are isomorphic.
  In particular, if one complex is exact then so is the other.
\end{proposition}

\begin{proof}
  The fact that $(\tXg{i},\tdg{i})_i$ is a complex follows from \eqref{eq:def.tdg} and \ref{P:E.R.cochain}.
  Denoting now the cohomology groups by $\mathsf{H}_i\coloneq\ker\dg{i}/\Image\dg{i-1}$ and $\s{\mathsf{H}}_i\coloneq\ker\tdg{i}/\Image\tdg{i-1}$, by \ref{P:E.R.cochain} we can project $\Eg{i}$ and $\Rg{i}$ on these groups, defining $[\Eg{i}]:\s{\mathsf{H}}_i\to \mathsf{H}_i$ and $[\Rg{i}]:\mathsf{H}_i\to \s{\mathsf{H}}_i$.
  By \ref{P:extension}, we have $[\Rg{i}][\Eg{i}]=\Id_{\s{\mathsf{H}}_i}$ and,
  by \ref{P:complex.ER} for $i-1$, $[\Eg{i}\Rg{i}-\Id_{\Xg{i}}]=0$, so that $[\Eg{i}][\Rg{i}]=\Id_{\mathsf{H}_i}$.
  This proves that $[\Eg{i}]:\s{\mathsf{H}}_i\to \mathsf{H}_i$ and $[\Rg{i}]:\mathsf{H}_i\to\s{\mathsf{H}}_i$ are isomorphisms.
\end{proof}

Our goal now is to deduce analytical properties of the bottom sequence based on corresponding properties of the top sequence. To do so, we assume that each $\Xg{i}$ is equipped with an inner product $(\cdot,\cdot)_{\Xg{},i}$, with norm $\norm[\Xg{},i]{{\cdot}}$.
The inner product
$(\cdot,\cdot)_{\tXg{},i}$ and associated norm $\norm[\tXg{},i]{{\cdot}}$ on $\tXg{i}$ are defined so that $\Eg{i}$ is an isometry, that is:
\begin{equation}\label{eq:def.ps.tXg}
  (\s{x},\s{y})_{\tXg{},i}\coloneq (\Eg{i}\s{x},\Eg{i}\s{y})_{\Xg{},i}\quad\mbox{ and }\quad
  \norm[\tXg{},i]{\s{x}}=\norm[\Xg{},i]{\Eg{i}\s{x}}\quad\forall (\s{x},\s{y})\in \tXg{i}\times\tXg{i}.
\end{equation}

Of importance when considering discrete complexes is the degree of polynomial consistency they preserve, which is measured through the interpolators $\Ig{i}$ in \eqref{eq:ser.diagram}. Oftentimes, some properties involving interpolators on larger spaces are also useful. We will therefore consider that, for each $\Xg{i}$, there is a Hilbert space $\Hg{i}\supset\Poly{k_i}$ such that the interpolator $\Ig{i}$ can be extended as $\Ig{i}:\Hg{i}\to\Xg{i}$. The norm on $\Hg{i}$ is denoted by $\norm[\Hg{},i]{{\cdot}}$, and we define interpolator for the bottom sequence by:
\begin{equation}
  \label{eq:def.tIg}
  \tIg{i}\coloneq\Rg{i}\Ig{i}:\Hg{i}\to\tXg{i}.
\end{equation}
To properly define the interpolator, the norm on $\Hg{i}$ is typically quite ``strong'' (this space corresponds to a Sobolev space). Some estimates and consistency properties will require an ``$L^2$-like'' space $\Lg{i}$ (with inner product $(\cdot,\cdot)_{\Lg{},i}$ and norm $\norm[\Lg{},i]{{\cdot}}$) in which $\Hg{i}$ is continuously embedded. We denote by $c_{\Hg{},i}$ the embedding constant such that
\begin{equation}\label{eq:embedding.Hg.Lg}
  \norm[\Lg{},i]{v}\le c_{\Hg{},i}\norm[\Hg{},i]{v}\qquad\forall v\in\Hg{i}.
\end{equation}

We now cite three assumptions that are involved in deducing analytical properties of the bottom sequence from properties of the top sequence.
\begin{assumption}[Analytical properties]\label{ass:analytical}~
  \begin{enumerate}[label=\textbf{(A{\arabic*})}]
  \item\label{P:continuity} \emph{Continuity of reductions}: $\Rg{i}:\Xg{i}\to\tXg{i}$ is continuous, with norm denoted by $\norm{\Rg{i}}$.
  \item\label{P:consistency} \emph{Polynomial consistency}: $\Rg{i}$ is a right-inverse of $\Eg{i}$ on the subspace $\Ig{i}(\Poly{k_i})$; recalling \eqref{eq:def.tIg}, this translate into: $\Eg{i}\tIg{i}x_p=\Ig{i}x_p$ for all $x_p\in\Poly{k_i}$.
  \item\label{P:continuity.interpolator} \emph{Continuity of interpolators}: $\Ig{i}:\Hg{i}\to\Xg{i}$ is continuous for the norms $\norm[\Hg{},i]{{\cdot}}$ and $\norm[\Xg{},i]{{\cdot}}$, with norm denoted by $\norm{\Ig{i}}$.
  \end{enumerate}
\end{assumption}

\begin{proposition}[Poincar\'e inequality]\label{prop:poincare}
  Assume \ref{P:E.R.cochain} and \ref{P:continuity}.
  If $(\Xg{i},\dg{i})$ satisfies a Poincar\'e inequality, then so does $(\tXg{i},\tdg{i})$. More specifically, if there is $c_P\ge 0$ such that
  \begin{equation}\label{eq:Poincare.top}
    \norm[\Xg{},i]{x}\le c_P\norm[\Xg{},i+1]{\dg{i}x}\qquad\forall x\in(\ker\dg{i})^\bot,
  \end{equation}
  then
  \begin{equation*}%% \label{eq:Poincare.bottom}
    \norm[\tXg{},i]{\s{x}}\le c_P\norm{\Rg{i}}\norm[\tXg{},i+1]{\tdg{i}\s{x}}\qquad\forall \s{x}\in(\ker\tdg{i})^\bot.
  \end{equation*}
\end{proposition}

\begin{proof}
  The proof is inspired by that done for conforming discrete complexes in \cite[Theorem 5.3]{Arnold:18}.
  Let $\s{x}\in(\ker\tdg{i})^\bot$. Since $\dg{i}:(\ker\dg{i})^\bot\to\Image\dg{i}$ is an isomorphism, there is $x\in(\ker\dg{i})^\bot$ such that $\dg{i}x=\dg{i}\Eg{i}\s{x}$.
  Using the cochain map property \ref{P:E.R.cochain} for the reduction, we infer that $\tdg{i}\Rg{i}x=\Rg{i+1}\dg{i}x=\Rg{i+1}\dg{i}\Eg{i}\s{x}=\tdg{i}\s{x}$ (where we have used \eqref{eq:def.tdg} to conclude), and thus that $\Rg{i}x-\s{x}\in\ker\tdg{i}$.
  Hence, $\s{x}\bot(\s{x}-\Rg{i}x)$, which gives
  \[
  \norm[\tXg{},i]{\s{x}}^2=(\s{x},\s{x})_{\tXg{},i}=(\s{x},\Rg{i}x)_{\tXg{},i}\le \norm[\tXg{},i]{\s{x}}\norm[\tXg{},i]{\Rg{i}x}\le\norm[\tXg{},i]{\s{x}}\norm{\Rg{i}}\norm[\Xg{},i]{x},
  \]
  where the last inequality follows from \ref{P:continuity}. Simplifying by $\norm[\tXg{},i]{\s{x}}$ and applying \eqref{eq:Poincare.top}, we obtain
  \[
  \norm[\tXg{},i]{\s{x}}
  \le c_P\norm{\Rg{i}}\norm[\Xg{},i+1]{\dg{i}x}
  = c_P\norm{\Rg{i}}\norm[\Xg{},i+1]{\dg{i}\Eg{i}\s{x}}
  = c_P\norm{\Rg{i}}\norm[\Xg{},i+1]{\Eg{i+1}\tdg{i}\s{x}},
  \]
  where we have used the cochain map property \ref{P:E.R.cochain} of the extensions in the last equality.
  The proof is completed by recalling that $\Eg{i+1}$ is an isometry by \eqref{eq:def.ps.tXg}, so that $\norm[\Xg{},i+1]{\Eg{i+1}\tdg{i}\s{x}} = \norm[\tXg{},i+1]{\tdg{i}\s{x}}$.
\end{proof}

\begin{proposition}[Primal consistency of inner products]\label{prop:primal.consistency}
  Assume \ref{P:consistency}, \ref{P:continuity.interpolator}, and suppose that the top sequence satisfies the polynomial consistency property
  \[
  (\Ig{i}v_p,\Ig{i}w_p)_{\Xg{},i}=(v_p,w_p)_{\Lg{},i}\qquad\forall v_p,w_p\in\Poly{k_i}.
  \]
  Then, the bottom sequence satisfies the following properties:
  \begin{enumerate}
  \item \emph{Polynomial consistency}:
    \begin{equation}\label{eq:t.polynomial.consistency}
      (\tIg{i}v_p,\tIg{i}w_p)_{\tXg{},i}=(v_p,w_p)_{\Lg{},i}
      \qquad\forall v_p,w_p\in\Poly{k_i}.
    \end{equation}

  \item \emph{Consistency in $\Hg{i}$}: Assuming \ref{P:continuity}, for all $v,w\in\Hg{i}$,
    \begin{align*}
      \left|(\tIg{i}v,\tIg{i}w)_{\tXg{},i}-(v,w)_{\Lg{},i}\right|\le{}&
      \left(\norm{\Rg{i}}^2\norm{\Ig{i}}^2+c_{\Hg{},i}^2\right)\\
      &\times\left(\omega_{\Hg{},\Poly{},i}(v)\norm[\Hg{},i]{w}
      +\omega_{\Hg{},\Poly{},i}(w)\norm[\Hg{},i]{v}
      +\omega_{\Hg{},\Poly{},i}(v)\omega_{\Hg{},\Poly{},i}(w)\right),
    \end{align*}
    where $\omega_{\Hg{},\Poly{},i}$ measures the best polynomial approximation error in the norm of $\Hg{i}$, i.e.,
    \begin{equation*}%% \label{def:omega.approx}
      \omega_{\Hg{},\Poly{},i}(\xi)\coloneq\inf_{\xi_p\in\Poly{k_i}}
      \norm[\Hg{},i]{\xi-\xi_p}
      \qquad\forall \xi\in\Hg{i}.
    \end{equation*}
  \end{enumerate}
\end{proposition}

\begin{remark}[Scaling of the $\Hg{i}$-norm]
To ensure that $\omega_{\Hg{},\Poly{},i}(\xi)$ has optimal decay properties with respect to the meshsize, the norm $\norm[\Hg{},i]{{\cdot}}$ on $\Hg{i}$ should be chosen to scale as an $L^2$-norm (e.g., multiplying norms of derivatives of $\xi$ by the appropriate characteristic length; see for example the choices in \cite[Lemma 6]{Di-Pietro.Droniou:21*2}).
\end{remark}

\begin{proof}
  For the polynomial consistency we simply write, for $v_p,w_p\in\Poly{k_i}$,
  \[
  (\tIg{i}v_p,\tIg{i}w_p)_{\tXg{},i} \eqtext{\eqref{eq:def.ps.tXg}}
  (\Eg{i}\tIg{i}v_p,\Eg{i}\tIg{i}w_p)_{\Xg{},i} \eqtext{\ref{P:consistency}}
  (\Ig{i}v_p,\Ig{i}w_p)_{\Xg{},i} =(v_p,w_p)_{\Lg{},i}.
  \]
  We now turn to the consistency on $\Hg{i}$. Let $v,w\in\Hg{i}$, and take arbitrary $v_p,w_p\in\Poly{k_i}$. Introducing $\pm\tIg{i}v_p$ and $\pm\tIg{i}w_p$ we have
  \begin{equation}\label{eq:cons.ip}
    (\tIg{i}v,\tIg{i}w)_{\tXg{},i}=(\tIg{i}(v-v_p),\tIg{i}w)_{\tXg{},i}+(\tIg{i}v_p,\tIg{i}(w-w_p))_{\tXg{},i}
    +(\tIg{i}v_p,\tIg{i}w_p)_{\tXg{},i}.
  \end{equation}
  Invoking the polynomial consistency \eqref{eq:t.polynomial.consistency}, we can write, for the last term in \eqref{eq:cons.ip},
  \[
  (\tIg{i}v_p,\tIg{i}w_p)_{\tXg{},i}=(v_p,w_p)_{\Lg{},i}=(v_p-v,w_p)_{\Lg{},i}+(v,w_p-w)_{\Lg{},i}+(v,w)_{\Lg{},i}.
  \]
  Plugging this into \eqref{eq:cons.ip} and using $\norm[\tXg{},i]{\tIg{i}\xi}\le \norm{\Rg{i}}\norm{\Ig{i}}\norm[\Hg{},i]{\xi}$ for all $\xi\in\Hg{i}$ (which results from \eqref{eq:def.tIg}, \ref{P:continuity}, and \ref{P:continuity.interpolator}) along with \eqref{eq:embedding.Hg.Lg}, we obtain
  \begin{align*}
    \left|(\tIg{i}v,\tIg{i}w)_{\tXg{},i}-(v,w)_{\Lg{},i}\right|\le{}&
    \norm{\Rg{i}}^2\norm{\Ig{i}}^2
      \left(\norm[\Hg{},i]{v-v_p}\norm[\Hg{},i]{w}+\norm[\Hg{},i]{w-w_p}\norm[\Hg{},i]{v_p}\right)\\
    &+
    c_{\Hg{},i}^2\left(\norm[\Hg{},i]{v_p-v}\norm[\Hg{},i]{w_p}+\norm[\Hg{},i]{w_p-w}\norm[\Hg{},i]{v}\right).
  \end{align*}
  The proof is completed by writing $\norm[\Hg{},i]{v_p}\le\norm[\Hg{},i]{v_p-v}+\norm[\Hg{},i]{v}$ and similarly for $w_p$, and taking the infimum over $v_p,w_p\in\Poly{k_i}$.
\end{proof}

As the sequences we consider in the following sections are ``fully discrete'' (that is, the discrete spaces are those of the degrees of freedom of the method), they are not naturally related to any polynomial function on the computational domain. Instead, such functions are obtained through \emph{potential reconstructions}, mappings that build from a set of degrees of freedom some polynomial functions on the mesh elements. We therefore assume, for all index $i$, the existence of a linear map $P_i:\Xg{i}\to\Poly{k_i}$, the norm of which (when $\Poly{k_i}$ is endowed with the $\Lg{i}$-norm) is denoted by $\norm{P_i}$. We then set
\[
\s{P}_i\coloneq P_i\Eg{i}:\tXg{i}\to\Poly{k_i}.
\]
The proof of the following primal consistency result is done as the proof of Proposition \ref{prop:primal.consistency} and is therefore omitted.

\begin{proposition}[Primal consistency of potential reconstructions]\label{prop:primal.consistency.Pi}
  Assume \ref{P:continuity}, \ref{P:consistency}, and \ref{P:continuity.interpolator}, and that $P_i$ is polynomially consistent, in the sense that $P_i\Ig{i}v_p=v_p$ for all $v_p\in\Poly{k_i}$.
  Then, $\s{P}_i$ is also polynomially consistent, in the sense that $\s{P}_i\tIg{i}v_p=v_p$ for all $v_p\in\Poly{k_i}$, and we moreover have
  \[
  \norm[\Lg{},i]{\s{P}_i\tIg{i}v-v}\le
  \left(\norm{P_i}\norm{\Rg{i}}\norm{\Ig{i}}+1\right)
  \omega_{\Hg{},\Poly{},i}(v),\qquad\forall v\in\Hg{i}.
  \]
\end{proposition}

Commutation properties for discrete differential operators are essential tools for designing schemes that are robust with respect to physical parameters \cite{Di-Pietro.Droniou:21,Beirao-da-Veiga.Dassi.ea:22}.
They also naturally lead to optimal consistency properties.
Both properties are considered in the following proposition.

\begin{proposition}[Commutation and consistency of differential operators]\label{prop:commutation.ID=dI}
  Assume \ref{P:E.R.cochain} and that, for some unbounded operator $D_i:\Hg{i}\to\Hg{i+1}$ with domain $\dom(D_i)$, the commutation property $\Ig{i+1}D_i=\dg{i}\Ig{i}$ holds on $\dom(D_i)$ for the top sequence. Then, it also holds for the bottom sequence: $\tIg{i+1}D_i=\tdg{i}\tIg{i}$ on $\dom(D_i)$.
  
  As a consequence, if \ref{P:continuity}, \ref{P:consistency}, and \ref{P:continuity.interpolator} hold, and if the potential reconstruction $P_{i+1}$ is polyno\-mial\-ly consistent, then
  \begin{equation}\label{eq:consistency.tdg}
    \norm[\Lg{},i+1]{\s{P}_{i+1}\tdg{i}\tIg{i}v-D_iv}\le
    \left(\norm{P_{i+1}}\norm{\Rg{i+1}}\norm{\Ig{i+1}}+1\right)
    \omega_{\Hg{},P,i+1}(D_iv)\qquad
  \forall v\in \dom(D_i).
  \end{equation}
\end{proposition}

\begin{proof}
  We simply write  
  \[
  \tdg{i}\tIg{i} \eqtext{\eqref{eq:def.tIg}} 
  \tdg{i}\Rg{i}\Ig{i} \eqtext{\ref{P:E.R.cochain}}
  \Rg{i+1}\dg{i}\Ig{i}=\Rg{i+1}\Ig{i+1}D_i \eqtext{\eqref{eq:def.tIg}}
  \tIg{i+1}D_i,
  \]
  where the third equality comes from the commutation property for the top sequence.
  The consistency property \eqref{eq:consistency.tdg} then directly follows by writing $\s{P}_{i+1}\tdg{i}\tIg{i}v-D_iv=\s{P}_{i+1}\tIg{i+1} D_iv-D_iv$ and applying Proposition \ref{prop:primal.consistency.Pi} with $(i+1,D_iv)$ instead of $(i,v)$. 
\end{proof}

\begin{remark}[Consistency in non-Hilbert spaces]
The primal consistency of the potential reconstruction and the discrete differential operator do not rely on the inner product of the spaces $\Hg{}$, and could therefore also be stated with non-Hilbert spaces.
\end{remark}

Adjoint consistency error bounds are crucial to establish error estimates for numerical schemes based on non-conforming complexes. The following proposition shows that, if the top sequence satisfies an adjoint consistency estimate, then so does the bottom sequence. Contrary to the previous results, which are mostly useful for local complexes (that is, with spaces restricted to one mesh element), adjoint consistency is really meaningful only for global complexes (that is, with spaces on the entire mesh). We therefore assume in what follows that the complex $(\Xg{i},\dg{i})_i$ is built by patching local complexes $(\Xg{i,T},\dg{i,T})_i$ together; the global inner products are also obtained by summing local inner products, and all local quantities (norms, potential reconstructions, polynomial approximation errors, etc.) are indicated by an index $T$.

\begin{proposition}[Adjoint consistency]\label{prop:adjoint.consistency}
  Assume that \ref{P:continuity}, \ref{P:consistency}, \ref{P:continuity.interpolator} hold for the local complexes, and that \ref{P:E.R.cochain} holds for the global complex. Consider an unbounded operator $D_i^*:\Lg{i+1}\to\Lg{i}$ with domain $\dom(D_i^*)$. Let $u\in\Hg{i+1}\cap\dom(D_i^*)$ and define the adjoint consistency error of the top and bottom sequences, respectively, by
  \begin{align*}
    \mathcal{E}_i(u)\coloneq{}& \sup_{x\in\Xg{i}\backslash\{0\}}\frac{(\Ig{i+1}u,\dg{i}x)_{\Xg{},i+1}+\displaystyle\sum_T (D_i^*u,P_{i,T}x)_{\Lg{},i,T}}{\norm[\Xg{},i]{x}+\norm[\Xg{},i+1]{\dg{i}x}},\\
    \s{\cal E}_i(u)\coloneq{}& \max_{\s{x}\in\tXg{i}\backslash\{0\}}\frac{(\tIg{i+1}u,\tdg{i}\s{x})_{\tXg{},i+1}+\displaystyle\sum_T (D_i^*u,\s{P}_{i,T}\s{x})_{\Lg{},i,T}}{\norm[\tXg{},i]{\s{x}}+\norm[\tXg{},i+1]{\tdg{i}\s{x}}}.
  \end{align*} 
  Then, the following estimate holds:
  \[
  \s{\cal E}_i(u)\le \mathcal E_i(u) +
    \left(\sum_T \left(\norm{\Rg{i+1,T}}+1\right)^2\norm{\Ig{i+1,T}}^2\omega_{\Hg{},\Poly{},i+1,T}(u)^2\right)^{1/2}.
  \]
\end{proposition}

\begin{proof}
  We first notice that, for all $u_p\in\Poly[T]{k_{i+1}}$, $\Eg{i+1,T}\tIg{i+1,T}u_p-\Ig{i+1,T}u_p=0$ by \ref{P:consistency} for the local complex corresponding to $T$ with $i+1$ instead of $i$.
  Hence, inserting this quantity into the left-hand side norm below, using a triangle inequality, invoking the isometry property of $\Eg{i+1,T}$ and \eqref{eq:def.tIg} to write $\norm[\Xg{},i+1,T]{\Eg{i+1,T}\tIg{i+1,T}\,{\cdot}}=\norm[\tXg{},i+1,T]{\Rg{i+1,T}\Ig{i+1,T}\,{\cdot}}\le \norm{\Rg{i+1,T}}\norm{\Ig{i+1,T}}\norm[\Hg{},i+1,T]{{\cdot}}$,
  and taking the infimum over $u_p$,
  we have
  \begin{align}\label{eq:est.ERImI}
      \norm[\Xg{},i+1,T]{\Eg{i+1,T}\tIg{i+1,T}u-\Ig{i+1,T}u}
      &\le
      \inf_{u_p\in\Poly[T]{k_{i+1}}}\left(\norm[\Xg{},i+1,T]{\Eg{i+1,T}\tIg{i+1,T}(u-u_p)}
      +\norm[\Xg{},i+1,T]{\Ig{i+1,T}(u_p-u)}\right)
      \nonumber\\
      &\le
      \left(\norm{\Rg{i+1,T}}+1\right)
      \norm{\Ig{i+1,T}}\omega_{\Hg{},\Poly{},i+1,T}(u).
  \end{align}
  Recalling \eqref{eq:def.ps.tXg}, we then write, for $\s{x}\in\tXg{i}$,
  \begin{alignat*}{3}
    (\tIg{i+1}u,\tdg{i}\s{x})_{\tXg{},i+1}={}&(\Eg{i+1}\tIg{i+1}u,\Eg{i+1}\tdg{i}\s{x})_{\Xg{},i+1}\\
    ={}& (\Ig{i+1}u,\dg{i}\Eg{i}\s{x})_{\Xg{},i+1}+(\Eg{i+1}\tIg{i+1}u-\Ig{i+1}u,\dg{i}\Eg{i}\s{x})_{\Xg{},i+1}.
  \end{alignat*}
  where the second line is obtained inserting $\pm\Ig{i+1}u$ and using the cochain map property \ref{P:E.R.cochain} of the extension to write $\Eg{i+1}\tdg{i}\s{x}=\dg{i}\Eg{i}\s{x}$.
  Invoking then \eqref{eq:est.ERImI} for each $T$, the definition of $\mathcal{E}_i(u)$ with $x = \Eg{i}\s{x}$, and $P_{i,T}x=P_{i,T}\Eg{i,T}\s{x}=\s{P}_{i,T}\s{x}$, together with the fact that the inner products on the complex spaces are obtained summing the local inner products for each $T$, we infer
  \begin{alignat*}{3}
    (\tIg{i+1}u,\tdg{i}\s{x})_{\tXg{},i+1}+\sum_T{}&(D_i^*u,\s{P}_{i,T}\s{x})_{\Lg{},i,T}\le
    \mathcal{E}_i(u)\left(\norm[\Xg{},i]{\Eg{i}\s{x}}+\norm[\Xg{},i+1]{\dg{i}\Eg{i}\s{x}}\right)\\
    &+
    \left(\sum_T \left(\norm{\Rg{i+1,T}}+1\right)^2\norm{\Ig{i+1,T}}^2\omega_{\Hg{},\Poly{},i+1,T}(u)^2\right)^{1/2}
    \norm[\Xg{},i+1]{\dg{i}\Eg{i}\s{x}}.
  \end{alignat*}
  The proof is completed by noticing that, by isometry and cochain map property \ref{P:E.R.cochain} of the extensions, $\norm[\Xg{},i]{\Eg{i}\s{x}}+\norm[\Xg{},i+1]{\dg{i}\Eg{i}\s{x}}=\norm[\tXg{},i]{\s{x}}+\norm[\tXg{},i+1]{\tdg{i}\s{x}}$.
\end{proof}

%------------------------------------------------------------------------------%

\section{Setting}\label{sec:setting}

\subsection{Domain and mesh}\label{sec:setting:domain.mesh}

Let $\Omega\subset\Real^3$ denote a connected polyhedral domain.
We consider, as in \cite{Di-Pietro.Droniou:21*2}, a polyhedral mesh $\Mh\coloneq\Th\cup\Fh\cup\Eh\cup\Vh$, where $\Th$ gathers the elements, $\Fh$ the faces, $\Eh$ the edges, and $\Vh$ the vertices.
For all $\sfP\in\Mh$, we denote by $h_\sfP$ its diameter and set $h\coloneq\max_{T\in\Th}h_T$.
For each face $F\in\Fh$, we fix a unit normal $\normal_F$ to $F$ and, for each edge $E\in\Eh$, a unit tangent $\tangent_E$. For $T\in\Th$, $\FT$ gathers the faces on the boundary $\partial T$ of $T$ and $\ET$ the edges in $\partial T$; if $F\in\Fh$, $\EF$ is the set of edges contained in the boundary $\partial F$ of $F$.
For $F\in\FT$, $\omega_{TF}\in\{-1,+1\}$ is such that $\omega_{TF}\normal_F$ is the outer normal on $F$ to $T$.

Each face $F\in\Fh$ is oriented counter-clockwise with respect to $\normal_F$ and, for $E\in\EF$, we let $\omega_{FE}\in\{-1,+1\}$ be such that $\omega_{FE}=+1$ if $\tangent_E$ points along the boundary $\partial F$ of $F$ in the clockwise sense, and $\omega_{FE}=-1$ otherwise; we also denote by $\normal_{FE}$ the unit normal vector to $E$, in the plane spanned by $F$, such that $(\tangent_E,\normal_{FE},\normal_F)$ is a right-handed system of coordinate (so that, in particular, $\omega_{FE}\normal_{FE}$ points outside $F$).
We denote by $\GRAD_F$ and $\DIV_F$ the tangent gradient and divergence operators acting on smooth enough functions.
Moreover, for any $r:F\to\Real$ and $\bvec{z}:F\to\Real^2$ smooth enough, we let
$\VROT_F r\coloneq (\GRAD_F r)^\perp$ and
$\ROT_F\bvec{z}=\DIV_F(\bvec{z}^\perp)$,
with $\perp$ denoting the rotation of angle $-\frac\pi2$ in the oriented tangent space to $F$.

We further assume that $(\Th,\Fh)$ belongs to a regular mesh sequence as per \cite[Definition 1.9]{Di-Pietro.Droniou:20}, with mesh regularity parameter $\varrho>0$.
This assumption ensures the existence, for each $\sfP\in\Th\cup\Fh\cup\Eh$, of a point $\bvec{x}_{\sfP}\in\sfP$ such that the ball centered at $\bvec{x}_\sfP$ and of radius $\varrho h_\sfP$ is contained in $\sfP$.

\subsection{Polynomial spaces and $L^2$-orthogonal projectors}

For any $\sfP\in\Mh$ and an integer $\ell\ge 0$, we denote by $\Poly{\ell}(\sfP)$ the space spanned by the restriction to $\sfP$ of polynomial functions of the space variables. The subspace of $\Poly{\ell}(\sfP)$ spanned by polynomials with zero average on $\sfP$ is denoted by $\Poly{0,\ell}(\sfP)$.
Letting $\vPoly{\ell}(\sfP)\coloneq\Poly{\ell}(\sfP)^n$ for $\sfP\in\Th$ and $n = 3$ or $\sfP\in\Fh$ and $n = 2$, it holds:
For all $F\in\Fh$,
\begin{alignat}{4}\label{eq:vPoly=Goly+cGoly.F}
    \vPoly{\ell}(F)
    &= \Goly{\ell}(F) \oplus \cGoly{\ell}(F)
    &\enspace&
    \text{%
      with $\Goly{\ell}(F)\coloneq\GRAD_F\Poly{\ell+1}(F)$
      and $\cGoly{\ell}(F)\coloneq(\bvec{x}-\bvec{x}_F)^\perp\Poly{\ell-1}(F)$%
    }
    \\ \label{eq:vPoly=Roly+cRoly.F}
    &= \Roly{\ell}(F) \oplus \cRoly{\ell}(F)
    &\enspace&
    \text{%
      with $\Roly{\ell}(F)\coloneq\VROT_F\Poly{\ell+1}(F)$
      and $\cRoly{\ell}(F)\coloneq(\bvec{x}-\bvec{x}_F)\Poly{\ell-1}(F)$
    }
\end{alignat}
and, for all $T\in\Th$,
\begin{alignat}{4}\label{eq:vPoly=Goly+cGoly}
  \vPoly{\ell}(T)
  &= \Goly{\ell}(T) \oplus \cGoly{\ell}(T)
  &\enspace&
  \text{%
    with $\Goly{\ell}(T)\coloneq\GRAD\Poly{\ell+1}(T)$
    and $\cGoly{\ell}(T)\coloneq(\bvec{x}-\bvec{x}_T)\times \vPoly{\ell-1}(T)$%
  } 
  \\ \label{eq:vPoly=Roly+cRoly}
  &= \Roly{\ell}(T) \oplus \cRoly{\ell}(T)
  &\enspace&
  \text{%
    with $\Roly{\ell}(T)\coloneq\CURL\Poly{\ell+1}(T)$
    and $\cRoly{\ell}(T)\coloneq(\bvec{x}-\bvec{x}_T)\Poly{\ell-1}(T)$.%
  }
\end{alignat}
  We extend the above notations to negative exponents $\ell$ by setting all the spaces appearing in the decompositions equal to the trivial vector space $\{\bvec{0}\}$.
Given a polynomial (sub)space $\mathcal{X}^\ell(\sfP)$ on $\sfP\in\Mh$, the corresponding $L^2$-orthogonal projector is denoted by $\pi_{\mathcal{X},\sfP}^\ell$.
Boldface font will be used when the elements of $\mathcal{X}^\ell(\sfP)$ are vector-valued, and $\Xcproj{\ell}{\sfP}$ denotes the $L^2$-orthogonal projector on $\cvec{X}^{{\rm c},\ell}(\sfP)$.
  The spaces in the above decompositions are hierarchical, i.e., for all $\cvec{X}\in\{\cvec{G},\cvec{R}\}$ and all $\sfP\in\Th\cup\Fh$, $\cvec{X}^{{\rm c},\ell-1}(\sfP)\subset\cvec{X}^{{\rm c},\ell}(\sfP)$.
  In what follows, we will frequently use this property to write, for all integers $\ell\le k-1$ and $\sfP\in\Th\cup\Fh$,
  \begin{equation}\label{eq:Xcproj.ell+1.k}
    \Xcproj{\ell+1}{\sfP}\Xcproj{k}{\sfP} = \Xcproj{\ell+1}{\sfP}.
  \end{equation}

%------------------------------------------------------------------------------%

\section{DDR complex}\label{sec:ddr}

In this section we briefly recall the DDR complex and refer to \cite[Section 3]{Di-Pietro.Droniou:21*2} for further details.
From this point on, we fix an integer $k\ge 0$ corresponding to the polynomial degree of the complex.

\subsection{Spaces and component norms}\label{sec:def.space.ddr}
The DDR counterparts of $H^1(\Omega)$, $\Hcurl{\Omega}$, $\Hdiv{\Omega}$ and $L^2(\Omega)$ are respectively defined as follows:
\begin{multline*}\label{eq:Xgrad.h}
  \Xgrad{h}\coloneq\Big\{
  \underline{q}_h=\big((q_T)_{T\in\Th},(q_F)_{F\in\Fh}, q_{\Eh}\big)\st
    \\
    \text{$q_T\in \Poly{k-1}(T)$ for all $T\in\Th$,
      $q_F\in\Poly{k-1}(F)$ for all $F\in\Fh$,
      and $q_{\Eh}\in\Poly[\rm c]{k+1}(\Eh)$}
    \Big\},
\end{multline*}
with $\Poly[\rm c]{k+1}(\Eh)$ denoting the space of functions that are continuous on the mesh edge skeleton and the restriction of which to each $E\in\Eh$ belongs to $\Poly{k+1}(E)$,
\begin{equation*}\label{eq:Xcurl.h}
  \Xcurl{h}\coloneq\Big\{
  \begin{aligned}[t]
    \uvec{v}_h
    &=\big(
    (\bvec{v}_{\cvec{R},T},\bvec{v}_{\cvec{R},T}^\compl)_{T\in\Th},(\bvec{v}_{\cvec{R},F},\bvec{v}_{\cvec{R},F}^\compl)_{F\in\Fh}, (v_E)_{E\in\Eh}
    \big)\st
    \\
    &\qquad\text{$\bvec{v}_{\cvec{R},T}\in\Roly{k-1}(T)$ and $\bvec{v}_{\cvec{R},T}^\compl\in\cRoly{k}(T)$ for all $T\in\Th$,}
    \\
    &\qquad\text{$\bvec{v}_{\cvec{R},F}\in\Roly{k-1}(F)$ and $\bvec{v}_{\cvec{R},F}^\compl\in\cRoly{k}(F)$ for all $F\in\Fh$,}
    \\
    &\qquad\text{and $v_E\in\Poly{k}(E)$ for all $E\in\Eh$}\Big\},
  \end{aligned}
\end{equation*}
\begin{equation*}%\label{eq:Xdiv.h}
  \Xdiv{h}\coloneq\Big\{
  \begin{aligned}[t]
    \uvec{w}_h
    &=\big((\bvec{w}_{\cvec{G},T},\bvec{w}_{\cvec{G},T}^\compl)_{T\in\Th}, (w_F)_{F\in\Fh}\big)\st
    \\
    &\qquad\text{$\bvec{w}_{\cvec{G},T}\in\Goly{k-1}(T)$ and $\bvec{w}_{\cvec{G},T}^\compl\in\cGoly{k}(T)$ for all $T\in\Th$,}
    \\
    &\qquad\text{and $w_F\in\Poly{k}(F)$ for all $F\in\Fh$}
    \Big\},
  \end{aligned}
\end{equation*}
and
\[
\Poly{k}(\Th)\coloneq\left\{
q_h\in L^2(\Omega)\st\text{$(q_h)_{|T}\in\Poly{k}(T)$ for all $T\in\Th$}
\right\}.
\]
Throughout this paper, we use underlines to denote vectors whose components are polynomials, and boldface fonts when some of these components are vector-valued. The same convention applies to the discrete operators defined in Section \ref{sec:ddr.complex} below.

The corresponding interpolators are denoted by $\Igrad{h}$, $\Icurl{h}$, $\Idiv{h}$, and $\lproj{k}{h}$ (using the same underlines/boldface convention as above).
The interpolator on $\Xgrad{h}$ is defined by \eqref{eq:Igradh} below, while the remaining interpolators correspond to the $L^2$-projections on the polynomial spaces appearing in the definitions of $\Xcurl{h}$, $\Xdiv{h}$, and $\Poly{k}(\Th)$, respectively; see \cite{Di-Pietro.Droniou:21*2} for details.
The restriction of each of these spaces to a mesh entity $\sfP\in\Mh$ that appears in its definition gathers the polynomial components on $\sfP$ and the geometric entities on its boundary, and is denoted replacing the subscript $h$ by $\sfP$ (e.g., $\Xgrad{T}$ or $\Xcurl{F}$).
A similar convention is used for the vectors in such restrictions (e.g., $\uvec{v}_F=(\bvec{v}_{\cvec{R},F},\bvec{v}_{\cvec{R},F}^\compl,(v_E)_{E\in\EF})\in\Xcurl{F}$) as well as for the interpolators (e.g., $\Igrad{T}$).

From this point on, to alleviate the notation, for all $\sfP\in\{\Omega\}\cup\Th\cup\Fh\cup\Eh$, we denote by $\norm[\sfP]{{\cdot}}$ the standard $L^2(\sfP)$-norm.
The same notation is used for the norm of $L^2(\sfP)^n$ (with $n$ denoting the dimension of $\sfP$), any ambiguity being removed by the scalar or vector nature of the argument.
In the analysis, we will additionally make use of the following local component $L^2$-norms $\tnorm[\bullet,\sfP]{{\cdot}}:\Xbullet{\sfP}\to\Real$, for $\bullet\in\{\GRAD,\CURL,\DIV\}$ and $\sfP\in\Th\cup\Fh$ (only $\sfP\in\Fh$ if $\bullet=\DIV$):
\begin{equation}\label{eq:tnorm}
  \begin{alignedat}{4}
    %% \label{eq:tnorm.grad.F}
    \tnorm[\GRAD,F]{\underline{q}_F}\coloneq{}&\bigg(
    \norm[F]{q_F}^2 + \sum_{E\in\EF} h_E\norm[E]{q_E}^2
    \bigg)^{\nicefrac12}
    &\quad&\forall F\in\Fh\,,\;\forall\underline{q}_F\in\Xgrad{F},\\
    %% \label{eq:tnorm.grad.T}
    \tnorm[\GRAD,T]{\underline{q}_T}\coloneq{}&\bigg(
    \norm[T]{q_T}^2 + \sum_{F\in\FT} h_F\tnorm[\GRAD,F]{\underline{q}_F}^2
    \bigg)^{\nicefrac12}&\quad&\forall T\in\Th,\; \forall\underline{q}_T\in\Xgrad{T},\\
    %% \label{eq:tnorm.curl.F}
    \tnorm[\CURL,F]{\uvec{v}_F}\coloneq{}&\bigg(
    \norm[F]{\bvec{v}_{\cvec{R},F}}^2
    + \norm[F]{\bvec{v}_{\cvec{R},F}^\compl}^2
    + \sum_{E\in\EF} h_E\norm[E]{v_E}^2
    \bigg)^{\nicefrac12}&\quad&\forall F\in\Fh\,,\;\forall\uvec{v}_F\in\Xcurl{F},\\
    %% \label{eq:tnorm.curl.T}
    \tnorm[\CURL,T]{\uvec{v}_F}\coloneq{}&\bigg(
    \norm[T]{\bvec{v}_{\cvec{R},T}}^2
    + \norm[T]{\bvec{v}_{\cvec{R},T}^\compl}^2
    + \sum_{F\in\FT} h_F\tnorm[\CURL,F]{\uvec{v}_F}^2
    \bigg)^{\nicefrac12}&\quad&\forall T\in\Th\,,\;\forall\uvec{v}_T\in\Xcurl{T},\\
    %%\label{eq:tnorm.div.T}
    \tnorm[\DIV,T]{\uvec{w}_T}\coloneq{}&\bigg(
    \norm[T]{\bvec{w}_{\cvec{G},T}}^2
    + \norm[T]{\bvec{w}_{\cvec{G},T}^\compl}^2
    + \sum_{F\in\FT} h_F \norm[F]{w_F}^2
    \bigg)^{\nicefrac12}&\quad&\forall T\in\Th,\;\forall\uvec{w}_T\in\Xdiv{T}.
  \end{alignedat}
\end{equation}
For $\bullet\in\{\GRAD,\CURL,\DIV\}$, a global component $L^2$-norm is obtained setting $\tnorm[\bullet,h]{{\cdot}}\coloneq\left(\sum_{T\in\Th}\tnorm[\bullet,T]{{\cdot}}^2\right)^{\nicefrac12}$.

\subsection{Discrete vector calculus operators and potentials}

The discrete operators and potentials are defined mimicking integration-by-parts formulas, replacing internal/boundary values with those provided by the unknowns in each space.

\subsubsection{Gradient}

For any $E\in\Eh$, the edge gradient $\GE:\Xgrad{E}\to\Poly{k}(E)$ is such that, for all $q_E\in\Xgrad{E}$, $\GE q_E = q_E'$, where $q_E$ denotes the restriction to $E$ of $q_{\Eh}$ and the derivative is taken is the direction $\tangent_E$.

For any $F\in\Fh$, the face gradient $\cGF:\Xgrad{F}\to\vPoly{k}(F)$ and the scalar trace $\trF:\Xgrad{F}\to\Poly{k+1}(F)$ are such that, for all $\underline{q}_F\in\Xgrad{F}$,
\begin{alignat}{4}\label{eq:cGF}
  \int_F\cGF\underline{q}_F\cdot\bvec{v}_F
  &= -\int_F q_F\DIV_F\bvec{v}_F
  + \sum_{E\in\EF}\omega_{FE}\int_E q_E~(\bvec{v}_F\cdot\normal_{FE})
  &\quad&\forall\bvec{v}_F\in\vPoly{k}(F),
  \\ \label{eq:trF}
  \int_F\trF\underline{q}_F\DIV_F\bvec{v}_F
  &= -\int_F\cGF\underline{q}_F\cdot\bvec{v}_F
  + \sum_{E\in\EF}\omega_{FE}\int_E \trE\underline{q}_E~(\bvec{v}_F\cdot\normal_{FE})
  &\quad&\forall\bvec{v}_F\in\cRoly{k+2}(F).
\end{alignat}
Note that \eqref{eq:trF} fully determines $\trF\underline{q}_F\in\Poly{k+1}(F)$ owing to the fact that $\DIV:\cRoly{k+2}(F)\to\Poly{k+1}(F)$ is an isomorphism.

Similarly, for all $T\in\Th$, the element gradient $\cGT:\Xgrad{T}\to\vPoly{k}(T)$ and the scalar potential $\Pgrad:\Xgrad{T}\to\Poly{k+1}(T)$ are defined such that, for all $\underline{q}_T\in\Xgrad{T}$,
\begin{alignat}{4} \label{eq:cGT}
  \int_T\cGT\underline{q}_T\cdot\bvec{v}_T
  &= -\int_T q_T\DIV\bvec{v}_T
  + \sum_{F\in\FT}\omega_{TF}\int_F\trF\underline{q}_F~(\bvec{v}_T\cdot\normal_F)
  &\qquad&\forall\bvec{v}_T\in\vPoly{k}(T),
  \\ \nonumber
  \int_T\Pgrad\underline{q}_T\DIV\bvec{v}_T
  &= -\int_T\cGT\underline{q}_T\cdot\bvec{v}_T
  + \sum_{F\in\FT}\omega_{TF}\int_F\trF\underline{q}_F~(\bvec{v}_T\cdot\normal_F)
  &\qquad&\forall\bvec{v}_T\in\cRoly{k+2}(T).
\end{alignat}
  For future use, we notice that, using discrete inverse and trace inequalities and, for mesh elements, the continuity of $\trF$ stated in \cite[Proposition 6]{Di-Pietro.Droniou:21*2}, it holds, for all $\sfP\in\Th\cup\Fh$,
  \begin{equation}\label{eq:GT.boundedness}
    \norm[\sfP]{\cGP\underline{q}_\sfP}\lesssim h_\sfP^{-1}\tnorm[\GRAD,\sfP]{\underline{q}_\sfP}
    \qquad\forall\underline{q}_\sfP\in\Xgrad{\sfP}.
  \end{equation}

\subsubsection{Curl}

For all $F\in\Fh$, the face curl $\CF:\Xcurl{F}\to\Poly{k}(F)$ and the corresponding vector potential (which can be interpreted as a tangential component) $\trFt:\Xcurl{F}\to\vPoly{k}(F)$ are such that, for all $\uvec{v}_F\in\Xcurl{F}$,
\begin{equation}\label{eq:CF}
  \int_F\CF\uvec{v}_F~r_F
  = \int_F\bvec{v}_{\cvec{R},F}\cdot\VROT_F r_F
  - \sum_{E\in\EF}\omega_{FE}\int_E v_E~r_F
  \qquad\forall r_F\in\Poly{k}(F)
\end{equation}
and, for all $(r_F,\bvec{w}_F)\in\Poly{0,k+1}(F)\times\cRoly{k}(F)$,
\begin{equation*}
  \int_F\trFt\uvec{v}_F\cdot(\VROT_F r_F + \bvec{w}_F)
  = \int_F\CF\uvec{v}_F~r_F
  + \sum_{E\in\EF}\omega_{FE}\int_E v_E~r_F
  + \int_F\bvec{v}_{\cvec{R},F}^\compl\cdot\bvec{w}_F.
\end{equation*}
We note that testing $\trFt\uvec{v}_F$ against $\VROT_F r_F$ fixes $\Rproj{k}{F}\trFt\uvec{v}_F$, while testing
against $\bvec{w}_F$ provides the remaining projection $\Rcproj{k}{F}\trFt\uvec{v}_F$ and therefore fixes $\trFt\uvec{v}_F\in\vPoly{k}(F)$ entirely due to \eqref{eq:vPoly=Goly+cGoly.F}.

For all $T\in\Th$, the element curl $\cCT:\Xcurl{T}\to\vPoly{k}(T)$ and the vector potential $\Pcurl:\Xcurl{T}\to\vPoly{k}(T)$ are defined such that, for all $\uvec{v}_T\in\Xcurl{T}$,
\begin{equation}\label{eq:cCT}
  \int_T\cCT\uvec{v}_T\cdot\bvec{w}_T
  = \int_T\bvec{v}_{\cvec{R},T}\cdot\CURL\bvec{w}_T
  + \sum_{F\in\FT}\omega_{TF}\int_F\trFt\uvec{v}_F\cdot(\bvec{w}_T\times\normal_F)
  \qquad\forall\bvec{w}_T\in\vPoly{k}(T)
\end{equation}
and, for all $(\bvec{w}_T,\bvec{z}_T)\in\cGoly{k+1}(T)\times\cRoly{k}(T)$,
\begin{equation*}
  \int_T\Pcurl\uvec{v}_T\cdot(\CURL\bvec{w}_T + \bvec{z}_T)
  = \int_T\cCT\uvec{v}_T\cdot\bvec{w}_T
  - \sum_{F\in\FT}\omega_{TF}\int_F\trFt\uvec{v}_F\cdot(\bvec{w}_T\times\normal_F)
  + \int_T\bvec{v}_{\cvec{R},T}^\compl\cdot\bvec{z}_T.
\end{equation*}

\subsubsection{Divergence}

For all $T\in\Th$, the element divergence $\DT:\Xdiv{T}\to\Poly{k}(T)$ and vector potential $\Pdiv:\Xdiv{T}\to\vPoly{k}(T)$ are defined by:
For all $\uvec{w}_T\in\Xdiv{T}$,
\begin{equation}\label{eq:DT}
  \int_T\DT\uvec{w}_T~q_T
  = -\int_T\bvec{w}_{\cvec{G},T}\cdot\GRAD q_T
  + \sum_{F\in\FT}\omega_{TF}\int_F w_F~q_T
  \qquad\forall q_T\in\Poly{k}(T)
\end{equation}
and, for all $(r_T,\bvec{z}_T)\in\Poly{0,k+1}(T)\times\cGoly{k}(T)$,
\begin{equation*}%% \label{eq:Pdiv}
  \int_T\Pdiv\uvec{w}_T\cdot(\GRAD r_T + \bvec{z}_T)
  = -\int_T\DT\uvec{w}_T~r_T
  + \sum_{F\in\FT}\omega_{TF}\int_Fw_F~r_T
  + \int_T\bvec{w}_{\cvec{G},T}^\compl\cdot\bvec{z}_T.
\end{equation*}

\subsection{DDR complex}\label{sec:ddr.complex}

The DDR complex reads:
\begin{equation*}%% \label{eq:ddr.complex}
  \begin{tikzcd}
    \Real\arrow{r}{\Igrad{h}} & \Xgrad{h}\arrow{r}{\uGh} & \Xcurl{h}\arrow{r}{\uCh} & \Xdiv{h}\arrow{r}{\Dh} & \Poly{k}(\Th)\arrow{r}{0} & \{0\},
  \end{tikzcd}
\end{equation*}
where the interpolator on $\Xgrad{h}$ is such that, for all $q:\Omega\to\Real$ smooth enough,
\begin{equation}\label{eq:Igradh}
  \begin{gathered}
    \text{%
      $\Igrad{h} q\coloneq\big(
      (\lproj{k-1}{T} q_{|T})_{T\in\Th},
      (\lproj{k-1}{F} q_{|F})_{F\in\Fh},
      q_{\Eh}
      \big)$,
      with $q_{\Eh}\in\Poly[\rm c]{k+1}(\Eh)$ such that%
    }
    \\
    \text{%
      $q_{\Eh}(\bvec{x}_V) = q(\bvec{x}_V)$ for all $V\in\Vh$ and
      $\lproj{k-1}{E} q_E = \lproj{k-1}{E} q$ for all $E\in\Eh$,
    }
  \end{gathered}
\end{equation}
and the global discrete gradient, curl, and divergence operators are obtained projecting the local face and element operators onto the appropriate spaces and patching together these projections: for all $(\underline{q}_h,\uvec{v}_h,\uvec{w}_h)\in\Xgrad{h}\times\Xcurl{h}\times\Xdiv{h}$,
\begin{equation*}%% \label{eq:uGh}
  \begin{aligned}
    \uGh\underline{q}_h
    &\coloneq
  \big(
  ( \Rproj{k-1}{T}\cGT\underline{q}_T,\Rcproj{k}{T}\cGT\underline{q}_T )_{T\in\Th},
  ( \Rproj{k-1}{F}\cGF\underline{q}_F,\Rcproj{k}{F}\cGF\underline{q}_F )_{F\in\Fh},
  ( \GE q_E )_{E\in\Eh}
  \big),
  \\
  \uCh\uvec{v}_h
  &\coloneq\big(
  ( \Gproj{k-1}{T}\cCT\uvec{v}_T,\Gcproj{k}{T}\cCT\uvec{v}_T )_{T\in\Th},
  ( \CF\uvec{v}_F )_{F\in\Fh}
  \big),
  \\
  (\Dh\uvec{w}_h)_{|T}
  &\coloneq\DT\uvec{w}_T\qquad\forall T\in\Th.
  \end{aligned}
\end{equation*}

\subsection{Discrete $L^2$-products and norms}\label{sec:ddr:L2-products.norms}

For $\bullet\in\{\GRAD,\CURL,\DIV\}$, we define a discrete $L^2$-product on $\Xbullet{h}$ setting, for all $\underline{x}_h,\,\underline{y}_h\in\Xbullet{h}$,
\[
(\underline{x}_h,\underline{y}_h)_{\bullet,h}
\coloneq\sum_{T\in\Th}(\underline{x}_T, \underline{y}_T)_{\bullet,T},
\]
with local contribution such that, denoting by $\Pbullet$ the potential on $\Xbullet{T}$ (with $l=k+1$ if $\bullet=\GRAD$, $l=k$ otherwise),
\[
(\underline{x}_T, \underline{y}_T)_{\bullet,T},
\coloneq\int_T\Pbullet\underline{x}_T\cdot\Pbullet\underline{y}_T
+ s_{\bullet,T}(\underline{x}_T, \underline{y}_T).
\]
Above, $s_{\bullet,T}:\Xbullet{T}\times\Xbullet{T}\to\Real$ is a local stabilisation bilinear form, the precise expression of which can be found in \cite[Section 4.4]{Di-Pietro.Droniou:21*2} and whose role is to control the differences between the traces of the potentials $\Pbullet$ and the degrees of freedom on the boundary of $T$.
We will also need, in what follows, the discrete global and local $L^2$-norms induced by the above $L^2$-products, that will be denoted by $\norm[\bullet,h]{{\cdot}}$ and $\norm[\bullet,T]{{\cdot}}$, respectively.

%------------------------------------------------------------------------------%

\section{Serendipity DDR complex}\label{sec:serendipity.ddr}

As demonstrated by the example of the DDR complex, polytopal discrete complexes have degrees of freedom (DOFs) attached to mesh entities of various dimensions: elements, faces, edges, vertices. Their specific choice is related to the degree of their polynomial consistency. The principle for constructing a serendipity version is to notice that, on a polygonal face (resp.~polyhedral element), the knowledge of a polynomial on the boundary alone can often provide sufficient information to reconstruct it inside the face (resp.~element).

Following the notations of Section \ref{sec:blueprint}, this principle guides the construction of the serendipity spaces $\tXg{i}$, in which face or element internal DOFs that are not required to reconstruct polynomials of the desired degree are removed. The reduction map $\Rg{}$ consists in removing these DOFs, while the extension map $\Eg{}$ re-creates these DOFs by fixing them to those of the polynomial reconstructed from the boundary and remaining internal DOFs; this way, DOFs representing in $\Xg{i}$ (through $\Ig{i}$) polynomials in $\Poly{k_i}$ are exactly recovered, ensuring the polynomial consistency expressed by \ref{P:consistency}, which is a driving property in the process.

This design, however, has a caveat, as the DOFs removed on a lower-dimensional entity (e.g., a face) can impact some components of the discrete differential operators on higher-dimensional entities (e.g., an element); in particular, to ensure the cochain map property \ref{P:E.R.cochain}, the values of the reduction operator attached to higher-dimensional entities might have to be adjusted to account for the removal of DOFs on lower-dimensional entities. The design of the reduction operator is thus initially inspired by the consistency property, but must then be tuned to fulfill Assumption \ref{ass:cohomology}, which relates to the choice of serendipity spaces in each vertical section of Diagram \eqref{eq:ser.diagram}.
\medskip

The rest of this section shows how to apply these principles to the DDR sequence presented in \cite{Di-Pietro.Droniou:21*2,Beirao-da-Veiga.Dassi.ea:22}. In this context, focusing on the local sequence on an element $T$ of the mesh, Diagram \eqref{eq:ser.diagram} becomes Diagram \eqref{eq:ser.ddr.seq} below.
The operators that appear in the construction of the SDDR complex are summarized in Table \ref{tab:serendipity.ddr}.

\begin{equation}\label{eq:ser.ddr.seq}
  \begin{tikzcd}
    & 
    \arrow{d}{\Igrad{T}}\Poly{k+1}(T)  &[2em] \arrow{d}{\Icurl{T}}\vPoly{k}(T)
    &[2em] \arrow{d}{\Idiv{T}}\vPoly{k}(T) & \arrow{d}{\Id} \Poly{k}(T)\\[2em]
    \arrow[<->]{d}{\Id}\Real\arrow{r}{\Igrad{T}} & \Xgrad{T}
    \arrow{r}{\uGT}\arrow[bend left, dashed]{d}{\Rgrad{T}} &[2em] \Xcurl{T}
    \arrow{r}{\uCT}\arrow[bend left, dashed]{d}{\Rcurl{T}} &[2em] \Xdiv{T} 
    \arrow{r}{\DT}\arrow[<->]{d}{\Id} &[2em] \Poly{k}(T) 
    \arrow[<->]{d}{\Id}\arrow{r} & 0
    \\[2em]
    \Real\arrow{r}{\sIgrad{T}} & \sXgrad{T}
    \arrow{r}{\suGT}\arrow[bend left]{u}{\Egrad{T}} &[2em] \sXcurl{T}
    \arrow{r}{\suCT}\arrow[bend left]{u}{\Ecurl{T}} &[2em] \Xdiv{T}
    \arrow{r}{\DT} &[2em] \Poly{k}(T)
    \arrow{r} & 0.
  \end{tikzcd}
\end{equation}

\begin{remark}[No serendipity for {$\Xdiv{T}$} spaces]\label{rem:no.seren.Xdiv}
  Notice that no serendipity is performed on the space $\Xdiv{T}$.
  The strong objection to reducing the DOFs in that space is discussed in Section \ref{sec:Xdiv.no.serendipity}.
\end{remark}

The global complexes are obtained collecting the local components of the spaces and operators and enforcing their single-valuedness on vertices, edges, and faces internal to the domain.
For the global complex, the polynomial consistency property is not relevant.
A comparison in terms of DOFs count for tetrahedral and hexahedral elements among SDDR, DDR, and (non-serendipity) finite elements is provided in Table \ref{tab:dof.count}.

\begin{table}\renewcommand{\arraystretch}{1.2}
  \begin{tabular}{ccc}
    \toprule
    \textbf{Notation} & \textbf{Name} & \textbf{Definition} \\
    \midrule
    $\SG{F}$ & Gradient serendipity operator on faces & Problem \eqref{eq:serendipity.problem:F} with $\mathcal{L}_F$ given by \eqref{eq:LF.grad} \\
    $\SC{F}$ & Curl serendipity operator on faces & Problem \eqref{eq:serendipity.problem:F} with $\mathcal{L}_F$ given by \eqref{eq:LF.curl} \\
    $\Egrad{F}$ & Gradient extension on faces & Eq. \eqref{eq:Egrad.F} \\
    $\Ecurl{F}$ & Curl extension on faces & Eq. \eqref{eq:Ecurl.F} \\
    $\Rgrad{F}$ & Gradient reduction on faces & Eq. \eqref{eq:R.grad.F} \\
    $\Rcurl{F}$ & Curl reduction on faces & Eq. \eqref{eq:R.curl.F} \\
    $\SG{T}$ & Gradient serendipity operator on elements & Problem \eqref{eq:serendipity.problem:T} with $\mathcal{L}_T$ given by \eqref{eq:LT.grad} \\
    $\SC{T}$ & Curl serendipity operator on elements & Problem \eqref{eq:serendipity.problem:T} with $\mathcal{L}_T$ given by \eqref{eq:LT.curl} \\
    $\Egrad{T}$ & Gradient extension on elements & Eq. \eqref{eq:Egrad.T} \\
    $\Ecurl{T}$ & Curl extension on elements & Eq. \eqref{eq:Ecurl.T} \\
    $\Rgrad{T}$ & Gradient reduction on elements & Eq. \eqref{eq:R.grad.T} \\
    $\Rcurl{T}$ & Curl reduction on elements & Eq. \eqref{eq:R.curl.T} \\
    \bottomrule
  \end{tabular}  
  \caption{Synoptic table of the operators appearing in the construction of the SDDR complex.\label{tab:serendipity.ddr}}
\end{table}

\begin{table}\centering
  \begin{tabular}{c|ccc|ccc}
    \toprule
    \multirow{2}{*}{Discrete space} & \multicolumn{3}{c|}{Tetrahedron} & \multicolumn{3}{c}{Hexahedron} \\
    & $k=0$ & $k=1$ & $k=2$ & $k=0$ & $k=1$ & $k=2$ \\
    \midrule
    $H^1(T)$ & 4 \sep{} 4 \sep{} 4 & 15 \sep{} 10 \sep{} 10 & 32 \sep{} 20 \sep{} 20 & 8 \sep{} 8 \sep{} 8 & 27 \sep{} 20 \sep{} 27 & 54 \sep{} 32 \sep{} 64 \\
    $\Hcurl{T}$ & 6 \sep{} 6 \sep{} 6 & 28 \sep{} 23 \sep{} 20 & 65 \sep{} 53 \sep{} 45 & 12 \sep{} 12 \sep{} 12 & 46 \sep{} 39 \sep{} 54 & 99 \sep{} 77 \sep{} 144 \\
    $\Hdiv{T}$ & 4 \sep{} 4 \sep{} 4 & 18 \sep{} 18 \sep{} 15 &  44 \sep{} 44 \sep{} 36 & 6 \sep{} 6 \sep{} 6 & 24 \sep{} 24 \sep{} 36 & 56 \sep{} 56 \sep{} 108 \\
    $L^2(T)$ & 1 \sep{} 1 \sep{} 1 & 4 \sep{} 4 \sep{} 4 & 10 \sep{} 10 \sep{} 10 & 1 \sep{} 1 \sep{} 1 & 4 \sep{} 4 \sep{} 8 & 10 \sep{} 10 \sep{} 27 \\
    \bottomrule
  \end{tabular}
  \caption{Dimension of the local spaces in the DDR \sep{} SDDR \sep{} Raviart--Thomas--N\'ed\'elec FEM discrete Hilbert complexes of comparable degree $k\in\{0,1,2\}$ on a tetrahedron ($\eta_T = 4$ and $\eta_F = 3$ for all $F\in\FT$) and a hexahedron ($\eta_T = 6$ and $\eta_F = 4$ for all $F\in\FT$).\label{tab:dof.count}}
\end{table}

In the rest of this paper, as can be seen in Diagram \ref{eq:ser.ddr.seq} and in line with the notations adopted in Section \ref{sec:blueprint}, the objects related to the SDDR complex (spaces, discrete operators, vectors), as well as the reduction maps (having co-domain in serendipity spaces), will be denoted with a hat, to visually distinguish them from objects related to the DDR complex.

\subsection{Boundary selection and notation for inequalities}

As discussed above, designing serendipity versions of the DDR spaces requires to reconstruct polynomials inside elements/faces from boundary values on selected faces/edges.
To ensure a stable reconstruction, the selection of these faces/edges must be done in such a way that they are not close to being aligned. This is stated more precisely in the following assumption.

\begin{assumption}[Boundaries selection for serendipity spaces]\label{assum:choice.detaP}
  For each $\sfP$ element (resp.~face) of the mesh, we select a set $\sbdry{\sfP}$ of $\eta_\sfP\ge 2$ faces (resp.~edges) that are not pairwise aligned and such that $\sfP$ lies entirely on one side of the hyperplane $H_\sfb$ spanned by each $\sfb\in\sbdry{\sfP}$. The scaled distance function to $H_\sfb$ is defined by $\dist_{\sfP\sfb}(\bvec{x})=h_{\sfP}^{-1}(\bvec{x}-\bvec{x}_{\sfb})\cdot\normal_{\sfP\sfb}$, with $\normal_{\sfP\sfb}$ outer normal to $\sfP$ on $\sfb$. Finally, we define a \emph{boundaries selection regularity parameter} $\theta>0$ such that $\dist_{\sfP\sfb}(\bvec{x}_{\sfb'})\ge \theta$ for all $\sfb\not=\sfb'$ in $\sbdry{\sfP}$.
\end{assumption}

From this point on, $a \lesssim b$ (resp., $a \gtrsim b$) means $a\le Cb$ (resp. $a \ge C b$) with $C$ only depending on $\Omega$, $k$, the mesh regularity parameter $\varrho$ (cf. Section \ref{sec:setting:domain.mesh}), and the boundaries selection regularity parameter $\theta$.
We also write $a\simeq b$ as a shorthand for ``$a\lesssim b$ and $b\lesssim a$''.

\begin{remark}[Trivial boundaries selection]
  We can always select two faces (resp.~edges) in each element (resp.~face) such that $\theta\gtrsim \varrho$.
  This minimal selection corresponds to the absence of serendipity, as can be easily verified by noticing that, in this case, \eqref{eq:def.ellF} and \eqref{eq:def.ellT} below give $\ell_\sfP = k - 1$ for $\sfP\in\Th\cup\Fh$.%
\end{remark}

\subsection{Estimates of polynomials from boundary values}

We state in this section some bounds of polynomial functions on polytopes in terms of boundary values and, possibly, reduced projections in the interior of the polytope. Similar estimates have been established in the recent preprint \cite{Beirao-da-Veiga.Dassi.ea:22.1}, but only for 2D convex polygons and with a proof by contradiction that, contrary to the one developed in Appendix \ref{sec:estimates.proofs}, does not allow for a direct estimate of the constants.

\begin{lemma}[Estimate of scalar polynomials from boundary values]\label{lem:reconstruction}
  Let $k\ge 0$ and let Assumption \ref{assum:choice.detaP} hold.
  Let $\sfP\in\Th\cup\Fh$ be a mesh element or face. Then, it holds
  \begin{equation}\label{eq:est.reconstruction}
    \norm[\sfP]{q}\lesssim \norm[\sfP]{\lproj{k+1-\eta_\sfP}{\sfP}q}+h_{\sfP}^{\frac12}\sum_{\sfb\in\sbdry{\sfP}}\norm[\sfb]{q_{|\sfb}}
    \qquad\forall q\in\Poly{k+1}(\sfP).
  \end{equation}
\end{lemma}

\begin{proof}
  See Section \ref{sec:estimates.proofs:reconstruction}.
\end{proof}

\begin{lemma}[Estimate of vector polynomials from tangential boundary values]\label{lem:reconstruction.vpoly}
  Let $k\ge 0$ and let Assumption \ref{assum:choice.detaP} hold.
  Let $\sfP\in\Th\cup\Fh$ be a mesh element or face. Then, denoting by $\bdry{\sfP}$ the set of faces (if $\sfP$ is an element) or edges (if $\sfP$ is a face) of $\sfP$, it holds
  \begin{equation}\label{eq:est.reconstruction.curl}
    \norm[\sfP]{\bvec{v}}\lesssim
    \norm[\sfP]{\Rproj{k-2}{\sfP}\bvec{v}}+\norm[\sfP]{\Rcproj{k+2-\eta_\sfP}{\sfP}\bvec{v}}
    + \left(\sum_{\sfb\in\bdry{\sfP}} h_{\sfP}\norm[\sfb]{\bvec{v}_{{\rm t},\sfb}}^2\right)^{\nicefrac12}
    \qquad\forall\bvec{v}\in\vPoly{k}(\sfP),
  \end{equation}
  where $\bvec{v}_{{\rm t},\sfb}$ is the tangential component of $\bvec{v}$ on $\sfb$ (that is, $\bvec{v}_{{\rm t},\sfb}=\bvec{v}_{|E}\cdot\tangent_E$ if $\sfb=E$ is an edge, $\bvec{v}_{{\rm t},\sfb}=\normal_F\times(\bvec{v}_{|F}\times\normal_F)$ if $\sfb=F$ is a face).
\end{lemma}

\begin{proof}
  See Section \ref{sec:estimates.proofs:reconstruction.vpoly}.
\end{proof}

\subsection{Serendipity, extension, and reduction operators on faces}

\subsubsection{Serendipity problem}

For a face $F\in\Fh$, recalling that $\eta_F$ is defined in Assumption \ref{assum:choice.detaP}, we let
\begin{equation}\label{eq:def.ellF}
\ell_F\coloneq k+1-\eta_F\le k-1.
\end{equation}
Given a linear form $\mathcal{L}_F:\vPoly{k}(F)\times\cRoly{\ell_F+1}(F)\to\Real$, we consider the following problem:
Find $(\bvec{\sigma},\bvec{\lambda})\in\vPoly{k}(F)\times\cRoly{\ell_F+1}(F)$ such that
\begin{equation}\label{eq:serendipity.problem:F}
  \mathcal{A}_F((\bvec{\sigma},\bvec{\lambda}), (\bvec{\tau},\bvec{\mu}))
  = \mathcal{L}_F(\bvec{\tau},\bvec{\mu})\qquad
  \forall (\bvec{\tau},\bvec{\mu})\in\vPoly{k}(F)\times\cRoly{\ell_F+1}(F),
\end{equation}
with bilinear form $\mathcal{A}_F:[\vPoly{k}(F)\times\cRoly{\ell_F+1}(F)]^2\to\Real$ such that
\begin{equation}\label{eq:AF}
  \mathcal{A}_F((\bvec{\upsilon},\bvec{\nu}), (\bvec{\tau},\bvec{\mu}))
  \coloneq
  h_F\sum_{E\in\EF}\int_E(\bvec{\upsilon}\cdot\tangent_E)~(\bvec{\tau}\cdot\tangent_E)
  + h_F^2\int_F\ROT_F\bvec{\upsilon}~\ROT_F\bvec{\tau}
  + \int_F\bvec{\upsilon}\cdot\bvec{\mu} - \int_F\bvec{\tau}\cdot\bvec{\nu}.
\end{equation}

\begin{proposition}[Well-posedness of the serendipity problem on faces]\label{prop:serendipity.well-posedness:F}
  The serendipity problem \eqref{eq:serendipity.problem:F} admits a unique solution which satisfies
  \begin{equation}\label{eq:serendipity.well-posedness:F}
    \norm[F]{(\bvec{\sigma}, \bvec{\lambda})}
    \coloneq \norm[F]{\bvec{\sigma}} + \norm[F]{\bvec{\lambda}}
    \lesssim \norm[F]{\mathcal{L}_F}
  \end{equation}
  where $\norm[F]{\mathcal{L}_F}$ is the norm of the linear form $\mathcal{L}_F$ induced by the $L^2$-norm on $\vPoly{k}(F)\times\cRoly{\ell_F+1}(F)$.
\end{proposition}
\begin{proof}
  The estimate \eqref{eq:serendipity.well-posedness:F} classically follows from the following inf-sup condition on $\mathcal{A}_F$:
  For all $(\bvec{\upsilon},\bvec{\nu})\in\vPoly{k}(F)\times\cRoly{\ell_F+1}(F)$,
  \begin{equation}\label{eq:AF:inf-sup}
    \norm[F]{(\bvec{\upsilon},\bvec{\nu})}
    \lesssim\sup_{(\bvec{\tau},\bvec{\mu})\in\vPoly{k}(F)\times\cRoly{\ell_F+1}(F)\setminus\{(\bvec{0},\bvec{0})\}}
    \frac{\mathcal{A}_F((\bvec{\upsilon},\bvec{\nu}), (\bvec{\tau},\bvec{\mu}))}{\norm[F]{(\bvec{\tau},\bvec{\mu})}}
    \eqcolon\$.
  \end{equation}
  Let us prove this condition.
  Taking $(\bvec{\tau},\bvec{\mu}) = (\bvec{\upsilon},\bvec{\nu})$ in the definition \eqref{eq:AF} of $\mathcal{A}_F$, we obtain
  \begin{equation}\label{eq:serendipity.well-posedness:F:1}
    h_F\sum_{E\in\EF}\norm[E]{\bvec{\upsilon}\cdot\tangent_E}^2
    + h_F^2\norm[F]{\ROT_F\bvec{\upsilon}}^2
    = \mathcal{A}_F((\bvec{\upsilon},\bvec{\nu}), (\bvec{\upsilon},\bvec{\nu}))
    \le\$\norm[F]{(\bvec{\upsilon},\bvec{\nu})}.
  \end{equation}
  We next notice that, for all $q\in\Poly{0,k-1}(F)$,
  \[
  \int_F\ROT_F\bvec{\upsilon}~q
  = \int_F\Rproj{k-2}{F}\bvec{\upsilon}\cdot\VROT_F q
  - \sum_{E\in\EF}\omega_{FE}\int_E(\bvec{\upsilon}\cdot\tangent_E)~q.
  \]
  Taking $q$ such that $\VROT_Fq = \Rproj{k-2}{F}\bvec{\upsilon}$ (which is possible since $\VROT_F:\Poly{0,k-1}(F)\to\Roly{k-2}(F)$ is an isomorphism) and using Cauchy--Schwarz, discrete inverse, and trace inequalities along with $\norm[F]{q}\lesssim h_F\norm[F]{\VROT_F q} = h_F\norm[F]{\Rproj{k-2}{F}\bvec{\upsilon}}$ this gives, after simplifying and raising to the square,
  \begin{equation}\label{eq:serendipity.well-posedness:F:2}
    \norm[F]{\Rproj{k-2}{F}\bvec{\upsilon}}^2
    \lesssim h_F\sum_{E\in\EF}\norm[E]{\bvec{\upsilon}\cdot\tangent_E}^2
    + h_F^2\norm[F]{\ROT_F\bvec{\upsilon}}^2
    \lesssim\$\norm[F]{(\bvec{\upsilon},\bvec{\nu})},
  \end{equation}
  the conclusion being a consequence of \eqref{eq:serendipity.well-posedness:F:1}.
  On the other hand, taking $(\bvec{\tau},\bvec{\mu}) = (\bvec{0},\Rcproj{\ell_F+1}{F}\bvec{\upsilon})$ in the definition \eqref{eq:AF} of $\mathcal{A}_F$ and recalling the definition \eqref{eq:AF:inf-sup} of $\$$, we obtain
  \begin{equation}\label{eq:serendipity.well-posedness:F:3}
    \norm[F]{\Rcproj{\ell_F+1}{F}\bvec{\upsilon}}^2
    = \mathcal{A}_F((\bvec{\upsilon},\bvec{\nu}), (\bvec{0},\Rcproj{\ell_F+1}{F}\bvec{\upsilon}))
    \le\$\norm[F]{\Rcproj{\ell_F+1}{F}\bvec{\upsilon}}
    \le\$\norm[F]{\bvec{\upsilon}}
    \le\$\norm[F]{(\bvec{\upsilon},\bvec{\nu})},
  \end{equation}
  where the second inequality follows from the $L^2$-boundedness of $\Rcproj{\ell_F+1}{F}$.
  Summing \eqref{eq:serendipity.well-posedness:F:1}, \eqref{eq:serendipity.well-posedness:F:2} and \eqref{eq:serendipity.well-posedness:F:3} and recalling \eqref{eq:est.reconstruction.curl} along with \eqref{eq:def.ellF}, we obtain
  \begin{equation}\label{eq:serendipity.well-posedness:F:4}
    \norm[F]{\bvec{\upsilon}}^2\lesssim\$\norm[F]{(\bvec{\upsilon},\bvec{\nu})}.
  \end{equation}
  To estimate the $L^2$-norm of $\bvec{\nu}$, we notice that $\bvec{\nu}\in\vPoly{k}(F)$ owing to \eqref{eq:def.ellF}, and we write
  \[
  \begin{aligned}
    \norm[F]{\bvec{\nu}}^2
    &= \mathcal{A}_F((\bvec{\upsilon},\bvec{\nu}), (-\bvec{\nu},\bvec{0}))
    + h_F\sum_{E\in\EF}\int_E(\bvec{\upsilon}\cdot\tangent_E)~(\bvec{\nu}\cdot\tangent_E)
    + h_F^2\int_F\ROT_F\bvec{\upsilon}~\ROT_F\bvec{\nu}
    \\
    &\lesssim
    \$\norm[F]{\bvec{\nu}}
    + h_F^{\frac12}\left(
    \sum_{E\in\EF}\norm[E]{\bvec{\upsilon}\cdot\tangent_E}^2
    \right)^{\frac12}~\norm[F]{\bvec{\nu}}
    + h_F\norm[F]{\ROT_F\bvec{\upsilon}}~\norm[F]{\bvec{\nu}},
  \end{aligned}
  \]
  where we have used Cauchy--Schwarz, discrete trace, and inverse inequalities to pass to the second line.
  Combining this estimate with \eqref{eq:serendipity.well-posedness:F:1}, simplifying, and raising to the square, we conclude that
  \begin{equation}\label{eq:serendipity.well-posedness:F:5}
    \norm[F]{\bvec{\nu}}^2\lesssim \$^2 + \$\norm[F]{(\bvec{\upsilon},\bvec{\nu})}.
  \end{equation}
  Summing \eqref{eq:serendipity.well-posedness:F:4} and \eqref{eq:serendipity.well-posedness:F:5}, using Young's inequality and taking the square root gives \eqref{eq:AF:inf-sup}.
\end{proof}

\subsubsection{Serendipity spaces and operators}\label{sec:serendipity.ddr:faces}

The serendipity gradient and curl spaces on a face $F\in\Fh$ are respectively defined as
\begin{align*}
  \sXgrad{F}&\coloneq
  \left\{\su{q}_F=(\s{q}_F,\s{q}_{\EF})\st\text{$\s{q}_F\in\Poly{\ell_F}(F)$ and $\s{q}_{\EF}\in\Poly[c]{k+1}(\EF)$}\right\},\\
  \sXcurl{F}&\coloneq
  \Big\{
  \begin{aligned}[t]
    \suvec{v}_F&=(\sbvec{v}_{\cvec{R},F},\sbvec{v}_{\cvec{R},F}^\compl,(\s{v}_E)_{E\in\EF})
    \st
    \\
    &\qquad
    \text{$\sbvec{v}_{\cvec{R},F}\in\Roly{k-1}(F)$, $\sbvec{v}_{\cvec{R},F}^\compl\in\cRoly{\ell_F+1}(F)$,
    and $\s{v}_E\in\Poly{k}(E)$ for all $E\in\EF$}      
    \Big\}.
  \end{aligned}
\end{align*}
We note that, since these serendipity spaces are subspaces of $\Xgrad{F}$ and $\Xcurl{F}$, the components $L^2$-norms defined by \eqref{eq:tnorm} can be applied to their elements. 

The \emph{gradient serendipity operator on faces} $\SG{F}:\sXgrad{F}\to\vPoly{k}(F)$, the role of which is to reconstruct a consistent gradient (see Proposition \ref{prop:consistency.S.E.sfP} below), is defined, for all $\su{q}_F\in\sXgrad{F}$, as the component $\bvec{\sigma}$ of the unique solution to problem \eqref{eq:serendipity.problem:F} with right-hand side linear form $\mathcal{L}_F = \mathcal{L}_{\GRAD,F}$ given by
\begin{equation}\label{eq:LF.grad}
  \mathcal{L}_{\GRAD,F}(\su{q}_F;\bvec{\tau},\bvec{\mu})
  =
  h_F\sum_{E\in\EF}\int_E\s{q}_E'~(\bvec{\tau}\cdot\tangent_E)
  -\int_F \s{q}_F\DIV_F\bvec{\mu} + \sum_{E\in\EF}\omega_{FE}\int_E\s{q}_E~(\bvec{\mu}\cdot\normal_{FE}).
\end{equation}
Combining the estimate \eqref{eq:serendipity.well-posedness:F} with Cauchy--Schwarz and discrete trace and inverse inequalities, we infer
\begin{equation}\label{eq:continuity.SGF}
  \norm[F]{\SG{F}\su{q}_F}
  \lesssim h_F^{-1}\tnorm[\GRAD,F]{\su{q}_F}
  \qquad\forall\su{q}_F\in\sXgrad{F}.
\end{equation}

The \emph{curl serendipity operator on faces} $\SC{F}:\sXcurl{F}\to\vPoly{k}(F)$, which reconstructs a consistent vector potential (as shown by Proposition \ref{prop:consistency.S.E.sfP} below), is defined for all $\suvec{v}_F\in\sXcurl{F}$ as the component $\bvec{\sigma}$ of the unique solution to problem \eqref{eq:serendipity.problem:F} with right-hand side linear form $\mathcal{L}_F = \mathcal{L}_{\CURL,F}$ given by
\begin{multline}\label{eq:LF.curl}
  \mathcal{L}_{\CURL,F}(\suvec{v}_F;\bvec{\tau},\bvec{\mu})
  = h_F\sum_{E\in\EF}\int_E \s{v}_E~(\bvec{\tau}\cdot\tangent_E)
  \\
  + h_F^2\bigg(
  \int_F\sbvec{v}_{\cvec{R},F}\cdot\VROT_F\ROT_F\bvec{\tau}
  - \sum_{E\in\EF}\omega_{FE}\int_E \s{v}_E~\ROT_F\bvec{\tau}
  \bigg)
  + \int_F\sbvec{v}_{\cvec{R},F}^\compl\cdot\bvec{\mu}.
\end{multline}
The estimate \eqref{eq:serendipity.well-posedness:F} along with Cauchy--Schwarz and discrete trace and inverse inequalities yields
\begin{equation}\label{eq:continuity.SCF}
  \norm[F]{\SC{F}\suvec{v}_F}
  \lesssim\tnorm[\CURL,F]{\suvec{v}_F}
  \qquad\forall\suvec{v}_F\in\sXcurl{F}.
\end{equation}

\subsubsection{Extensions and reductions}

The \emph{extensions on faces} $\Egrad{F}:\sXgrad{F}\to\Xgrad{F}$ and $\Ecurl{F}:\sXcurl{F}\to\Xcurl{F}$ are such that
\begin{alignat}{4} \label{eq:Egrad.F}
  \Egrad{F}\su{q}_F
  &= (\EPoly{F}\su{q}_F, \s{q}_{\EF})
  &\qquad&\forall\su{q}_F\in\sXgrad{F},
  \\ \label{eq:Ecurl.F}
  \Ecurl{F}\suvec{v}_F
  &= (\sbvec{v}_{\cvec{R},F}, \Rcproj{k}{F}\SC{F}\suvec{v}_F, (\s{v}_E)_{E\in\EF})
  &\qquad&\forall\suvec{v}_F\in\sXcurl{F},
\end{alignat}
with $\EPoly{F}:\sXgrad{F}\to\Poly{k-1}(F)$ defined through a formal integration-by-parts using the serendipity gradient:
For all $\bvec{w}_F\in\cRoly{k}(F)$,
\begin{equation}\label{eq:def.hatErF}
  \int_F\EPoly{F}\su{q}_F\DIV_F\bvec{w}_F
  = - \int_F\SG{F}\su{q}_F\cdot\bvec{w}_F
  + \sum_{E\in\EF}\omega_{FE}{}\int_E \s{q}_E~(\bvec{w}_F\cdot\normal_{FE}).
\end{equation}
Using discrete trace inequalities, \eqref{eq:continuity.SGF}, \eqref{eq:continuity.SCF}, and the isomorphism property of $\DIV_F:\cRoly{k}(F)\to\Poly{k-1}(F)$ stated in \cite[Lemma 9]{Di-Pietro.Droniou:21*2}, we get the following continuity properties:
\begin{alignat}{4}
\label{eq:est.EgradF}
\tnorm[\GRAD,F]{\Egrad{F}\su{q}_F}\lesssim{}& \tnorm[\GRAD,F]{\su{q}_F}&\qquad&\forall\su{q}_F\in\sXgrad{F},\\
\label{eq:est.EcurlF}
\tnorm[\CURL,F]{\Ecurl{F}\suvec{v}_F}\lesssim{}& \tnorm[\CURL,F]{\suvec{v}_F}&\qquad&\forall\suvec{v}_F\in\sXcurl{F}.
\end{alignat}

The \emph{reductions on faces}, which also define the interpolators $\sIgrad{F}$ on $\sXgrad{F}$ and $\sIcurl{F}$ on $\sXcurl{F}$ through \eqref{eq:def.tIg}, are $\Rgrad{F}:\Xgrad{F}\to\sXgrad{F}$ and $\Rcurl{F}:\Xcurl{F}\to\sXcurl{F}$ such that
\begin{alignat}{4}\label{eq:R.grad.F}
  \Rgrad{F}\underline{q}_F
  &= (\lproj{\ell_F}{F}q_F, q_{\EF})
  &\qquad&\forall\underline{q}_F\in\Xgrad{F},
  \\ \label{eq:R.curl.F}
  \Rcurl{F}\uvec{v}_F
  &= (\bvec{v}_{\cvec{R},F}, \Rcproj{\ell_F+1}{F}\bvec{v}_{\cvec{R},F}^\compl, (v_E)_{E\in\EF})
  &\qquad&\forall\uvec{v}_F\in\Xcurl{F}.
\end{alignat}

\subsection{Serendipity, extension, and reduction operators on elements}

\subsubsection{Serendipity problem}

For a fixed element $T\in\Th$, recalling that $\eta_T$ is given by Assumption \ref{assum:choice.detaP}, we let
\begin{equation}\label{eq:def.ellT}
\ell_T\coloneq k+1-\eta_T\le k-1
\end{equation}
and, given a linear form $\mathcal{L}_T:\vPoly{k}(T)\times\cRoly{\ell_T+1}(T)\to\Real$, consider the following problem:
Find $(\bvec{\sigma},\bvec{\lambda})\in\vPoly{k}(T)\times\cRoly{\ell_T+1}(T)$ such that
\begin{equation}\label{eq:serendipity.problem:T}
  \mathcal{A}_T((\bvec{\sigma},\bvec{\lambda}), (\bvec{\tau},\bvec{\mu}))
  = \mathcal{L}_T(\bvec{\tau},\bvec{\mu})\qquad
  \forall (\bvec{\tau},\bvec{\mu})\in\vPoly{k}(T)\times\cRoly{\ell_T+1}(T),
\end{equation}
with bilinear form $\mathcal{A}_T:[\vPoly{k}(T)\times\cRoly{\ell_T+1}(T)]^2\to\Real$ such that
\begin{equation*}%% \label{eq:AT}
  \mathcal{A}_T((\bvec{\upsilon},\bvec{\nu}), (\bvec{\tau},\bvec{\mu}))
  \coloneq
  h_T\sum_{F\in\FT}\int_{F}\bvec{\upsilon}_{{\rm t},F}\cdot\bvec{\tau}_{{\rm t},F}
  + h_T^2\int_T\CURL\bvec{\upsilon}\cdot\CURL\bvec{\tau}
  + \int_T\bvec{\upsilon}\cdot\bvec{\mu} - \int_T\bvec{\tau}\cdot\bvec{\nu}.
\end{equation*}
We remind the reader that, for any $\bvec{\upsilon}:T\to\Real^3$ smooth enough and any $F\in\FT$, $\bvec{\upsilon}_{{\rm t},F}\coloneq\normal_F\times(\bvec{\upsilon}_{|F}\times\normal_F)$ denotes the tangential component of $\bvec{\upsilon}$ on $F$.
Proceeding as in the proof of Proposition \ref{prop:serendipity.well-posedness:F}, we can prove the following result.

\begin{proposition}[Well-posedness of the serendipity problem on elements]\label{prop:serendipity.well-posedness:T}
  The serendipity problem \eqref{eq:serendipity.problem:T} admits a unique solution which satisfies
  \begin{equation}\label{eq:serendipity.well-posedness:T}
    \norm[T]{\bvec{\sigma}} + \norm[T]{\bvec{\lambda}}
    \lesssim \norm[T]{\mathcal{L}_T}
  \end{equation}
  where $\norm[T]{\mathcal{L}_T}$ is the norm of the linear form $\mathcal{L}_T$ induced by the $L^2$-norm on $\vPoly{k}(T)\times\cRoly{\ell_T+1}(T)$.
\end{proposition}

\subsubsection{Serendipity spaces and operators}\label{sec:serendipity.ddr:elements}

The serendipity spaces on an element $T\in\Th$ are defined as
\begin{align*}
  \sXgrad{T}&\coloneq
  \Big\{
    \begin{aligned}[t]
      \su{q}_T=(\s{q}_T,(\s{q}_F)_{F\in\FT},\s{q}_{\ET})\st{}&
       \text{$\s{q}_T\in\Poly{\ell_T}(T)$, $\s{q}_F\in\Poly{\ell_F}(F)$ for all $F\in\FT$},\\
      & \text{and $\s{q}_{\ET}\in\Poly[c]{k+1}(\ET)$}
    \Big\},
    \end{aligned}\\
  \sXcurl{T}&\coloneq
  \Big\{
  \begin{aligned}[t]
    \suvec{v}_T&=(\sbvec{v}_{\cvec{R},T},\sbvec{v}_{\cvec{R},T}^\compl,(\sbvec{v}_{\cvec{R},F},\sbvec{v}_{\cvec{R},F}^\compl)_{F\in\FT},(\s{v}_E)_{E\in\ET})
    \st
    \\
    &\qquad
    \text{$\sbvec{v}_{\cvec{R},T}\in\Roly{k-1}(T)$ and $\sbvec{v}_{\cvec{R},T}^\compl\in\cRoly{\ell_T+1}(T)$,}\\
    &\qquad
    \text{$\sbvec{v}_{\cvec{R},F}\in\Roly{k-1}(F)$ and $\sbvec{v}_{\cvec{R},F}^\compl\in\cRoly{\ell_F+1}(F)$ for all $F\in\FT$,}\\
    &\qquad
    \text{and $\s{v}_E\in\Poly{k}(E)$ for all $E\in\ET$}      
    \Big\}.
  \end{aligned}
\end{align*}
The component norms defined by \eqref{eq:tnorm} naturally apply to these spaces.

The \emph{gradient serendipity operator on elements} $\SG{T}:\sXgrad{T}\to\vPoly{k}(T)$ is defined, for all $\su{q}_T\in\sXgrad{T}$, as the component $\bvec{\sigma}$ of the unique solution to problem \eqref{eq:serendipity.problem:T} with $\mathcal{L}_T = \mathcal{L}_{\GRAD,T}$, where
\begin{multline}\label{eq:LT.grad}
  \mathcal{L}_{\GRAD,T}(\su{q}_T;\bvec{\tau},\bvec{\mu})
  =
  h_T\sum_{F\in\FT}\int_F\cGF\Egrad{F}\su{q}_F\cdot\bvec{\tau}_{{\rm t},F}
  \\
  -\int_T \s{q}_T\DIV\bvec{\mu}
  + \sum_{F\in\FT}\omega_{TF}\int_F\trF\Egrad{F}\su{q}_F~(\bvec{\mu}\cdot\normal_F).
\end{multline}
Combining the a priori estimate \eqref{eq:serendipity.well-posedness:T} with Cauchy--Schwarz and discrete trace and inverse inequalities,
observing that $\norm[F]{\cGF\Egrad{F}\su{q}_F}\lesssim h_F^{-1}\tnorm[\GRAD,F]{\su{q}_F}$ by \eqref{eq:GT.boundedness} for $(\sfP,\underline{q}_\sfP) = (F,\Egrad{F}\su{q}_F)$ followed by \eqref{eq:est.EgradF}, and that $\norm[F]{\trF\Egrad{F}\su{q}_F}\lesssim\tnorm[\GRAD,F]{\su{q}_F}$ (consequence of \cite[Proposition 6]{Di-Pietro.Droniou:21*2} followed by \eqref{eq:est.EgradF}), we have
\begin{equation}\label{eq:continuity.SGT}
  \norm[T]{\SG{T}\su{q}_T}
  \lesssim h_T^{-1}\tnorm[\GRAD,T]{\su{q}_T}
  \qquad\forall\su{q}_T\in\sXgrad{T}.
\end{equation}

The \emph{curl serendipity operator on elements} $\SC{T}:\sXcurl{T}\to\vPoly{k}(T)$ is defined, for all $\suvec{v}_T\in\sXcurl{T}$, as the component $\bvec{\sigma}$ of the unique solution to problem \eqref{eq:serendipity.problem:T} with $\mathcal{L}_T\coloneq \mathcal{L}_{\CURL,T}$ defined by
\begin{multline}\label{eq:LT.curl}
  \mathcal{L}_{\CURL,T}(\suvec{v}_T;\bvec{\tau},\bvec{\mu})
  = h_T\sum_{F\in\FT}\int_F \trFt\Ecurl{F}\suvec{v}_F\cdot\bvec{\tau}_{{\rm t},F}
  \\
  + h_T^2\left(
  \int_T\sbvec{v}_{\cvec{R},T}\cdot\CURL\CURL\bvec{\tau}
  + \sum_{F\in\FT}\omega_{TF}\int_F\trFt\Ecurl{F}\suvec{v}_F\cdot(\CURL\bvec{\tau}\times\normal_F)
  \right)
  + \int_T\sbvec{v}_{\cvec{R},T}^\compl\cdot\bvec{\mu}.
\end{multline}
Using the a priori estimate \eqref{eq:serendipity.well-posedness:T}, Cauchy--Schwarz and discrete trace and inverse inequalities, as well as estimates of $\trFt\Ecurl{F}\suvec{v}_F$ following from \cite[Proposition 6]{Di-Pietro.Droniou:21*2} and \eqref{eq:est.EcurlF}, we have
\begin{equation}\label{eq:continuity.SCT}
  \norm[T]{\SC{T}\suvec{v}_T}
  \lesssim
  \tnorm[\CURL,T]{\suvec{v}_T}
  \qquad\forall\suvec{v}_T\in\sXcurl{T}.
\end{equation}

\subsubsection{Extensions and reductions}

The \emph{extensions on elements} are $\Egrad{T}:\sXgrad{T}\to\Xgrad{T}$ and $\Ecurl{T}:\sXcurl{T}\to\Xcurl{T}$ such that
\begin{alignat}{4} \label{eq:Egrad.T}
  \Egrad{T}\su{q}_T
  &= (\EPoly{T}\su{q}_T, (\EPoly{F}\su{q}_F)_{F\in\FT}, \s{q}_{\ET})
  &\enspace&\forall\su{q}_T\in\sXgrad{T},
  \\ \label{eq:Ecurl.T}
  \Ecurl{T}\suvec{v}_T
  &= (
  \sbvec{v}_{\cvec{R},T}, \Rcproj{k}{T}\SC{T}\suvec{v}_T,
  (\sbvec{v}_{\cvec{R},F}, \Rcproj{k}{F}\SC{F}\suvec{v}_F)_{F\in\FT}, (\s{v}_E)_{E\in\ET})
  &\enspace&\forall\suvec{v}_T\in\sXcurl{T},
\end{alignat}
with $\EPoly{F}$ defined by \eqref{eq:def.hatErF} and $\EPoly{T}:\sXgrad{T}\to\Poly{k-1}(T)$ such that, for all $\bvec{w}_T\in\cRoly{k}(T)$,
\begin{equation*}%\label{eq:def.hatErT}
  \int_T\EPoly{T}\su{q}_T\DIV\bvec{w}_T
  = - \int_T\SG{T}\su{q}_T\cdot\bvec{w}_T
  + \sum_{F\in\FT}\omega_{TF}\int_F \trF\Egrad{F}\su{q}_F~(\bvec{w}_T\cdot\normal_{F}).
\end{equation*}
Proceeding in a similar way as for the face extensions, \eqref{eq:continuity.SGT} and \eqref{eq:continuity.SCT} yield
\begin{alignat}{4}
   \label{eq:cont.Egrad}
  \tnorm[\GRAD,T]{\Egrad{T}\su{q}_T}
  &\lesssim \tnorm[\GRAD,T]{\su{q}_T}
  &\qquad&\forall\su{q}_T\in\sXgrad{T},\\
   %\label{eq:cont.Ecurl}
  \nonumber
  \tnorm[\CURL,T]{\Ecurl{T}\suvec{v}_T}
  &\lesssim \tnorm[\CURL,T]{\suvec{v}_T}
  &\qquad&\forall\suvec{v}_T\in\sXcurl{T}.
\end{alignat}

The \emph{reductions on elements} are $\Rgrad{T}:\Xgrad{T}\to\sXgrad{T}$ and $\Rcurl{T}:\Xcurl{T}\to\sXcurl{T}$ such that
\begin{alignat}{4}\label{eq:R.grad.T}
  \Rgrad{T}\underline{q}_T
  &= (\RPoly{T}\underline{q}_T,(\lproj{\ell_F}{F}q_F)_{F\in\FT},q_{\ET})
  &\quad&\forall\underline{q}_T\in\Xgrad{T},\\ \label{eq:R.curl.T}
  \Rcurl{T}\uvec{v}_T
  &= (\RRoly{T}\uvec{v}_T, \Rcproj{\ell_T+1}{T}\bvec{v}_{\cvec{R},T}^\compl,(\bvec{v}_{\cvec{R},F},\Rcproj{\ell_F+1}{F}\bvec{v}_{\cvec{R},F}^\compl)_{F\in\FT},(v_E)_{E\in\ET})
  &\quad&\forall\uvec{v}_T\in\Xcurl{T},
\end{alignat}
where $\RPoly{T}\underline{q}_T\in\Poly{\ell_T}(T)$ is such that, for all $\bvec{w}_T\in\cRoly{\ell_T+1}(T)$,
\begin{equation}\label{eq:def.hatRqT}
  \int_T\RPoly{T}\underline{q}_T\DIV\bvec{w}_T
  =-\int_T\cGT\underline{q}_T\cdot\bvec{w}_T
  + \sum_{F\in\FT}\omega_{TF}\int_F \trF\Egrad{F}\Rgrad{F}\underline{q}_F~(\bvec{w}_T\cdot\normal_F)
\end{equation}
while $\RRoly{T}\uvec{v}_T\in\Roly{k-1}(T)$ is such that, for all $\bvec{w}_T\in\cGoly{k}(T)$,
\begin{equation}
  \label{eq:def.hatu}
  \int_T \RRoly{T}\uvec{v}_T\cdot\CURL\bvec{w}_T
  = \int_T\cCT\uvec{v}_T\cdot\bvec{w}_T
  - \sum_{F\in\FT}\omega_{TF}\int_F \trFt\Ecurl{F}\Rcurl{F}\uvec{v}_F\cdot (\bvec{w}_T\times\normal_F).
\end{equation}

\section{Properties of the SDDR complex}\label{sec:prop.sddr}

\subsection{Preliminary results}

\begin{proposition}[Polynomial consistency of the serendipity operators and extensions]\label{prop:consistency.S.E.sfP}
  The serendipity operators and extensions enjoy the following polynomial consistency properties:
  For all $\sfP\in\Th\cup\Fh$,
  \begin{alignat}{4}
    \label{eq:SG.sfP:polynomial.consistency}
    \SG{\sfP}\sIgrad{\sfP} q ={}& \GRAD_\sfP q &\qquad&\forall q\in\Poly{k+1}(\sfP),\\
    \label{eq:EG.sfP:polynomial.consistency}
    \Egrad{\sfP}\sIgrad{\sfP}q ={}& \Igrad{\sfP}q&\qquad&\forall q\in\Poly{k+1}(\sfP),\\ 
    \label{eq:SC.sfP:polynomial.consistency}
    \SC{\sfP}\sIcurl{\sfP}\bvec{v} ={}& \bvec{v}&\qquad&\forall\bvec{v}\in\vPoly{k}(\sfP),\\
    \label{eq:EC.sfP:polynomial.consistency}
    \Ecurl{\sfP}\sIcurl{\sfP}\bvec{v} ={}& \Icurl{\sfP}\bvec{v}&\qquad&\forall\bvec{v}\in\vPoly{k}(\sfP).
    \end{alignat}
\end{proposition}

\begin{proof}
  Let us first consider the case $\sfP=F\in\Fh$.
  Let $q\in\Poly{k+1}(F)$. By uniqueness of the solution to \eqref{eq:serendipity.problem:F}, \eqref{eq:SG.sfP:polynomial.consistency} follows from proving that $(\GRAD_F q,\bvec{0})$ solves this problem with $\mathcal{L}_F=\mathcal{L}_{\GRAD,F}$ defined by \eqref{eq:LF.grad} with $\su{q}_F=\sIgrad{F}q$.
  Recalling the definition \eqref{eq:AF} of $\mathcal{A}_F$, it holds, for all $(\bvec{\tau},\bvec{\mu})\in\vPoly{k}(F)\times\cRoly{\ell_F+1}(F)$,  
  \[
  \begin{aligned}
    &\mathcal{A}_F((\GRAD_F q,\bvec{0}), (\bvec{\tau},\bvec{\mu}))
    = h_F\sum_{E\in\EF}\int_E(\GRAD_F q\cdot\tangent_E)~(\bvec{\tau}\cdot\tangent_E)
    + \int_F\GRAD_F q\cdot\bvec{\mu}
    \\
    &\qquad
    = h_F\sum_{E\in\EF}\int_E q_{|E}'~(\bvec{\tau}\cdot\tangent_E)
    - \int_F \lproj{\ell_F}{F} q~\DIV_F\bvec{\mu}
    + \sum_{E\in\EF}\omega_{FE}\int_E \lproj{k-1}{E} q~(\bvec{\mu}\cdot\normal_{FE})
    \\
    &\qquad
    = \mathcal{L}_{\GRAD,F}(\sIgrad{F} q; \bvec{\tau},\bvec{\mu}),
  \end{aligned}
  \]
  where we have used $\ROT_F\GRAD_F=0$ in the first line,
  an integration by parts along with the fact that $\DIV_F\bvec{\mu}\in\Poly{\ell_F}(F)$ and, for all $E\in\EF$, $\bvec{\mu}\cdot\normal_{FE}\in\Poly{\ell_F}(E)\subset\Poly{k-1}(E)$ (see \cite[Proposition 8]{Di-Pietro.Droniou:21*2}) to insert $\lproj{\ell_F}{F}$ and $\lproj{k-1}{E}$ in the second line,
  and the definition of $\sIgrad{F}$ (based on \eqref{eq:def.tIg} with full interpolator and reduction respectively given by \eqref{eq:Igradh} and \eqref{eq:R.grad.F}) together with \eqref{eq:LF.grad} to conclude.
  This proves \eqref{eq:SG.sfP:polynomial.consistency} for $\sfP=F$.
  From this result, using an integration by parts in \eqref{eq:def.hatErF} together with the fact that $\DIV_F:\cRoly{k}(F)\to\Poly{k-1}(F)$ is an isomorphism, we infer that $\EPoly{F}\sIgrad{F}q=\lproj{k-1}{F}q$, which gives \eqref{eq:EG.sfP:polynomial.consistency}.
  \smallskip

  We proceed similarly to prove \eqref{eq:SC.sfP:polynomial.consistency} for $\sfP=F$.
  Specifically, let $\bvec{v}\in\vPoly{k}(F)$.
  For all $(\bvec{\tau},\bvec{\mu})\in\vPoly{k}(F)\times\cRoly{\ell_F+1}(F)$, it holds
  \[
  \begin{aligned}
    &\mathcal{A}_F ((\bvec{v},\bvec{0}),(\bvec{\tau},\bvec{\mu}))
    = h_F\sum_{E\in\EF}\int_E(\bvec{v}\cdot\tangent_E)~(\bvec{\tau}\cdot\tangent_E)
    + h_F^2\int_F\ROT_F\bvec{v}~\ROT_F\bvec{\tau}
    + \int_F\bvec{v}\cdot\bvec{\mu}
    \\
    &\qquad
    = h_F\sum_{E\in\EF}\int_E\lproj{k}{E}(\bvec{v}\cdot\tangent_E)~(\bvec{\tau}\cdot\tangent_E)
    \\
    &\qquad\quad
    + h_F^2\bigg(
    \int_F\Rproj{k-2}{F}\bvec{v}\cdot\VROT_F\ROT_F\bvec{\tau}
    - \sum_{E\in\EF}\omega_{FE}\int_E\lproj{k}{E}(\bvec{v}\cdot\tangent_E)~\ROT_F\bvec{\tau}
    \bigg)
    + \int_F\Rcproj{\ell_F+1}{F}\bvec{v}\cdot\bvec{\mu}
    \\
    &\qquad
    = \mathcal{L}_{\CURL,F}(\sIcurl{F}\bvec{v};(\bvec{\tau},\bvec{\mu})),
  \end{aligned}
  \]
  where we have used an integration by parts along with the definitions of the $L^2$-orthogonal projectors to pass to the second line, and \eqref{eq:LF.curl} together with the definition of $\sIcurl{F}$ (obtained, according to \eqref{eq:def.tIg}, composing \eqref{eq:R.curl.F} and the $L^2$-projectors on the components of $\Xcurl{F}$) to conclude.
  This proves that $(\bvec{v},\bvec{0})$ solves \eqref{eq:serendipity.problem:F} with $\mathcal{L}_F = \mathcal{L}_{\CURL,F}$ given by \eqref{eq:LF.curl} with $\suvec{v}_F = \sIcurl{F}\bvec{v}$. This establishes \eqref{eq:SC.sfP:polynomial.consistency}, from which \eqref{eq:EC.sfP:polynomial.consistency} immediately follows by definition \eqref{eq:Ecurl.F} of $\Ecurl{F}$.
\medskip

We now briefly discuss the case $\sfP=T\in\Th$.
Let $q\in\Poly{k+1}(T)$. Using $\Rgrad{\sfQ}\Igrad{\sfQ}=\sIgrad{\sfQ}$ for $ \sfQ\in \Th\cup\Fh$ (by definition \eqref{eq:def.tIg}), \eqref{eq:EG.sfP:polynomial.consistency} for $\sfP=F$, and the polynomial consistencies of $\cGT$ and $\trF$ (see \cite[Lemma 3]{Di-Pietro.Droniou:21*2}), the definition \eqref{eq:def.hatRqT} shows that the element component of $\sIgrad{T}q$ is $\RPoly{T}\Igrad{T}q=\lproj{\ell_T}{T}q$. We can then proceed as above, using the polynomial consistency of $\cGF$ and $\trF$ together with \eqref{eq:EG.sfP:polynomial.consistency} for $\sfP=F$ to see that \eqref{eq:SG.sfP:polynomial.consistency} and \eqref{eq:EG.sfP:polynomial.consistency} hold for $\sfP=T$.
\smallskip

The proof of \eqref{eq:SC.sfP:polynomial.consistency} and \eqref{eq:EC.sfP:polynomial.consistency} for $\sfP=T$ is similar, noticing that the polynomial consistencies of $\trFt$ and $\cCT$ (see \cite[Eqs.~(3.25) and (3.29)]{Di-Pietro.Droniou:21*2}) and \eqref{eq:EC.sfP:polynomial.consistency} for $\sfP=F$ yield $\RRoly{T}\Icurl{T}\bvec{v}=\Rproj{k-1}{T}\bvec{v}$.
\end{proof}

\begin{lemma}[Projections of extension, serendipity and gradient operators]
  For all $\sfP\in\Th\cup\Fh$, it holds:
  \begin{alignat}{4}
    \label{eq:proj.ErP}
    \lproj{\ell_\sfP}{\sfP}\EPoly{\sfP}\su{q}_\sfP &= \s{q}_\sfP
    &\quad&\forall \su{q}_\sfP\in\sXgrad{\sfP},\\
    \label{eq:proj.GPErP}
    \Rcproj{k}{\sfP}\cGP\Egrad{\sfP}\su{q}_\sfP &= \Rcproj{k}{\sfP}\SG{\sfP}\su{q}_\sfP
    &\quad&\forall \su{q}_\sfP\in\sXgrad{\sfP}.
    \\ \label{eq:Rcproj.SC=compl}
    \Rcproj{\ell_\sfP+1}{\sfP}\SC{\sfP}\suvec{v}_\sfP &= \sbvec{v}_{\cvec{R},\sfP}^\compl
    &\quad&\forall\suvec{v}_\sfP\in\sXcurl{\sfP}.
  \end{alignat}
\end{lemma}

\begin{proof}
  We focus, for the sake of brevity, on the case $\sfP = F\in\Fh$, the case $\sfP = T\in\Th$ being similar.
  Taking a generic $\bvec{w}_F\in\cRoly{\ell_F+1}(F)$ and plugging $(\bvec{\tau},\bvec{\mu}) = (\bvec{0}, \bvec{w}_F)$ into the variational problem defining $\SG{F}\s{q}_F$ (i.e., \eqref{eq:serendipity.problem:F} with $\mathcal{L}_F = \mathcal{L}_{\GRAD,F}$ given by \eqref{eq:LF.grad}), we infer that
  \begin{equation}\label{eq:SG.cRolyl}
  \int_F\SG{F}\su{q}_F\cdot\bvec{w}_F
  = -\int_F \s{q}_F~\DIV_F\bvec{w}_F
  + \sum_{E\in\EF}\omega_{FE}\int_E \s{q}_E~(\bvec{w}_F\cdot\normal_{FE}).
  \end{equation}
  Writing the definition \eqref{eq:def.hatErF} of $\EPoly{F}$ with this choice of $\bvec{w}_F$ (which is possible since $\cRoly{\ell_F+1}(F)\subset\cRoly{k}(F)$ by \eqref{eq:def.ellF}) and subtracting \eqref{eq:SG.cRolyl} from the resulting expression yields
  \[
  \int_F \EPoly{F}\su{q}_F\DIV_F\bvec{w}_F = \int_F\s{q}_F\DIV_F\bvec{w}_F,
  \]
  which proves \eqref{eq:proj.ErP} since $\DIV_F:\cRoly{\ell_F+1}(F)\to\Poly{\ell_F}(F)$ is an isomorphism.
  To establish \eqref{eq:proj.GPErP}, we take $\bvec{w}_F\in\cRoly{k}(F)$ and write the definition \eqref{eq:cGF} of $\cGF\Egrad{F}\su{q}_F$ to get
  \[
  \int_F \cGF\Egrad{F}\su{q}_F{}\cdot\bvec{w}_F
  = -\int_F\EPoly{F}\su{q}_F\DIV_F\bvec{w}_F
  + \sum_{E\in\EF}\omega_{FE}\int_E \s{q}_E~(\bvec{w}_F\cdot\normal_{FE})
  = \int_F\SG{F}\su{q}_F\cdot\bvec{w}_F,
  \]
  where the second equality follows from the definition \eqref{eq:def.hatErF} of $\EPoly{F}$.
    Finally, to prove \eqref{eq:Rcproj.SC=compl}, it suffices to take test functions of the form $(\bvec{0},\bvec{\mu})$, with $\bvec{\mu}$ spanning $\cRoly{\ell_F+1}(F)$, in the problem defining $\SC{F}\suvec{v}_F$ (i.e., \eqref{eq:serendipity.problem:F} with right-hand side $\mathcal{L}_F = \mathcal{L}_{\CURL,F}$ given by \eqref{eq:LF.curl}).%
\end{proof}

\begin{lemma}[Equivalence of norms on {$\sXgrad{T}$} and {$\sXcurl{T}$}]\label{lem:equiv.norms.sXbullet}
  For $\bullet\in\{\GRAD,\CURL\}$, it holds $\norm[\bullet,\ser,T]{{\cdot}}\simeq \tnorm[\bullet,T]{{\cdot}}$ on $\sXbullet{T}$ where, in accordance with \eqref{eq:def.ps.tXg}, $\norm[\bullet,\ser,T]{{\cdot}}\coloneq\norm[\bullet,T]{\Ebullet{T}{\cdot}}$ with $\norm[\bullet,T]{{\cdot}}$ defined in Section \ref{sec:ddr:L2-products.norms}.
\end{lemma}

\begin{proof}
  We only consider the case $\bullet=\GRAD$, the proof for $\bullet=\CURL$ being similar.
  For all $\su{q}_T\in\sXgrad{T}$, we have
  \[
  \norm[\GRAD,\ser,T]{\su{q}_T}
  = \norm[\GRAD,T]{\Egrad{T}\su{q}_T}
  \lesssim \tnorm[\GRAD,T]{\Egrad{T}\su{q}_T}
  \lesssim \tnorm[\GRAD,T]{\su{q}_T},
  \]
  where the first inequality comes from the equivalence of norms on $\Xgrad{T}$, see \cite[Lemma 5]{Di-Pietro.Droniou:21*2}, and the second inequality is \eqref{eq:cont.Egrad}.
  To prove the converse inequality, we use \eqref{eq:proj.ErP} to write
  \begin{align*}
    \tnorm[\GRAD,T]{\su{q}_T}^2
    &= \norm[T]{\lproj{\ell_T}{T}\EPoly{T}\su{q}_T}^2 + \sum_{F\in\FT}h_F\norm[F]{\lproj{\ell_F}{F}\EPoly{F}\su{q}_F}^2 + \sum_{F\in\FT}h_F\sum_{E\in\EF}h_E\norm[E]{q_E}^2
    \\
    &\le\tnorm[\GRAD,T]{\Egrad{T}\su{q}_T}^2,
  \end{align*}
  where the last bound follows from the $L^2$-boundedness of the $L^2$-projectors and the definition \eqref{eq:Egrad.T} of $\Egrad{T}\su{q}_T$. We then invoke the equivalence of norms in $\Xgrad{T}$ (see \cite[Lemma 5]{Di-Pietro.Droniou:21*2}) to get $\tnorm[\GRAD,T]{\Egrad{T}\su{q}_T}^2\lesssim \norm[\GRAD,T]{\Egrad{T}\su{q}_T}^2=\norm[\GRAD,\ser,T]{\su{q}_T}^2$.
\end{proof}

\subsection{Commutation property for the serendipity operators}

The property in the following lemma justifies the choice of serendipity operators made in Sections \ref{sec:serendipity.ddr:faces} and \ref{sec:serendipity.ddr:elements} and will play a crucial role in the proof of Lemma \ref{lem:cochain.E.curl} below (cochain property for the extension map).

\begin{lemma}[Commutation property for the serendipity operators]\label{lem:commutation.serendipity}
  For all $\sfP\in\Th\cup\Fh$, it holds, recalling that $\suGP\coloneq\Rcurl{\sfP}\uGP\Egrad{\sfP}$ by \eqref{eq:def.tdg},
  \begin{equation}
    \label{eq:SRecRcurlG=GRAD}
    \SC{\sfP}\suGP\su{q}_\sfP
    = \SG{\sfP}\su{q}_\sfP
    \qquad\forall\su{q}_\sfP\in\sXgrad{\sfP},
  \end{equation}
  which expresses the fact that the following diagram commutes:
  \[
  \begin{tikzcd}
    \sXgrad{\sfP} \arrow[r, "\SG{\sfP}"] \arrow[rd, swap, "\suGP"] & \vPoly{k}(\sfP) \\ 
    {} & \sXcurl{\sfP} \arrow[u, swap, "\SC{\sfP}"] 
  \end{tikzcd}
  \]
\end{lemma}

\begin{proof}
  Let us consider $\sfP=F$, and set, for the sake of brevity, $\suvec{v}_F\coloneq\suGF\su{q}_F$.
  Then, $\s{v}_E = \s{q}_E'$ and
  \begin{equation}\label{eq:ser.red.ext:0}
    \sbvec{v}_{\cvec{R},F}^\compl
    = \Rcproj{\ell_F+1}{F}\cGF\Egrad{F}\su{q}_F
    = \Rcproj{\ell_F+1}{F}\SG{F}\su{q}_F
  \end{equation}
  by \eqref{eq:proj.GPErP} for $\sfP = F$ recalling \eqref{eq:Xcproj.ell+1.k} along with $\ell_F+1\le k$ (cf.~\eqref{eq:def.ellF}).
  Moreover, for all $z_F\in\Poly{k}(F)$,
  \begin{align}
    \int_F \sbvec{v}_{\cvec{R},F}\cdot\VROT_Fz_F
    - \sum_{E\in\EF}\omega_{FE}\int_E \s{v}_E z_F
    &=
    \int_F \Rproj{k-1}{F}\cGF\Egrad{F}\su{q}_F\cdot\VROT_Fz_F-\sum_{E\in\EF}\omega_{FE}\int_E \s{q}_E' z_F\nonumber\\ 
    &=
    \int_F \CF\uGF\Egrad{F}\su{q}_F~z_F=0,
  \label{eq:ser.red.ext:1}
  \end{align}
  where, to pass to the second line, we have invoked the definition \eqref{eq:CF} of $\CF$ together with the fact that the edge components of $\uGF\Egrad{F}\su{q}_F$ are $(\s{q}_E')_{E\in\EF}$, and we have concluded thanks to the local complex property stated in \cite[Proposition 2]{Di-Pietro.Droniou:21*2}.
  We next notice that, recalling \eqref{eq:ser.red.ext:0} and \eqref{eq:SG.cRolyl}, for all $\bvec{\mu}\in\cRoly{\ell_F+1}(F)$,
  \begin{equation}\label{eq:ser.red.ext:2}
    \int_F\sbvec{v}_{\cvec{R},F}^\compl\cdot\bvec{\mu}
    = \int_F\SG{F}\su{q}_F\cdot\bvec{\mu}
    = -\int_F \s{q}_F\DIV_F\bvec{\mu}
    + \sum_{E\in\EF}\omega_{FE}\int_E\s{q}_E~(\bvec{\mu}\cdot\normal_{FE}).
  \end{equation}
  Recalling that $v_E = \s{q}_E'$ for all $E\in\Eh$, plugging \eqref{eq:ser.red.ext:1} and \eqref{eq:ser.red.ext:2} into the expression \eqref{eq:LF.curl} of the linear form $\mathcal{L}_{\CURL,F}(\suvec{v}_F;\cdot,\cdot)$, and recalling the expression \eqref{eq:LF.grad} of the linear form $\mathcal{L}_{\GRAD,F}$, we conclude that $\mathcal{L}_{\CURL,F}(\suvec{v}_F;\bvec{\tau},\bvec{\mu}) = \mathcal{L}_{\GRAD,F}(\su{q}_F;\bvec{\tau},\bvec{\mu})$ for all $(\bvec{\tau},\bvec{\mu})\in\vPoly{k}(F)\times\cRoly{\ell_F+1}(F)$.
  Letting $\bvec{\lambda}_{\suvec{v}_F},\bvec{\lambda}_{\su{q}_F}\in\cRoly{\ell_F+1}(F)$ be the second components in the serendipity problems \eqref{eq:serendipity.problem:F} defining $\SC{F}\suvec{v}_F$ and $\SG{F}\su{q}_F$, respectively, this implies
  \[
  \mathcal{A}_F((\SC{F}\suvec{v}_F-\SG{F}\su{q}_F,\bvec{\lambda}_{\suvec{v}_F}-\bvec{\lambda}_{\su{q}_F}),(\bvec{\tau},\bvec{\mu}))
  = 0
  \qquad\forall(\bvec{\tau},\bvec{\mu})\in\vPoly{k}\times\cRoly{\ell_F+1}(F),
  \]
  showing that $\SC{F}\suvec{v}_F=\SG{F}\su{q}_F$.
  The proof in the case $\sfP=T\in\Th$ is similar. The details are left to the reader.
\end{proof}

\subsection{Homological and analytical properties for the gradient}

\begin{lemma}[Properties of the serendipity gradient space]\label{lem:properties.sXgrad}
  The serendipity construction for the gradient space satisfies both the homological properties of Assumption \ref{ass:cohomology} and the analytical properties of Assumption \ref{ass:analytical}, with continuity constants in \ref{P:continuity} and \ref{P:continuity.interpolator} that do not depend on $h$.
\end{lemma}

\begin{proof}
  \emph{(i) Proof of \ref{P:extension}.}
  Let $T\in\Th$ and $\su{q}_T\in\sXgrad{T}$ and let, for the sake of brevity, $\underline{q}_T\coloneq\Egrad{T}\su{q}_T$.
  Owing to \eqref{eq:proj.ErP}, for all $F\in\FT$ we have $\Rgrad{F}\underline{q}_F=\su{q}_F$.
  Plugging this relation into the definition \eqref{eq:def.hatRqT} of $\RPoly{T}\underline{q}_T$ we find, for all $\bvec{w}_T\in\cRoly{\ell_T+1}(T)$,
  \[
  \int_T\RPoly{T}\underline{q}_T~\DIV\bvec{w}_T
  = -\int_T\cGT\underline{q}_T\cdot\bvec{w}_T
  + \sum_{F\in\FT}\omega_{TF}\int_F \trF\underline{q}_F~(\bvec{w}_T\cdot\normal_F)
  = \int_T\EPoly{T}\su{q}_T\DIV\bvec{w}_T,
  \]
  where the conclusion comes from the definitions \eqref{eq:cGT} of $\cGT$ and \eqref{eq:Egrad.T} of $\Egrad{T}$. Recalling that $\DIV: \cRoly{\ell_T+1}(T)\to\Poly{\ell_T}(T)$ is an isomorphism and applying \eqref{eq:proj.ErP} for $\sfP=T$ yields $\RPoly{T}\underline{q}_T=\s{q}_T$. This proves that $\Rgrad{h}\Egrad{h}=\Id$ on the whole of $\sXgrad{T}$.
  \medskip\\
  \emph{(ii) Proof of \ref{P:consistency}.} The polynomial consistency is only meaningful for the local complex on a mesh element $T\in\Th$, for which it is simply \eqref{eq:EG.sfP:polynomial.consistency} with $\sfP=T$.
  \medskip\\
  \emph{(iii) Proof of \ref{P:complex.ER}.}
  Take $\underline{q}_h\in \ker\uGh$. We claim that $\Egrad{T}\Rgrad{T}\underline{q}_T=\underline{q}_T$ for all $T\in\Th$, which establishes a stronger version of \ref{P:complex.ER} (namely, $\Egrad{h}\Rgrad{h}-\Id=0$ on $\ker\uGh$). The condition $\uGh\underline{q}_h=\underline{0}$ implies $\uGT\underline{q}_T=\underline{0}$ for all $T\in\Th$ and thus, by exactness of the DDR sequence on $T$, $\underline{q}_T=\Igrad{T}M$ for some $M\in\Real$. The property \ref{P:consistency} then yields $\Egrad{T}\Rgrad{T}\underline{q}_T=\underline{q}_T$ as claimed.
  \medskip\\  
  \emph{(iv) Proof of \ref{P:E.R.cochain}.}
  For this slice of the complex between $\Real$ and $\Xgrad{h}$, we have $\Eg{i}=\Rg{i}=\Id_\Real$ and the operators in the sequences are $\Igrad{h}$ and $\sIgrad{h}$.
  The fact that the reductions are cochain maps then follows from the definition \eqref{eq:def.tIg} of $\sIgrad{h}$, while the cochain map property of the extension is simply \ref{P:consistency}, proved above, applied to constant polynomials.
  \medskip\\
  \emph{(iv) Proof of \ref{P:continuity}.}
  For all $T\in\Th$, the continuity of $\Rgrad{T}:\Xgrad{T}\to\sXgrad{T}$ is a direct consequence of the norm equivalences in $\Xgrad{T}$ and $\sXgrad{T}$ (see \cite[Lemma 5]{Di-Pietro.Droniou:21*2} and Lemma \ref{lem:equiv.norms.sXbullet}), the boundedness of the $L^2$-projector $\lproj{\ell_F}{F}$, and, to estimate $\RPoly{T}\underline{q}_T$, \cite[Lemma 9]{Di-Pietro.Droniou:21*2} and the continuities of $\cGT$ and $\trF$ (see \eqref{eq:GT.boundedness} and \cite[Eq.~(4.22)]{Di-Pietro.Droniou:21*2}).
  \medskip\\
  \emph{(iv) Proof of \ref{P:continuity.interpolator}.} This property on the original DDR complex is proved in \cite[Lemma 6]{Di-Pietro.Droniou:21*2}.
\end{proof}

\subsection{Homological and analytical properties for the curl}

The goal of this section is to prove homological and analytical properties for the curl.
We remind the reader that, according to \eqref{eq:def.tdg},
\[
\text{%
  $\suGh\coloneq\Rcurl{h}\uGh\Egrad{h}$\quad and\quad
  $\suCh\coloneq\Rdiv{h}\uCh\Ecurl{h}$.
}
\]
We start by addressing the cochain property for the reduction and extension maps.

\begin{lemma}[Cochain property of the reduction]\label{lem:RGER=RG}
  For the slice of the complex between $\Xgrad{h}$ and $\Xcurl{h}$, property \ref{P:E.R.cochain} holds for the reduction, i.e.,
  \begin{equation}\label{eq:RGER=RG}
    \suGh\Rgrad{h}\underline{q}_h=\Rcurl{h}\uGh\underline{q}_h\qquad\forall\underline{q}_h\in\Xgrad{h}.
  \end{equation}
\end{lemma}

\begin{proof}
  Let $\underline{q}_h\in\Xgrad{h}$ and set, for the sake of brevity, $\su{q}_h \coloneq \Rgrad{h}\underline{q}_h$.
  \medskip\\
  \emph{(i) Gradient components on $\cRoly{\ell_\sfP+1}(\sfP)$.} We start by proving the following result:
  For all $\sfP\in\Th\cup\Fh$,
  \begin{equation}\label{eq:RGER=RG.cRoly}
    \Rcproj{\ell_\sfP+1}{\sfP}\cGP\Egrad{\sfP}\Rgrad{\sfP}\underline{q}_\sfP
    = \Rcproj{\ell_\sfP+1}{\sfP}\cGP\underline{q}_\sfP\qquad
    \forall\underline{q}_\sfP\in\Xgrad{\sfP}.
  \end{equation}
  We detail the proof for the case $\sfP = T\in\Th$, the case $\sfP = F\in\Fh$ being similar.
  Owing to \eqref{eq:proj.GPErP} and \eqref{eq:Xcproj.ell+1.k} with $\ell = \ell_T$, we only have to prove that $\Rcproj{\ell_T+1}{T}\SG{T}\su{q}_T = \Rcproj{\ell_T+1}{T}\cGT\underline{q}_T$.
  This relation can be established taking $(\bvec{\tau},\bvec{\mu}) = (\bvec{0},\bvec{w}_T)$ with $\bvec{w}_T\in\cRoly{\ell_T+1}(T)$ as a test function in the problem defining $\SG{T}\su{q}_T$ (i.e., \eqref{eq:serendipity.problem:T} with $\mathcal{L}_T = \mathcal{L}_{\GRAD,T}$ given by \eqref{eq:LT.grad}) to write
  \[
  \int_T\SG{T}\su{q}_T\cdot\bvec{w}_T
  =
  -\int_T\RPoly{T}\underline{q}_T\DIV\bvec{w}_T
  + \sum_{F\in\FT}\omega_{TF}\int_F\trF\Egrad{F}\su{q}_F~(\bvec{w}_T\cdot\normal_F)
  \eqtext{\eqref{eq:def.hatRqT}} \int_T\cGT\underline{q}_T\cdot\bvec{w}_T.
  \]
  \\
  \emph{(ii) Cochain property of the reduction.}
  The components of $\Egrad{h}\su{q}_h$ and $\underline{q}_h$ on the mesh edge skeleton coincide, and thus so do their edge gradient as well as the components of their discrete gradients on $\Roly{k-1}(F)$, $F\in\Fh$ (depending only on the skeletal components).
  By definition of $\Rcurl{h}$, this shows that the equality in \eqref{eq:RGER=RG} holds for these components.
    The equality of the components on $\cRoly{\ell_F+1}(F)$, $F\in\Fh$, is an immediate consequence of \eqref{eq:RGER=RG.cRoly} for $\sfP = F$.

  Let now $T\in\Th$.
  We have just proved that, for all $F\in\FT$, $\suGF\su{q}_F=\Rcurl{F}\uGF\underline{q}_F$.  
  By \cite[Remark 15]{Di-Pietro.Droniou:21*2}, we have $\cCT\uGT=\bvec{0}$. Hence, the definition \eqref{eq:def.hatu} of $\RRoly{T}$ applied to $\uvec{v}_T=\uGT\Egrad{T}\su{q}_T$ and $\uvec{v}_T=\uGT\underline{q}_T$ shows that the components of each side of \eqref{eq:RGER=RG} on $\Roly{k-1}(T)$ coincide.
    The equality of the components on $\cRoly{\ell_T+1}(T)$, $T\in\Th$, is an immediate consequence of \eqref{eq:RGER=RG.cRoly}.
\end{proof}

\begin{lemma}[Cochain property of the extension]\label{lem:cochain.E.curl}
  For the slice of the complex between $\Xgrad{T}$ and $\Xcurl{T}$, property \ref{P:E.R.cochain} holds for the extensions, i.e.,
  \begin{equation}\label{eq:ERGE=GE}
    \Ecurl{h}\suGh\su{q}_h=\uGh\Egrad{h}\su{q}_h\qquad\forall \su{q}_h\in\sXgrad{h}.
  \end{equation}
\end{lemma}

\begin{proof}
  Given that $\Ecurl{h}$ and $\Rcurl{h}$ leave the components on $\Poly{k}(E)$, $E\in\Eh$, and $\Roly{k-1}(F)$, $F\in\Fh$, unchanged, the equality of these components on each side of \eqref{eq:ERGE=GE} is trivial.
  For $F\in\Fh$ we have, owing to \eqref{eq:SRecRcurlG=GRAD} and \eqref{eq:proj.GPErP} for $\sfP = F$,
  \[
  \Rcproj{k}{F}\SC{F}\suGF\su{q}_F
  = \Rcproj{k}{F}\SG{F}\su{q}_F
  = \Rcproj{k}{F}\cGF\Egrad{F}\su{q}_F,
  \]
  which proves that the components on $\cRoly{k}(F)$ of each side of \eqref{eq:ERGE=GE}  also coincide, so that
  \begin{equation}\label{eq:ERGE=GE:1}
    \Ecurl{F}\suGF\su{q}_F
    = \Ecurl{F}\Rcurl{F}\uGF\Egrad{F}\su{q}_F
    = \uGF\Egrad{F}\su{q}_F
    \qquad\forall F\in\FT.
  \end{equation}
  With the same approach, using \eqref{eq:SRecRcurlG=GRAD} and \eqref{eq:proj.GPErP} for $\sfP=T\in\Th$, we show that the components on $\cRoly{k}(T)$ also coincide.
  It remains to analyse, for $T\in\Th$, the components on $\Roly{k-1}(T)$.
  Setting $\uvec{v}_T \coloneq \uGT\Egrad{T}\su{q}_T$, this component for the left-hand side of \eqref{eq:ERGE=GE} is $\RRoly{T}\uvec{v}_T$ given by \eqref{eq:def.hatu}.
  Plugging into this expression $\cCT\uvec{v}_T = \bvec{0}$ (consequence of \cite[Remark 15]{Di-Pietro.Droniou:21*2}),
  using \eqref{eq:ERGE=GE:1} to write $\Ecurl{F}\Rcurl{F}\uvec{v}_F = \uGF\Egrad{F}\su{q}_F$,
  using \cite[Eq.~(3.26)]{Di-Pietro.Droniou:21*2} to write $\trFt\uGF\Egrad{F}\su{q}_F = \cGF\Egrad{F}\su{q}_F$,
  and recalling the link between element and face gradients expressed by \cite[Eq. (3.17)]{Di-Pietro.Droniou:21*2},
  we infer that 
  \[
  \int_T\RRoly{T}\uvec{v}_T\cdot\CURL\bvec{w}_T=\int_T\cGT\Egrad{T}\su{q}_T\cdot\CURL\bvec{w}_T
  \qquad\forall \bvec{w}_T\in\cGoly{k}(T),
  \]
  which proves the equality of the components on $\Roly{k-1}(T)$ of each side of \eqref{eq:ERGE=GE}.
\end{proof}

The following  intermediate result will be used to establish \ref{P:complex.ER}, as well as \ref{P:E.R.cochain} for the last part of the complex involving $\Xdiv{h}$.

\begin{lemma}[Relation among {$\uCh$, $\Ecurl{h}$, and $\Rcurl{h}$}]
  It holds
  \begin{equation}\label{eq:CERu=Cu}
    \uCh\Ecurl{h}\Rcurl{h}\uvec{v}_h = \uCh\uvec{v}_h\qquad\forall \uvec{v}_h\in\Xcurl{h}.
  \end{equation}
\end{lemma}

\begin{proof}
  Let $\uvec{v}_h\in\Xcurl{h}$ and set, for the sake of brevity, $\suvec{v}_h\coloneq\Rcurl{h}\uvec{v}_h$.
  The components of $\uvec{v}_h$ and $\Ecurl{h}\suvec{v}_h$ on $\Poly{k}(E)$, $E\in\Eh$, and on $\Roly{k-1}(F)$, $F\in\Fh$, coincide.
  Since the face curls only depend on these components, we infer that $\CF\Ecurl{F}\suvec{v}_F = \CF\uvec{v}_F$ for all $F\in\Fh$.
  The link between full curl and face curls given in \cite[Proposition 4]{Di-Pietro.Droniou:21*2} then shows that $\Gproj{k-1}{T}\cCT\Ecurl{T}\suvec{v}_T = \Gproj{k-1}{T}\cCT\uvec{v}_T$ for all $T\in\Th$.
  We have therefore proved that the components on the faces and on $\Goly{k-1}(T)$, $T\in\Th$, of both sides of \eqref{eq:CERu=Cu} coincide.
  Let now $T\in\Th$ and consider the components of the discrete curls on $\cGoly{k}(T)$.  
  By definition \eqref{eq:Ecurl.T} of $\Ecurl{T}$, the component of $\Ecurl{T}\suvec{v}_T$ on $\Roly{k-1}(T)$ is that of $\suvec{v}_T$, which is $\RRoly{T}\uvec{v}_T$ given by \eqref{eq:def.hatu}.
  The definitions \eqref{eq:cCT} of $\cCT\Ecurl{T}\suvec{v}_T$ and \eqref{eq:def.hatu} of $\RRoly{T}\uvec{v}_T$ then give, for all $\bvec{w}_T\in\cGoly{k}(T)$,
  \begin{multline*}
    \int_T \cCT\Ecurl{T}\suvec{v}_T\cdot\bvec{w}_T
    \\
    =\int_T\RRoly{T}\uvec{v}_T\cdot\CURL\bvec{w}_T+\sum_{F\in\FT}\omega_{TF}\int_F\trFt\Ecurl{F}\suvec{v}_F\cdot(\bvec{w}_T\times\normal_F)=\int_T\cCT\uvec{v}_T\cdot\bvec{w}_T.
  \end{multline*}
  This implies that $\Gcproj{k}{T}\cCT\Ecurl{T}\suvec{v}_T = \Gcproj{k}{T}\cCT\uvec{v}_T$ and concludes the proof.
\end{proof}

\begin{lemma}[Properties of the serendipity curl space]\label{lem:properties.sXcurl}
  The serendipity construction for the curl space satisfies both the homological properties of Assumption \ref{ass:cohomology} and the analytical properties of Assumption \ref{ass:analytical}, with continuity constants in \ref{P:continuity} and \ref{P:continuity.interpolator} that do not depend on $h$.
\end{lemma}

\begin{proof}
  \emph{(i) Proof of \ref{P:extension}}.
  Let $F\in\Fh$ and $\suvec{v}_F\in\sXcurl{F}$.
  By \eqref{eq:Rcproj.SC=compl} with $\sfP = F$, we have $\Rcproj{\ell_F+1}{F}\SC{F}\suvec{v}_F = \sbvec{v}_{\cvec{R},F}^\compl$.
  Combined with the definition \eqref{eq:R.curl.F} of $\Rcurl{F}$ with $\uvec{v}_F = \Ecurl{F}\suvec{v}_F$ and \eqref{eq:Xcproj.ell+1.k}, this gives
  \begin{equation}\label{eq:Rcurl.Ecurl.F}
    \forall F\in\Fh,\qquad
    \Rcurl{F}\Ecurl{F}\suvec{v}_F=\suvec{v}_F
    \qquad \forall \suvec{v}_F\in\sXcurl{F},
  \end{equation}
  the equality of the components on $\Poly{k}(E)$, $E\in\EF$, and $\Roly{k-1}(F)$ being trivial as these are not affected by the reduction/extension operators.
  
  Take now $\suvec{v}_h\in\sXcurl{h}$ and set, for the sake of brevity, $\uvec{v}_h\coloneq\Ecurl{h}\suvec{v}_h$.
  For all $T\in\Th$, invoking \eqref{eq:Rcurl.Ecurl.F} and using again \eqref{eq:Rcproj.SC=compl} and \eqref{eq:Xcproj.ell+1.k} gives
  \[
  \Rcurl{T}\uvec{v}_T
  = (\RRoly{T}\uvec{v}_T,\sbvec{v}_{\cvec{R},T}^\compl,(\sbvec{v}_{\cvec{R},F},\sbvec{v}_{\cvec{R},F}^\compl)_{F\in\FT},(\s{v}_E)_{E\in\ET}).
  \]
  According to its definition \eqref{eq:def.hatu}, $\RRoly{T}\uvec{v}_T\in\Roly{k-1}(T)$ satisfies, for all $\bvec{w}_T\in\cGoly{k}(T)$,
  \[
    \int_T \RRoly{T}\uvec{v}_T\cdot\CURL\bvec{w}_T
    =
    \int_T\cCT\uvec{v}_T\cdot\bvec{w}_T
    - \sum_{F\in\FT}\omega_{TF}\int_F\trFt\uvec{v}_F\cdot(\bvec{w}_T\times\normal_F)
    =\int_T\sbvec{v}_{\cvec{R},T}\cdot\CURL\bvec{w}_T,
  \]
  where we have used \eqref{eq:Rcurl.Ecurl.F} to write 
  \[
    \trFt\Ecurl{F}\Rcurl{F}\uvec{v}_F = \trFt\Ecurl{F}\Rcurl{F}\Ecurl{F}\suvec{v}_F = \trFt\Ecurl{F}\suvec{v}_F = \trFt\uvec{v}_F
  \]
  and concluded using the definition \eqref{eq:cCT} of $\cCT$ together with the fact that the component of $\uvec{v}_T = \Ecurl{T}\suvec{v}_T$ on $\Roly{k-1}(T)$ is $\sbvec{v}_{\cvec{R},T}$.
  Since $\CURL:\cGoly{k}(T)\to\Roly{k-1}(T)$ is onto, this gives $\RRoly{T}\uvec{v}_T=\sbvec{v}_{\cvec{R},T}$, concluding the proof of a stronger property than \ref{P:extension}, namely $\Rcurl{h}\Ecurl{h}=\Id$ on the whole of $\sXcurl{h}$.
  \medskip\\
  \emph{(ii) Proof of \ref{P:complex.ER}.}
  Let $\uvec{v}_h\in\ker\uCh$ and set $\suvec{v}_h \coloneq \Rcurl{h}\uvec{v}_h$.
  By \eqref{eq:CERu=Cu}, the vector
  \begin{equation}\label{eq:Erhv_minus_v}
  \begin{aligned}
    \Ecurl{h}\suvec{v}_h - \uvec{v}_h
    = \big(
    &(\RRoly{T}\uvec{v}_T-\bvec{v}_{\cvec{R},T},\Rcproj{k}{T}\SC{T}\suvec{v}_T-\bvec{v}_{\cvec{R},T}^\compl)_{T\in\Th},
    \\
    &(\bvec{0},\Rcproj{k}{F}\SC{F}\suvec{v}_F-\bvec{v}_{\cvec{R},F}^\compl)_{F\in\Fh}, \\
    &(0)_{E\in\Eh}\big)    
  \end{aligned}
  \end{equation}
  belongs to $\ker\uCh$.
  By local exactness, for all $T\in\Th$ we therefore have $\Ecurl{T}\suvec{v}_T - \uvec{v}_T = \uGT\underline{q}_T$ for some $\underline{q}_T\in\Xgrad{T}$ that is constant on the edge skeleton of $T$ (since its derivative there vanishes by \eqref{eq:Erhv_minus_v}).
  We can therefore assume, possibly after translation, that $q_{\ET}=0$.
  Moreover, for all $F\in\Fh$, by definition \eqref{eq:cGF} of $\cGF$, $q_F\in\Poly{k-1}(F)$ is entirely and uniquely fixed by the component $\Rcproj{k}{F}\cGF\underline{q}_F$ of $\uGT\underline{q}_T$ on $\cRoly{k}(F)$ (and the zero edge value of $\underline{q}_T$), which is $\Rcproj{k}{F}\SC{F}\suvec{v}_F-\bvec{v}_{\cvec{R},F}^\compl$ by \eqref{eq:Erhv_minus_v}; hence, $q_F$ matches between neighbouring elements. This enables us to trivially glue together all the $\underline{q}_T$ in each element into a vector $\underline{q}_h\in\Xgrad{h}$ that satisfies $\Ecurl{h}\suvec{v}_h - \uvec{v}_h = \uGh\underline{q}_h$, which concludes the proof of \ref{P:complex.ER}.
  \medskip\\
  \emph{(iii) Proof of \ref{P:E.R.cochain}.}
  This property is established in Lemmas \ref{lem:RGER=RG} and \ref{lem:cochain.E.curl}.
  \medskip\\
  \emph{(iv) Proof of \ref{P:continuity}.}
  The continuity of $\Rcurl{T}$ is proved in a similar way as that of $\Rgrad{T}$, using Lemma \ref{lem:equiv.norms.sXbullet}, the norm equivalence in $\Xcurl{T}$ stated in \cite[Lemma 5]{Di-Pietro.Droniou:21*2}, and the continuity of $\cCT$ stated in \cite[Proposition 13]{Di-Pietro.Droniou:21*1}.
  \medskip\\
  \emph{(v) Proof of \ref{P:consistency}.}
  This property is only meaningful for the local complex on a mesh element $T\in\Th$, for which it is simply \eqref{eq:EC.sfP:polynomial.consistency} for $\sfP=T$.
    \medskip\\
  \emph{(iv) Proof of \ref{P:continuity.interpolator}.} This property is proved in \cite[Lemma 6]{Di-Pietro.Droniou:21*2}.
\end{proof}

\subsection{Divergence space}\label{sec:Xdiv.no.serendipity}

As mentioned in Remark \ref{rem:no.seren.Xdiv}, no serendipity reduction of DOFs is performed on $\Xdiv{h}$. Hence, $\Ediv{T}=\Rdiv{T}=\Id_{\Xdiv{T}}$ and \ref{P:extension}, \ref{P:complex.ER}, \ref{P:continuity}, and \ref{P:consistency} are trivially satisfied.
Regarding \ref{P:E.R.cochain}, the co-chain property of the reduction follows from \eqref{eq:CERu=Cu}, while that of the extension is simply due to the definition \ref{eq:def.tdg} of $\suCh$.
  The property \ref{P:continuity.interpolator} is proved in \cite[Lemma 6]{Di-Pietro.Droniou:21*2}
\smallskip

Reducing the number of DOFs in discrete spaces associated with the divergence operator is actually technically quite challenging, as already noticed in \cite{Beirao-da-Veiga.Brezzi.ea:17}. However, we argue here that such a reduction is simply not possible if one wants to preserve even only two (important) properties of the discrete complex, namely the commutation property of Proposition \ref{prop:commutation.ID=dI} and the local exactness property.

Let $T\in\Th$.
The definition \eqref{eq:DT} of $\DT$ shows that this operator depends on the components on $\Poly{k}(F)$, $F\in\FT$, and $\Goly{k-1}(T)$ of $\Xdiv{T}$.
To preserve the commutation property, only the component $\cGoly{k}(T)$ can therefore be reduced (as already noticed in \cite{Beirao-da-Veiga.Brezzi.ea:17} in the context of VEM spaces). A candidate for the local serendipity space would therefore be:
For some $m_T<k-1$,
\begin{equation*}
  \sXdiv{T}\coloneq\Big\{
  \begin{aligned}[t]
    \uvec{w}_T
    &=\big(\bvec{w}_{\cvec{G},T},\bvec{w}_{\cvec{G},T}^\compl, (w_F)_{F\in\Fh}\big)\st
    \\
    &\qquad\text{$\bvec{w}_{\cvec{G},T}\in\Goly{k-1}(T)$ and $\bvec{w}_{\cvec{G},T}^\compl\in\cGoly{m_T+1}(T)$,
      and $w_F\in\Poly{k}(F)$ for all $F\in\FT$}
    \Big\},
  \end{aligned}
\end{equation*}
with reduction operator $\Rdiv{T}:\Xdiv{T}\to\sXdiv{T}$ defined by
\[
\Rdiv{T}\uvec{w}_T
=(\bvec{w}_{\cvec{G},T},\Gcproj{m_T+1}{T}\bvec{w}_{\cvec{G},T}^\compl,(w_F)_{F\in\FT})\quad\forall \uvec{w}_T\in\Xdiv{T}.
\]
However, such a choice would lead to a defect of the exactness property of the complex, specifically: 
\begin{equation}\label{eq:defect.exactness}
\ker\suCT\not\subset\Image\suGT.
\end{equation}
Let us clearly demonstrate this.
Since $m_T<k-1$, we can find $\sbvec{v}_{\cvec{R},T}\in\Roly{k-1}(T)$ such that $\sbvec{v}_{\cvec{R},T}\not=\bvec{0}$ and $\sbvec{v}_{\cvec{R},T}$ is orthogonal to $\Roly{m_T}(T)$ for the $L^2(T)$-inner product. Set
$\suvec{v}_T \coloneq (\sbvec{v}_{\cvec{R},T},\bvec{0},(\bvec{0},\bvec{0})_{F\in\FT},(0)_{E\in\ET})\in\sXcurl{T}$
and let $\uvec{v}_T \coloneq \Ecurl{T}\suvec{v}_T$.
We prove \eqref{eq:defect.exactness} by showing that 
\begin{equation}\label{eq:cex.exact}
  \suCT\suvec{v}_T
  \coloneq\Rdiv{T}\uCT\uvec{v}_T
  =\uvec{0}\quad\mbox{ and }\quad \suvec{v}_T\not\in\Image\suGT.
\end{equation}
We first notice that $\CF\uvec{v}_F=0$ for all $F\in\FT$, since all face and edge components of $\uvec{v}_T$ are equal zero.
The link between element and face discrete curls of \cite[Proposition 4]{Di-Pietro.Droniou:21*2} then shows that $\Gproj{k-1}{T}\cCT\uvec{v}_T=\bvec{0}$.
It remains to prove that the component of $\Rdiv{T}\uCT\uvec{v}_T$ on $\cGoly{m_T+1}(T)$ vanishes.
Since $m_T + 1 < k$, this component is $\Gcproj{m_T+1}{T}\cCT\uvec{v}_T$. Since $\uvec{v}_F=\uvec{0}$ for all $F\in\FT$, the definition \eqref{eq:cCT} of $\cCT$ yields
\[
\int_T \cCT\uvec{v}_T\cdot\bvec{z}_T
= \int_T \sbvec{v}_{\cvec{R},T}\cdot\CURL\bvec{z}_T=0\qquad\forall \bvec{z}_T\in\cGoly{m_T+1}(T),
\]
where we have used in the first equality $\bvec{v}_{\cvec{R},T}=\sbvec{v}_{\cvec{R},T}$ (see \eqref{eq:Ecurl.T}) and, to conclude, the orthogonality of $\sbvec{v}_{\cvec{R},T}$ and $\Roly{m_T}(T)$. This proves the first relation in \eqref{eq:cex.exact}.

We prove the second relation by contradiction. Assume that $\suvec{v}_T=\suGT\su{q}_T$ for some $\su{q}_T\in\sXgrad{T}$. This gives, for all $F\in\FT$, $\suGF\su{q}_F=\uvec{v}_F=\uvec{0}$ and thus, by the cochain property \ref{P:E.R.cochain}, 
$\uGF\Egrad{F}\su{q}_F=\Ecurl{F}\suGF\su{q}_F=\uvec{0}$.
Using the definition \eqref{eq:def.hatu} of $\RRoly{T}$ together with $\cCT\uGT=\bvec{0}$ (see \cite[Remark 15]{Di-Pietro.Droniou:21*2}), we deduce that $\RRoly{T}\uGT\Egrad{T}\su{q}_T=\bvec{0}$. Since $\suGT\su{q}_T=\suvec{v}_T$, this condition implies $\sbvec{v}_{\cvec{R},T}=\bvec{0}$, which contradicts our choice of $\sbvec{v}_{\cvec{R},T}$.

\subsection{Homological and analytical properties of the SDDR complex and numerical tests} \label{sec:tests}

Lemmas \ref{lem:properties.sXgrad} and \ref{lem:properties.sXcurl} show that the Diagram \eqref{eq:ser.ddr.seq} (as well as its global counterpart) fulfils Assumptions \ref{ass:cohomology} and \ref{ass:analytical}. By the results in Section \ref{sec:blueprint}, the SDDR complex thus inherits the homological and analytical properties of the DDR complex. Specifically:
\begin{itemize}[left=0pt,parsep=2pt,itemsep=2pt]
\item The SDDR complex has the same cohomology as the de Rham complex, whether the latter is exact (as in the case of domains without holes/tunnels) or has a non-trivial cohomology (\emph{Proposition \ref{prop:cohomologies}, \cite[Section 3]{Di-Pietro.Droniou:21*2} and \cite{Di-Pietro.Droniou.ea:22}});
\item The SDDR complex satisfies Poincar\'e's inequalities (\emph{Proposition \ref{prop:poincare} and \cite[Section 5]{Di-Pietro.Droniou:21*2}});
\item The SDDR complex satisfies primal and adjoint consistency properties of the reconstructed potentials and discrete differential operators (\emph{Propositions \ref{prop:primal.consistency}, \ref{prop:primal.consistency.Pi}, \ref{prop:commutation.ID=dI} and \ref{prop:adjoint.consistency}, and \cite[Section 6]{Di-Pietro.Droniou:21*2}}).
\end{itemize}

As a consequence, the numerical scheme for the magnetostatics problem \eqref{eq:magnetostatics} obtained replacing in \cite[Section 7]{Di-Pietro.Droniou:21*2} the space $\Xcurl{h}$ with its serendipity version $\sXcurl{h}$ has analogous convergence properties as the original one (see \cite[Theorem 12]{Di-Pietro.Droniou:21*2} for a precise statement of an error estimate). The same holds for the scheme for the Stokes problem obtained replacing in \cite[Eq.~(3.14)]{Beirao-da-Veiga.Dassi.ea:22} the spaces $\Xcurl{h}$ and $\Xgrad{h}$ with their serendipity counterparts.
To assess the possible difference in practical accuracy and the gain resulting from the use of the serendipity space, we therefore compare the original DDR and the SDDR schemes on the numerical examples of \cite[Section 7.3]{Di-Pietro.Droniou:21*2} (for the magnetostatics problem) and \cite[Section 5]{Beirao-da-Veiga.Dassi.ea:22} (for the Stokes problem, with pressure scaling $\lambda=1$), for polynomial degrees $k=1,2,3$; no DOF reduction is achieved for the lowest order degree $k=0$, which is therefore removed from our comparison.
All schemes are implemented within the open-source HArDCore3D C++ framework (see \url{https://github.com/jdroniou/HArDCore}), using standard C++ multi-threading routines and linear algebra facilities from the Eigen3 library (see \url{http://eigen.tuxfamily.org}).
All tests were performed on a Dell Precision 5820 desktop with a 14-core Intel Xeon processor (W-2275) clocked at 3.3GHz and equipped with 128Gb of DDR4 RAM, running Ubuntu {20.04.4} LTS.
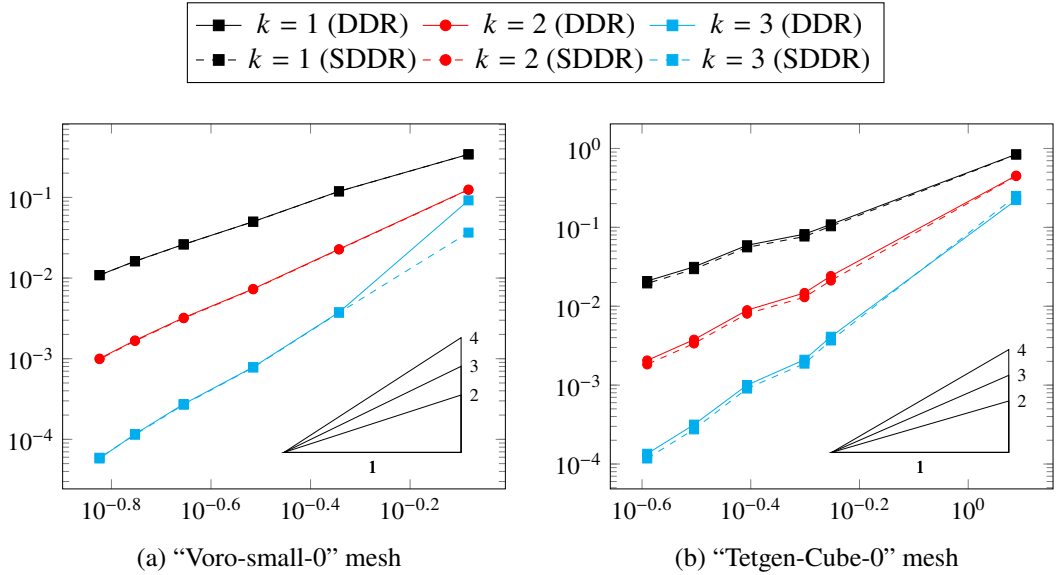
\begin{figure}\centering
  {
    \hypersetup{hidelinks}
    \ref{conv.legend}
  }
  \vspace{0.50cm}\\
  % Voronoi
  \begin{minipage}{0.45\textwidth}
    \begin{tikzpicture}[scale=0.85]
      \begin{loglogaxis}[legend columns=3 ]       
%        \addplot [mark=*, blue] table[x=MeshSize,y=HcurlHdivError] {dat/magnetostatics/ddr/Voro-small-0_0/data_rates.dat};
%        \logLogSlopeTriangle{0.90}{0.4}{0.1}{1}{black};
        \addplot [mark=square*, black] table[x=MeshSize,y=HcurlHdivError] {dat/magnetostatics/ddr/Voro-small-0_1/data_rates.dat};
        \logLogSlopeTriangle{0.90}{0.4}{0.1}{2}{black};
        \addplot [mark=*, red] table[x=MeshSize,y=HcurlHdivError] {dat/magnetostatics/ddr/Voro-small-0_2/data_rates.dat};
        \logLogSlopeTriangle{0.90}{0.4}{0.1}{3}{black};
        \addplot [mark=square*, cyan] table[x=MeshSize,y=HcurlHdivError] {dat/magnetostatics/ddr/Voro-small-0_3/data_rates.dat};
        \logLogSlopeTriangle{0.90}{0.4}{0.1}{4}{black};        
%        \addplot [mark=*, mark options=solid, blue, dashed] table[x=MeshSize,y=HcurlHdivError] {dat/magnetostatics/sddr/Voro-small-0_0/data_rates.dat};
        \addplot [mark=square*, mark options=solid, black, dashed] table[x=MeshSize,y=HcurlHdivError] {dat/magnetostatics/sddr/Voro-small-0_1/data_rates.dat};
        \addplot [mark=*, mark options=solid, red, dashed] table[x=MeshSize,y=HcurlHdivError] {dat/magnetostatics/sddr/Voro-small-0_2/data_rates.dat};
        \addplot [mark=square*, mark options=solid, cyan, dashed] table[x=MeshSize,y=HcurlHdivError] {dat/magnetostatics/sddr/Voro-small-0_3/data_rates.dat};
      \end{loglogaxis}            
    \end{tikzpicture}
    \subcaption{``Voro-small-0'' mesh}
  \end{minipage}
  % Tetgen-Cube-0
  \begin{minipage}{0.45\textwidth}
    \begin{tikzpicture}[scale=0.85]
      \begin{loglogaxis}[legend columns=3, legend to name=conv.legend]       
%        \addplot [mark=*, blue] table[x=MeshSize,y=HcurlHdivError] {dat/magnetostatics/ddr/Tetgen-Cube-0_0/data_rates.dat};
%        \logLogSlopeTriangle{0.90}{0.4}{0.1}{1}{black};
        \addplot [mark=square*, black] table[x=MeshSize,y=HcurlHdivError] {dat/magnetostatics/ddr/Tetgen-Cube-0_1/data_rates.dat};
        \logLogSlopeTriangle{0.90}{0.4}{0.1}{2}{black};
        \addplot [mark=*, red] table[x=MeshSize,y=HcurlHdivError] {dat/magnetostatics/ddr/Tetgen-Cube-0_2/data_rates.dat};
        \logLogSlopeTriangle{0.90}{0.4}{0.1}{3}{black};
        \addplot [mark=square*, cyan] table[x=MeshSize,y=HcurlHdivError] {dat/magnetostatics/ddr/Tetgen-Cube-0_3/data_rates.dat};
        \logLogSlopeTriangle{0.90}{0.4}{0.1}{4}{black};        
%        \addplot [mark=*, mark options=solid, blue, dashed] table[x=MeshSize,y=HcurlHdivError] {dat/magnetostatics/sddr/Tetgen-Cube-0_0/data_rates.dat};
        \addplot [mark=square*, mark options=solid, black, dashed] table[x=MeshSize,y=HcurlHdivError] {dat/magnetostatics/sddr/Tetgen-Cube-0_1/data_rates.dat};
        \addplot [mark=*, mark options=solid, red, dashed] table[x=MeshSize,y=HcurlHdivError] {dat/magnetostatics/sddr/Tetgen-Cube-0_2/data_rates.dat};
        \addplot [mark=square*, mark options=solid, cyan, dashed] table[x=MeshSize,y=HcurlHdivError] {dat/magnetostatics/sddr/Tetgen-Cube-0_3/data_rates.dat};
%        \legend{$k=0$ (DDR), $k=1$ (DDR), $k=2$ (DDR), $k=3$ (DDR), $k=0$ (SDDR), $k=1$ (SDDR), $k=2$ (SDDR), $k=3$ (SDDR)};
        \legend{$k=1$ (DDR), $k=2$ (DDR), $k=3$ (DDR), $k=1$ (SDDR), $k=2$ (SDDR), $k=3$ (SDDR)};
      \end{loglogaxis}                     
    \end{tikzpicture}
    \subcaption{``Tetgen-Cube-0'' mesh}
  \end{minipage}
  \caption{Relative errors in the discrete $\Hcurl{\Omega}\times\Hdiv{\Omega}$ norm vs.\ $h$, for the standard DDR scheme (continuous lines), and the SDDR scheme (dashed lines) applied to the magnetostatic problem. \label{fig:conv}}
\end{figure}

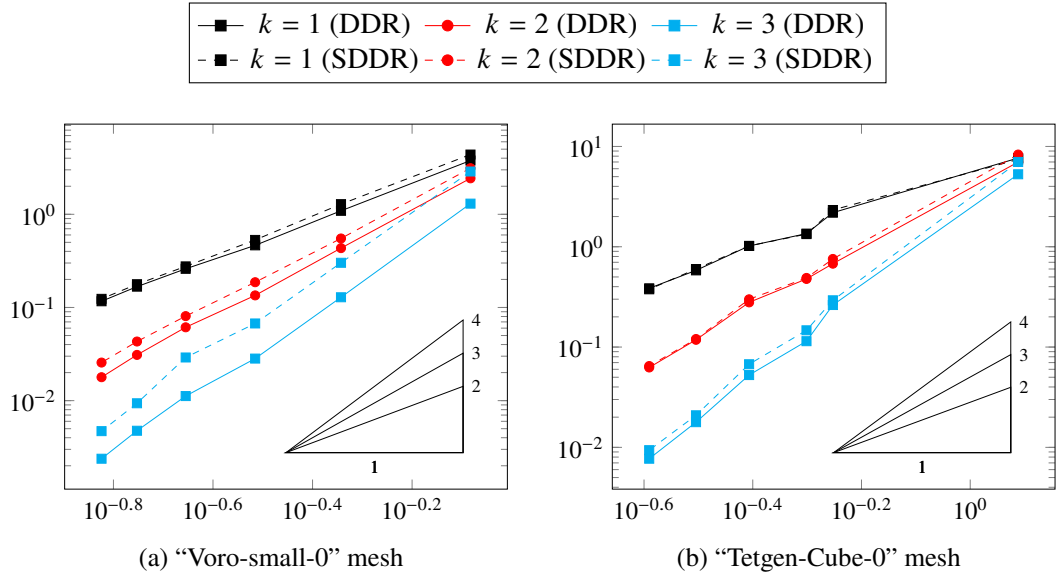
\begin{figure}\centering
  {
    \hypersetup{hidelinks}
    \ref{conv.stokes.legend}
  }
  \vspace{0.50cm}\\
  % Voronoi
  \begin{minipage}{0.45\textwidth}
    \begin{tikzpicture}[scale=0.85]
      \begin{loglogaxis}[legend columns=3 ]       
       \addplot [mark=square*, black] table[x=MeshSize,y expr=\thisrow{E_cHcurlVel}+\thisrow{E_cL2GradPre}] {dat/stokes/ddr/Voro-small-0_1/data_rates.dat};
        \logLogSlopeTriangle{0.90}{0.4}{0.1}{2}{black};
        \addplot [mark=*, red] table[x=MeshSize,y expr=\thisrow{E_cHcurlVel}+\thisrow{E_cL2GradPre}] {dat/stokes/ddr/Voro-small-0_2/data_rates.dat};
        \logLogSlopeTriangle{0.90}{0.4}{0.1}{3}{black};
        \addplot [mark=square*, cyan] table[x=MeshSize,y expr=\thisrow{E_cHcurlVel}+\thisrow{E_cL2GradPre}] {dat/stokes/ddr/Voro-small-0_3/data_rates.dat};
        \logLogSlopeTriangle{0.90}{0.4}{0.1}{4}{black};        
        \addplot [mark=square*, mark options=solid, black, dashed] table[x=MeshSize,y expr=\thisrow{E_cHcurlVel}+\thisrow{E_cL2GradPre}] {dat/stokes/sddr/Voro-small-0_1/data_rates.dat};
        \addplot [mark=*, mark options=solid, red, dashed] table[x=MeshSize,y expr=\thisrow{E_cHcurlVel}+\thisrow{E_cL2GradPre}] {dat/stokes/sddr/Voro-small-0_2/data_rates.dat};
        \addplot [mark=square*, mark options=solid, cyan, dashed] table[x=MeshSize,y expr=\thisrow{E_cHcurlVel}+\thisrow{E_cL2GradPre}] {dat/stokes/sddr/Voro-small-0_3/data_rates.dat};
      \end{loglogaxis}            
    \end{tikzpicture}
    \subcaption{``Voro-small-0'' mesh}
  \end{minipage}
  % Tetgen-Cube-0
  \begin{minipage}{0.45\textwidth}
    \begin{tikzpicture}[scale=0.85]
      \begin{loglogaxis}[legend columns=3, legend to name=conv.stokes.legend]       
        \addplot [mark=square*, black] table[x=MeshSize,y expr=\thisrow{E_cHcurlVel}+\thisrow{E_cL2GradPre}] {dat/stokes/ddr/Tetgen-Cube-0_1/data_rates.dat};
        \logLogSlopeTriangle{0.90}{0.4}{0.1}{2}{black};
        \addplot [mark=*, red] table[x=MeshSize,y expr=\thisrow{E_cHcurlVel}+\thisrow{E_cL2GradPre}] {dat/stokes/ddr/Tetgen-Cube-0_2/data_rates.dat};
        \logLogSlopeTriangle{0.90}{0.4}{0.1}{3}{black};
        \addplot [mark=square*, cyan] table[x=MeshSize,y expr=\thisrow{E_cHcurlVel}+\thisrow{E_cL2GradPre}] {dat/stokes/ddr/Tetgen-Cube-0_3/data_rates.dat};
        \logLogSlopeTriangle{0.90}{0.4}{0.1}{4}{black};        
        \addplot [mark=square*, mark options=solid, black, dashed] table[x=MeshSize,y expr=\thisrow{E_cHcurlVel}+\thisrow{E_cL2GradPre}] {dat/stokes/sddr/Tetgen-Cube-0_1/data_rates.dat};
        \addplot [mark=*, mark options=solid, red, dashed] table[x=MeshSize,y expr=\thisrow{E_cHcurlVel}+\thisrow{E_cL2GradPre}] {dat/stokes/sddr/Tetgen-Cube-0_2/data_rates.dat};
        \addplot [mark=square*, mark options=solid, cyan, dashed] table[x=MeshSize,y expr=\thisrow{E_cHcurlVel}+\thisrow{E_cL2GradPre}] {dat/stokes/sddr/Tetgen-Cube-0_3/data_rates.dat};
        \legend{$k=1$ (DDR), $k=2$ (DDR), $k=3$ (DDR), $k=1$ (SDDR), $k=2$ (SDDR), $k=3$ (SDDR)};
      \end{loglogaxis}                     
    \end{tikzpicture}
    \subcaption{``Tetgen-Cube-0'' mesh}
  \end{minipage}
  \caption{Relative errors in the $\Hcurl{\Omega}\times L^2(\Omega)^d$ norm (for the couple velocity--gradient of the pressure) vs.\ $h$, for the standard DDR scheme (continuous lines), and the SDDR scheme (dashed lines) applied to the Stokes problem. \label{fig:conv.stokes}}
\end{figure}

A comparison of the corresponding discrete errors as functions of the mesh size for the ``Tetgen-Cube-0'' (matching simplicial) and ``Voro-small-0'' (Voronoi) mesh families available in the HArDCore3D repository is provided in Figs.~\ref{fig:conv} and \ref{fig:conv.stokes}, showing that the DDR and SDDR achieve optimal rates of convergence, irrespective of the considered mesh, polynomial degree or problem. The accuracies of both methods are almost identical for the magnetostatics problem. For the Stokes problem, SDDR appears to be slightly less accurate that DDR on some meshes (by a maximum factor of about 2-3); a more precise analysis of each error component shows that the loss of accuracy only occurs on the pressure variable, which can be explained by the fact that $\sXgrad{h}$ has about 50\% fewer DOFs than $\Xgrad{h}$ on Voronoi meshes.

%------------------------------------------------------------------------------%
% CPU time
%------------------------------------------------------------------------------%

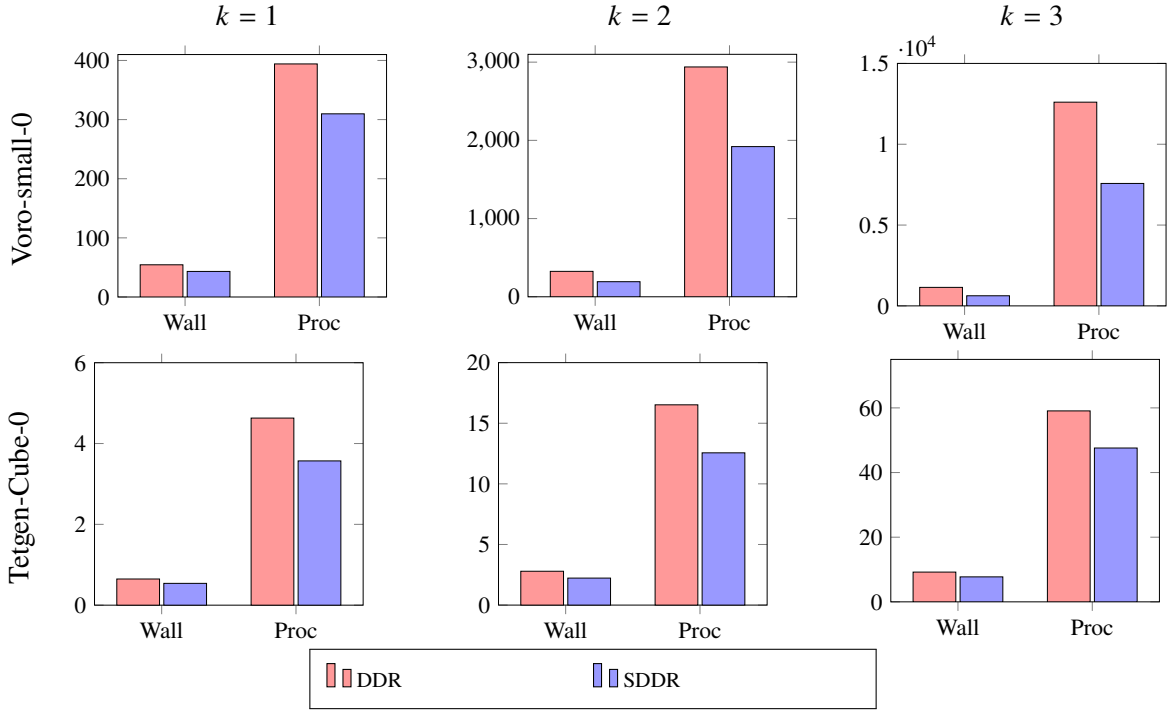
\begin{figure}\centering
  \begin{tabular}{cccc}
    %% Table header
    & $k=1$ & $k=2$ & $k=3$ \\
    \rotatebox[origin=c]{90}{Voro-small-0}
    &
    %% k=1
    \begin{minipage}{0.3\textwidth}
      \begin{tikzpicture}[scale=0.80]
        \begin{axis}[
            ybar,
            bar width=20pt,
            xticklabels={Wall,Proc},
            xtick=data,
            enlarge x limits=0.5,
            ymin=0,
            ymax=410,
            height=5.6cm,
            width=6cm,
            legend style={
              legend columns=-1,
              draw=none,
              font=\scriptsize
            },
            legend to name=leg:times:ddr
          ]
          \addplot[fill=red!40] table[x expr=\coordindex,y=Solve] {dat/magnetostatics/ddr/Voro-small-0_1/wall-proc-times.dat};
          \addplot[fill=blue!40] table[x expr=\coordindex,y=Solve] {dat/magnetostatics/sddr/Voro-small-0_1/wall-proc-times.dat};
          \legend{\parbox{24ex}{DDR},\parbox{24ex}{SDDR}}
        \end{axis}
      \end{tikzpicture}
    \end{minipage}
    &
    %% k=2
    \begin{minipage}{0.3\textwidth}
      \begin{tikzpicture}[scale=0.80]
        \begin{axis}[
            ybar,
            bar width=20pt,
            xticklabels={Wall,Proc},
            xtick=data,
            enlarge x limits=0.5,
            ymin=0,
            ymax=3100,
            height=5.6cm,
            width=6cm
          ]
          \addplot[fill=red!40] table[x expr=\coordindex,y=Solve] {dat/magnetostatics/ddr/Voro-small-0_2/wall-proc-times.dat};
          \addplot[fill=blue!40] table[x expr=\coordindex,y=Solve] {dat/magnetostatics/sddr/Voro-small-0_2/wall-proc-times.dat};
        \end{axis}
      \end{tikzpicture}
    \end{minipage}
    &
    %% k=3
    \begin{minipage}{0.3\textwidth}
      \begin{tikzpicture}[scale=0.80]
        \begin{axis}[
            ybar,
            bar width=20pt,
            xticklabels={Wall,Proc},
            xtick=data,
            enlarge x limits=0.5,
            ymin=0,
            ymax=15000,
            height=5.6cm,
            width=6cm
          ]
          \addplot[fill=red!40] table[x expr=\coordindex,y=Solve] {dat/magnetostatics/ddr/Voro-small-0_3/wall-proc-times.dat};
          \addplot[fill=blue!40] table[x expr=\coordindex,y=Solve] {dat/magnetostatics/sddr/Voro-small-0_3/wall-proc-times.dat};
        \end{axis}
      \end{tikzpicture}
    \end{minipage}
    \\
    \rotatebox[origin=c]{90}{Tetgen-Cube-0}
    &
    %% Tet, k = 1
    \begin{minipage}{0.3\textwidth}
      \begin{tikzpicture}[scale=0.80]
        \begin{axis}[
            ybar,
            bar width=20pt,
            xticklabels={Wall,Proc},
            xtick=data,
            enlarge x limits=0.5,
            ymin=0,
            ymax=6,
            height=5.6cm,
            width=6cm
          ]
          \addplot[fill=red!40] table[x expr=\coordindex,y=Solve] {dat/magnetostatics/ddr/Tetgen-Cube-0_1/wall-proc-times.dat};
          \addplot[fill=blue!40] table[x expr=\coordindex,y=Solve] {dat/magnetostatics/sddr/Tetgen-Cube-0_1/wall-proc-times.dat};
        \end{axis}
      \end{tikzpicture}
    \end{minipage}
    &
    %% Tet, k = 2
    \begin{minipage}{0.3\textwidth}
      \begin{tikzpicture}[scale=0.80]
        \begin{axis}[
            ybar,
            bar width=20pt,
            xticklabels={Wall,Proc},
            xtick=data,
            enlarge x limits=0.5,
            ymin=0,
            ymax=20,
            height=5.6cm,
            width=6cm
          ]
          \addplot[fill=red!40] table[x expr=\coordindex,y=Solve] {dat/magnetostatics/ddr/Tetgen-Cube-0_2/wall-proc-times.dat};
          \addplot[fill=blue!40] table[x expr=\coordindex,y=Solve] {dat/magnetostatics/sddr/Tetgen-Cube-0_2/wall-proc-times.dat};
        \end{axis}
      \end{tikzpicture}
    \end{minipage}
    &
    %% Tet, k = 3
    \begin{minipage}{0.3\textwidth}
      \begin{tikzpicture}[scale=0.80]
        \begin{axis}[
            ybar,
            bar width=20pt,
            xticklabels={Wall,Proc},
            xtick=data,
            enlarge x limits=0.5,
            ymin=0,
            ymax=75,
            height=5.6cm,
            width=6cm
          ]
          \addplot[fill=red!40] table[x expr=\coordindex,y=Solve] {dat/magnetostatics/ddr/Tetgen-Cube-0_3/wall-proc-times.dat};
          \addplot[fill=blue!40] table[x expr=\coordindex,y=Solve] {dat/magnetostatics/sddr/Tetgen-Cube-0_3/wall-proc-times.dat};
        \end{axis}
      \end{tikzpicture}
    \end{minipage}
  \end{tabular}  
          \centerline{\fbox{\hypersetup{hidelinks}\ref{leg:times:ddr}}}
      \caption{Wall and processor times (in seconds) for the resolution of the linear system for the DDR and SDDR methods for the magnetostatics problem.
        Times correspond to the finest mesh of each sequence.}\label{fig:times}
\end{figure}

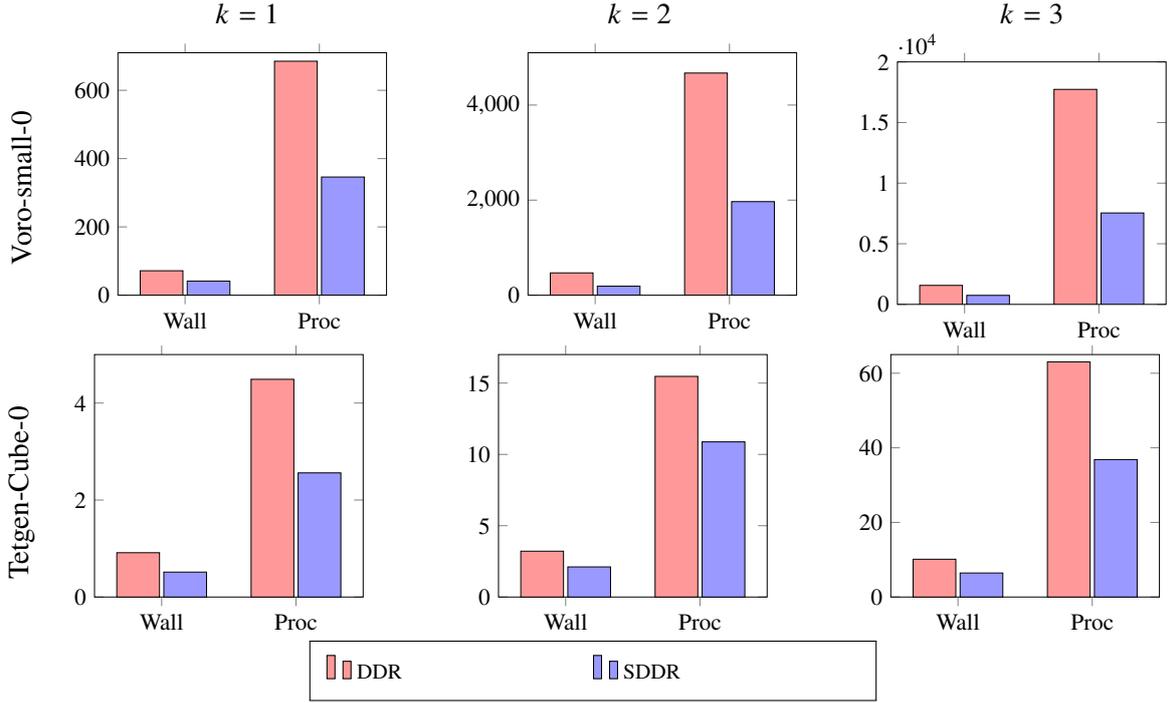
\begin{figure}\centering
  \begin{tabular}{cccc}
    %% Table header
    & $k=1$ & $k=2$ & $k=3$ \\
    \rotatebox[origin=c]{90}{Voro-small-0}
    &
    %% k=1
    \begin{minipage}{0.3\textwidth}
      \begin{tikzpicture}[scale=0.80]
        \begin{axis}[
            ybar,
            bar width=20pt,
            xticklabels={Wall,Proc},
            xtick=data,
            enlarge x limits=0.5,
            ymin=0,
            ymax=710,
            height=5.6cm,
            width=6cm,
            legend style={
              legend columns=-1,
              draw=none,
              font=\scriptsize
            },
            legend to name=leg:times:ddr.stokes
          ]
          \addplot[fill=red!40] table[x expr=\coordindex,y=Solve] {dat/stokes/ddr/Voro-small-0_1/wall-proc-times.dat};
          \addplot[fill=blue!40] table[x expr=\coordindex,y=Solve] {dat/stokes/sddr/Voro-small-0_1/wall-proc-times.dat};
          \legend{\parbox{24ex}{DDR},\parbox{24ex}{SDDR}}
        \end{axis}
      \end{tikzpicture}
    \end{minipage}
    &
    %% k=2
    \begin{minipage}{0.3\textwidth}
      \begin{tikzpicture}[scale=0.80]
        \begin{axis}[
            ybar,
            bar width=20pt,
            xticklabels={Wall,Proc},
            xtick=data,
            enlarge x limits=0.5,
            ymin=0,
            ymax=5100,
            height=5.6cm,
            width=6cm
          ]
          \addplot[fill=red!40] table[x expr=\coordindex,y=Solve] {dat/stokes/ddr/Voro-small-0_2/wall-proc-times.dat};
          \addplot[fill=blue!40] table[x expr=\coordindex,y=Solve] {dat/stokes/sddr/Voro-small-0_2/wall-proc-times.dat};
        \end{axis}
      \end{tikzpicture}
    \end{minipage}
    &
    %% k=3
    \begin{minipage}{0.3\textwidth}
      \begin{tikzpicture}[scale=0.80]
        \begin{axis}[
            ybar,
            bar width=20pt,
            xticklabels={Wall,Proc},
            xtick=data,
            enlarge x limits=0.5,
            ymin=0,
            ymax=20000,
            height=5.6cm,
            width=6cm
          ]
          \addplot[fill=red!40] table[x expr=\coordindex,y=Solve] {dat/stokes/ddr/Voro-small-0_3/wall-proc-times.dat};
          \addplot[fill=blue!40] table[x expr=\coordindex,y=Solve] {dat/stokes/sddr/Voro-small-0_3/wall-proc-times.dat};
        \end{axis}
      \end{tikzpicture}
    \end{minipage}
    \\
    \rotatebox[origin=c]{90}{Tetgen-Cube-0}
    &
    %% Tet, k = 1
    \begin{minipage}{0.3\textwidth}
      \begin{tikzpicture}[scale=0.80]
        \begin{axis}[
            ybar,
            bar width=20pt,
            xticklabels={Wall,Proc},
            xtick=data,
            enlarge x limits=0.5,
            ymin=0,
            ymax=5,
            height=5.6cm,
            width=6cm
          ]
          \addplot[fill=red!40] table[x expr=\coordindex,y=Solve] {dat/stokes/ddr/Tetgen-Cube-0_1/wall-proc-times.dat};
          \addplot[fill=blue!40] table[x expr=\coordindex,y=Solve] {dat/stokes/sddr/Tetgen-Cube-0_1/wall-proc-times.dat};
        \end{axis}
      \end{tikzpicture}
    \end{minipage}
    &
    %% Tet, k = 2
    \begin{minipage}{0.3\textwidth}
      \begin{tikzpicture}[scale=0.80]
        \begin{axis}[
            ybar,
            bar width=20pt,
            xticklabels={Wall,Proc},
            xtick=data,
            enlarge x limits=0.5,
            ymin=0,
            ymax=17,
            height=5.6cm,
            width=6cm
          ]
          \addplot[fill=red!40] table[x expr=\coordindex,y=Solve] {dat/stokes/ddr/Tetgen-Cube-0_2/wall-proc-times.dat};
          \addplot[fill=blue!40] table[x expr=\coordindex,y=Solve] {dat/stokes/sddr/Tetgen-Cube-0_2/wall-proc-times.dat};
        \end{axis}
      \end{tikzpicture}
    \end{minipage}
    &
    %% Tet, k = 3
    \begin{minipage}{0.3\textwidth}
      \begin{tikzpicture}[scale=0.80]
        \begin{axis}[
            ybar,
            bar width=20pt,
            xticklabels={Wall,Proc},
            xtick=data,
            enlarge x limits=0.5,
            ymin=0,
            ymax=65,
            height=5.6cm,
            width=6cm
          ]
          \addplot[fill=red!40] table[x expr=\coordindex,y=Solve] {dat/stokes/ddr/Tetgen-Cube-0_3/wall-proc-times.dat};
          \addplot[fill=blue!40] table[x expr=\coordindex,y=Solve] {dat/stokes/sddr/Tetgen-Cube-0_3/wall-proc-times.dat};
        \end{axis}
      \end{tikzpicture}
    \end{minipage}
  \end{tabular}  
          \centerline{\fbox{\hypersetup{hidelinks}\ref{leg:times:ddr.stokes}}}
      \caption{Wall and processor times (in seconds) for the resolution of the linear system for the DDR and SDDR methods for the Stokes problem.
        Times correspond to the finest mesh of each sequence.}\label{fig:times.stokes}
\end{figure}

The (processor and wall) CPU times for the resolution the linear systems (after static condensation of the remaining element unknowns) through the \texttt{MKL PARDISO} solver (see \url{https://software.intel.com/en-us/mkl}) are depicted in Figs.~\ref{fig:times} and \ref{fig:times.stokes}.
The observed gain is significant: for the magnetostatics problem and $k=3$ on the finest ``Voro-small-0'' mesh, a 40\% reduction of the solution processor time is observed for the SDDR scheme, translating into a 45\% reduction of the wall time (626s vs.~1145s). The gain is even more pronounced for the Stokes problem, with a reduction of about 57\% in processor time and 52\% in wall time (742s vs.~1566s), which again can be explained by the fact that the serendipity process is more efficient on the discrete $H^1$ space, used in the scheme for the Stokes problem but not for the magnetostatics problem.
A comparison of the assembly wall and processor times, not reported here for the sake of conciseness, also confirms that, being an embarrassingly parallel step, assembly can fully benefit from shared-memory parallelism and that the additional calculations required for the serendipity spaces form a negligible portion of that assembly time.

%------------------------------------------------------------------------------%
% Appendix
%------------------------------------------------------------------------------%

\appendix

\section{Estimates of polynomials from boundary values}\label{sec:estimates.proofs}

\subsection{Proof of Lemma \ref{lem:reconstruction}}\label{sec:estimates.proofs:reconstruction}

Let $q\in\Poly{k+1}(\sfP)$ and let us list the elements of $\sbdry{\sfP}$ as $\sfb_1,\ldots,\sfb_{\eta_\sfP}$.
We will construct by induction polynomials $(q_i)_{i=1,\ldots,\eta_\sfP}$ and $(\widetilde{q}_i)_{i=1,\ldots,\eta_\sfP}$ such that, for all $i=1,\ldots,\eta_\sfP$,
\begin{equation}\label{eq:decomposition.q}
\begin{aligned}
&q=q_1+\dist_{\sfP\sfb_1}q_2+\cdots+(\dist_{\sfP\sfb_1}\cdots\dist_{\sfP\sfb_{i-1}})q_i+(\dist_{\sfP\sfb_1}\cdots\dist_{\sfP\sfb_i})\widetilde{q}_i,\\
&q_i\in\Poly{k+1-(i-1)}(\sfP)\,,\quad\widetilde{q}_i\in\Poly{k+1-i}(\sfP)\,,\quad \norm[\sfP]{q_i}\lesssim h_\sfP^{\nicefrac12}(\norm[\sfb_1]{q_{|\sfb_1}}+\cdots+\norm[\sfb_i]{q_{|\sfb_i}}).
\end{aligned}
\end{equation}

Let us first consider $i=1$. We can easily extend $q_{|\sfb_1}$ (e.g.~by making it independent of the normal coordinate to $\sfb_1$) as a polynomial $q_1\in\Poly{k+1}(\sfP)$ such that $\norm[\sfP]{q_1}\lesssim h_\sfP^{\nicefrac12}\norm[\sfb_1]{q_{|\sfb_1}}$.
Then, $q-q_1\in\Poly{k+1}(\sfP)$ vanishes on $\sfb_1$, and can thus be factorised by $\dist_{\sfP\sfb_1}\in\Poly{1}(\sfP)$: there is $\widetilde{q}_1\in\Poly{k}(\sfP)$ such that $q-q_1=\dist_{\sfP\sfb_1}\widetilde{q}_1$. This concludes the proof of \eqref{eq:decomposition.q} for $i=1$.

Let us now take $j\le \eta_\sfP-1$ and show that this relation holds for $i=j+1$ if it holds for $i=1,\ldots,j$. Since 
\begin{equation}\label{eq:est.dist}
|\dist_{\sfP\sfb_r}|\lesssim 1\qquad\forall r=1,\ldots,\eta_\sfP,
\end{equation}
we have, by triangle inequality and the equality in \eqref{eq:decomposition.q} for $i=j$,
\begin{align}
\norm[\sfb_{j+1}]{(\dist_{\sfP\sfb_1}\cdots\dist_{\sfP\sfb_j})_{|\sfb_{j+1}}(\widetilde{q}_j)_{|\sfb_{j+1}}}\le{}&
\norm[\sfb_{j+1}]{q_{|\sfb_{j+1}}}+\norm[\sfb_{j+1}]{(q_1)_{|\sfb_{j+1}}}+\cdots+\norm[\sfb_{j+1}]{(q_j)_{|\sfb_{j+1}}}\nonumber\\
\lesssim{}&\norm[\sfb_{j+1}]{q_{|\sfb_{j+1}}}+\norm[\sfb_1]{q_{|\sfb_1}}+\cdots+\norm[\sfb_j]{q_{|\sfb_j}}
\label{eq:est.dec.q.1}
\end{align}
where the second line is obtained using discrete trace inequalities and the estimates in \eqref{eq:decomposition.q} for $i=1,\ldots,j$.
Assumption \ref{assum:choice.detaP} ensures that $0\le (\dist_{\sfP\sfb_1}\cdots\dist_{\sfP\sfb_j})_{|\sfb_{j+1}}\lesssim 1$, and that $(\dist_{\sfP\sfb_1}\cdots\dist_{\sfP\sfb_j})_{|\sfb_{j+1}}\gtrsim 1$ on a ball in $\sfb_{j+1}$ of radius $\gtrsim h_{\sfb_{j+1}}$. The arguments in the proof of \cite[Lemma 1.25]{Di-Pietro.Droniou:20} then show that the norms $\norm[\sfb_{j+1}]{{\cdot}}$ and $\norm[\sfb_{j+1}]{(\dist_{\sfP\sfb_1}\cdots\dist_{\sfP\sfb_j})_{|\sfb_{j+1}}{\cdot}}$ are equivalent on $\Poly{k+1-j}(\sfb_{j+1})$, and we infer from \eqref{eq:est.dec.q.1} that
\[
\norm[\sfb_{j+1}]{(\widetilde{q}_j)_{|\sfb_{j+1}}}\lesssim \norm[\sfb_1]{q_{|\sfb_1}}+\cdots+\norm[\sfb_j]{q_{|\sfb_j}}+\norm[\sfb_{j+1}]{q_{|\sfb_{j+1}}}.
\]
We then extend $(\widetilde{q}_j)_{|\sfb_{j+1}}$ into $q_{j+1}\in\Poly{k+1-j}(\sfP)$ such that
\[
\norm[\sfP]{q_{j+1}}\lesssim h_{\sfP}^{\nicefrac12}\norm[\sfb_{j+1}]{(\widetilde{q}_j)_{|\sfb_{j+1}}}\lesssim
h_{\sfP}^{\nicefrac12}(\norm[\sfb_1]{q_{|\sfb_1}}+\cdots+\norm[\sfb_{j+1}]{q_{|\sfb_{j+1}}})
\]
and we note that, since $\widetilde{q}_j-q_{j+1}\in\Poly{k+1-j}(\sfP)$ vanishes on $\sfb_{j+1}$, we can factorise this polynomial
into $\widetilde{q}_j-q_{j+1}=\dist_{\sfP\sfb_{j+1}}\widetilde{q}_{j+1}$ with $\widetilde{q}_{j+1}\in\Poly{k+1-j-1}(\sfP)$, which concludes
the proof of \eqref{eq:decomposition.q} for $i=j+1$.

To conclude the proof of the lemma, we consider \eqref{eq:decomposition.q} with $i=\eta_\sfP$. A triangle inequality and the estimate stated in this relation, together with \eqref{eq:est.dist}, give
\begin{equation}\label{eq:est.dec.q}
\norm[\sfP]{q}\lesssim h_\sfP^{\nicefrac12}(\norm[\sfb_1]{q_{|\sfb_1}}+\cdots+\norm[\sfb_i]{q_{|\sfb_{\eta_\sfP}}})+\norm[\sfP]{\widetilde{q}_{\eta_\sfP}}.
\end{equation}
Multiplying the equality in \eqref{eq:decomposition.q} by $\widetilde{q}_{\eta_\sfP}$ and integrating over $\sfP$ yields
\begin{align*}
\int_\sfP (\dist_{\sfP\sfb_1}\cdots\dist_{\sfP\sfb_{\eta_\sfP}})\widetilde{q}_{\eta_\sfP}^2={}&
  \int_\sfP q\widetilde{q}_{\eta_\sfP}-\int_\sfP q_1\widetilde{q}_{\eta_\sfP}-\cdots-\int_\sfP(\dist_{\sfP\sfb_1}\cdots\dist_{\sfP\sfb_{\eta_\sfP-1}})q_{\eta_\sfP}\widetilde{q}_{\eta_\sfP}\\
={}&
\int_\sfP (\lproj{k+1-\eta_\sfP}{\sfP}q)\widetilde{q}_{\eta_\sfP}-\int_\sfP q_1\widetilde{q}_{\eta_\sfP}-\cdots-\int_\sfP(\dist_{\sfP\sfb_1}\cdots\dist_{\sfP\sfb_{\eta_\sfP-1}})q_{\eta_\sfP}\widetilde{q}_{\eta_\sfP},
\end{align*}
where the introduction of the projector in the second line is justified by $\widetilde{q}_{\eta_\sfP}\in\Poly{k+1-\eta_\sfP}(\sfP)$. Since $\dist_{\sfP\sfb_1}\cdots\dist_{\sfP\sfb_{\eta_\sfP}}$
is nonnegative and, by mesh regularity assumption, $\gtrsim 1$ on a ball in $\sfP$ of radius $\gtrsim h_{\sfP}$, we can use
the arguments in the proof of \cite[Lemma 1.25]{Di-Pietro.Droniou:20}, \eqref{eq:est.dist} and Cauchy--Schwarz inequalities to get
\[
\norm[\sfP]{\widetilde{q}_{\eta_\sfP}}^2\lesssim
\int_\sfP (\dist_{\sfP\sfb_1}\cdots\dist_{\sfP\sfb_{\eta_\sfP}})\widetilde{q}_{\eta_\sfP}^2\\
\lesssim \norm[\sfP]{\lproj{k+1-\eta_\sfP}{\sfP}q}\norm[\sfP]{\widetilde{q}_{\eta_\sfP}}+\norm[\sfP]{q_1}\norm[\sfP]{\widetilde{q}_{\eta_\sfP}}
+\cdots+\norm[\sfP]{q_{\eta_\sfP}}\norm[\sfP]{\widetilde{q}_{\eta_\sfP}}.
\]
Simplifying by $\norm[\sfP]{\widetilde{q}_{\eta_\sfP}}$ gives an upper bound on this quantity which, plugged into \eqref{eq:est.dec.q} together with the estimates on $\norm[\sfP]{q_i}$ stated in \eqref{eq:decomposition.q} for $i=1,\ldots,\eta_\sfP$, concludes the proof of \eqref{eq:est.reconstruction}.

\subsection{Proof of Lemma \ref{lem:reconstruction.vpoly}}\label{sec:estimates.proofs:reconstruction.vpoly}

We focus on the case $\sfP = T\in\Th$, the case $\sfP = F\in\Fh$ being similar.
Let $\bvec{v}\in\vPoly{k}(T)$. We first prove that
\begin{equation}\label{eq:est.curlv}
h_T\norm[T]{\CURL\bvec{v}}\lesssim \norm[T]{\Rproj{k-2}{T}\bvec{v}}+\left(\sum_{F\in\FT} h_{T}\norm[F]{\bvec{v}_{{\rm t},F}}^2\right)^{\nicefrac12}.
\end{equation}
For all $\bvec{w}_T\in\vPoly{k-1}(T)$, integrating by parts and introducing $\Rproj{k-2}{T}$ (owing to $\CURL\bvec{w}_T\in\Roly{k-2}(T)$) gives
\[
  \int_T \CURL\bvec{v} \cdot\bvec{w}_T
  =
  \int_T \Rproj{k-2}{T}\bvec{v}\cdot\CURL\bvec{w}_T
  + \sum_{F\in\FT}\omega_{TF}\int_F \bvec{v}_{{\rm t},F}\cdot(\bvec{w}_T\times\normal_F).
\]
Making $\bvec{w}_T = h_T\CURL\bvec{v}$, using Cauchy--Schwarz and discrete inverse and trace inequalities, and simplifying leads to \eqref{eq:est.curlv}.

We now turn to the estimate on $\bvec{v}$, which we decompose as $\bvec{v}=\GRAD q+\bvec{z}$ with $(q,\bvec{z})\in\Poly{k+1}(T)\times\cGoly{k}(T)$ (see \eqref{eq:vPoly=Goly+cGoly}).
Since $\CURL\GRAD=0$, we have $\CURL\bvec{v}=\CURL\bvec{z}$ and, using the isomorphism estimate of $\CURL:\cGoly{k}(T)\to\Roly{k-1}(T)$ (see \cite[Lemma 9]{Di-Pietro.Droniou:21*2}), \eqref{eq:est.curlv} gives
\begin{equation}\label{est:bvecv.cGoly}
\norm[T]{\bvec{z}}\lesssim h_T\norm[T]{\CURL\bvec{z}}\lesssim \norm[T]{\Rproj{k-2}{T}\bvec{v}}+\left(\sum_{F\in\FT} h_{T}\norm[F]{\bvec{v}_{{\rm t},F}}^2\right)^{\nicefrac12}.
\end{equation}
Upon adding a constant to $q$, we can assume that $q$ has a zero average on $\partial T$; a Poincar\'e--Wirtinger inequality along this boundary, which can be obtained gluing Poincar\'e--Wirtinger inequalities on each face (following similar ideas as in \cite[Proof of Lemma 7]{Di-Pietro.Droniou:21*2}, but in a simpler way since $q_{|\partial T}$ is here continuous across the edges), then gives $\norm[\partial T]{q}\lesssim h_{T}\sum_{F\in\FT}\norm[F]{(\GRAD q)_{{\rm t},F}}$.
Since $(\GRAD q)_{{\rm t},F}=\bvec{v}_{{\rm t},F}-\bvec{z}_{{\rm t},F}$, the estimate \eqref{est:bvecv.cGoly} and discrete trace inequalities lead to
\begin{equation}\label{est:bvecv.qbdry}
\norm[\partial T]{q}\lesssim h_{T}\sum_{F\in\FT}\norm[F]{\bvec{v}_{{\rm t},F}}
+h_{T}\sum_{F\in\FT}\norm[F]{\bvec{z}_{{\rm t},F}}
\lesssim h_{T}^{\nicefrac12}\norm[T]{\Rproj{k-2}{T}\bvec{v}}+h_{T}^{\nicefrac12}\left(\sum_{F\in\FT} h_{T}\norm[F]{\bvec{v}_{{\rm t},F}}^2\right)^{\nicefrac12}.
\end{equation}
We now estimate $\lproj{k+1-\eta_T}{T}q$. For all $\bvec{w}_T\in\cRoly{k+2-\eta_T}(T)$, recalling that $\GRAD q=\bvec{v}-\bvec{z}$ we have
\[
\begin{aligned}
  \int_T q\DIV\bvec{w}_T
  &=
  -\int_T(\bvec{v}-\bvec{z})\cdot\bvec{w}_{T}
  + \sum_{F\in\FT}\omega_{TF}\int_F q~(\bvec{w}_{T}\cdot\normal_{F})
  \\
  &=
  -\int_T(\Rcproj{k+2-\eta_T}{T}\bvec{v}-\bvec{z})\cdot\bvec{w}_{T}
  + \sum_{F\in\FT}\omega_{TF}\int_F q~(\bvec{w}_{T}\cdot\normal_{F}).
\end{aligned}
\]
We now use Cauchy--Schwarz inequalities, discrete trace inequalities, \eqref{est:bvecv.qbdry}, and \eqref{est:bvecv.cGoly} to deduce
\[
\int_T q\DIV\bvec{w}_T\lesssim \norm[T]{\Rcproj{k+2-\eta_T}{T}\bvec{v}}\norm[T]{\bvec{w}_T}
+\left[
  \norm[T]{\Rproj{k-2}{T}\bvec{v}}+\left(\sum_{F\in\FT} h_{T}\norm[F]{\bvec{v}_{{\rm t},F}}^2\right)^{\nicefrac12}
  \right]\norm[T]{\bvec{w}_{T}}.
\]
This estimate and the isomorphism bound of $\DIV:\cRoly{k+2-\eta_T}(T)\to\Poly{k+1-\eta_T}(T)$ stated in \cite[Lemma 9]{Di-Pietro.Droniou:21*2} yield
\[
\norm[T]{\lproj{k+1-\eta_T}{T}q}\lesssim  h_T\norm[T]{\Rcproj{k+2-\eta_T}{T}\bvec{v}}
+h_T\norm[T]{\Rproj{k-2}{T}\bvec{v}}+h_T\left(\sum_{F\in\FT} h_{T}\norm[F]{\bvec{v}_{{\rm t},F}}^2\right)^{\nicefrac12}\eqcolon h_T\mathcal{N}(\bvec{v}).
\]
Plugging this estimate and \eqref{est:bvecv.qbdry} into \eqref{eq:est.reconstruction}, we infer $\norm[T]{q}\lesssim h_T\mathcal{N}(\bvec{v})$ and thus, via a discrete inverse inequality, $\norm[T]{\GRAD q}\lesssim \mathcal{N}(\bvec{v})$.
Since $\bvec{v}=\GRAD q+\bvec{z}$, this bound and \eqref{est:bvecv.cGoly} conclude the proof of \eqref{eq:est.reconstruction.curl}.

%------------------------------------------------------------------------------%

\section*{Acknowledgements}
The authors acknowledge the partial support of \emph{Agence Nationale de la Recherche} grant ANR-20-MRS2-0004 NEMESIS.
Daniele Di Pietro also acknowledges the partial support of I-Site MUSE grant ANR-16-IDEX-0006 RHAMNUS. The authors also thank Liam Yemm for his work on the mesh module in HArDCore3D.

%------------------------------------------------------------------------------%
% Bibliography
%------------------------------------------------------------------------------%

\printbibliography

\end{document}